\documentclass{article}

\usepackage{microtype}
\usepackage{graphicx}
\usepackage{subfigure}
\usepackage{booktabs} 
\usepackage[pagestyles]{titlesec}

\usepackage[utf8]{inputenc} 
\usepackage[T1]{fontenc}    
\usepackage{booktabs}       
\usepackage{amsfonts}       
\usepackage{nicefrac}       
\usepackage{microtype}      
\usepackage{makecell}

\usepackage{amssymb}
\usepackage{amsmath}
\usepackage{comment}
\usepackage{todonotes}
\usepackage[english]{babel}
\usepackage{secdot}
\usepackage[hyphens]{url}
\usepackage{hyperref}
\hypersetup{colorlinks=true,breaklinks=true}
\usepackage{url}
\usepackage{mathtools}
\usepackage{algorithm}
\usepackage{algorithmic}
\renewcommand{\algorithmicrequire}{\textbf{Input:}}
\renewcommand{\algorithmicensure}{\textbf{Output:}}
\newcommand{\eps}{\varepsilon}

\newcommand{\la}{\langle}
\newcommand{\ra}{\rangle}
\def\vp{\varphi}
\def\R{\mathbb{R}}

\usepackage{microtype}
\usepackage{graphicx}
\usepackage{subfigure}
\usepackage{wrapfig}

\usepackage{amsmath,amsfonts,amssymb,amsthm,mathtools}
\usepackage{mathtext} 

\newtheorem{theorem}{Theorem}
\newtheorem{lemma}[theorem]{Lemma}

\renewcommand{\leq}{\leqslant}
\renewcommand{\geq}{\geqslant}

\newcommand{\argmin}{\operatornamewithlimits{argmin}}
\newcommand{\argmax}{\operatornamewithlimits{argmax}}

\newcommand{\bm}[1]{\mathbf{#1}}
\def\one{\mathbf{1}}
\def\vp{\varphi}
\def\R{\mathbb{R}}

\newcommand{\sm}[2]{\begin{smallmatrix}\item1\\\item2 \end{smallmatrix}}

\def\la{\langle}
\def\ra{\rangle}

\newcommand\ddfrac[2]{\frac{\displaystyle \item1}{\displaystyle \item2}}

\def\eps{\varepsilon}
\newcommand{\mc}[1]{\mathbb{\item1}}
\newcommand{\circled}[1]{\raisebox{.5pt}{\textcircled{\raisebox{-.9pt} {#1}}}}
\def\e{\varepsilon}

\def\R{\mathbb{R}}
\def\U{\mathcal{U}}

\DeclareMathOperator{\spn}{span}
\DeclareMathOperator{\diag}{diag}
\DeclareMathOperator{\vectr}{vec}
\renewcommand{\algorithmicrequire}{\textbf{Input:}}
\renewcommand{\algorithmicensure}{\textbf{Output:}}

\def\na#1{{\color{red}#1}} 
\def\sg#1{{\color{blue}#1}} 
\def\todo#1{{\color{red}TODO: #1}} 

\usepackage{hyperref}

\usepackage[accepted]{icml2021}

\icmltitlerunning{Accelerated Alternating Minimization, Accelerated Sinkhorn's Algorithm and Accelerated Iterative Bregman Projections}
\usepackage{footnote}
\makesavenoteenv{tabular}
\makesavenoteenv{table}
\usepackage{placeins}

\begin{document}

\onecolumn
\icmltitle{On a Combination of Alternating Minimization and Nesterov's Momentum}



\icmlsetsymbol{equal}{*}

\begin{center}
\icmlauthor{Sergey Guminov}{mipt,iitp,hse}
\icmlauthor{Pavel Dvurechensky}{wias,iitp,hse}
\icmlauthor{Nazarii Tupitsa}{mipt,iitp,hse}
\icmlauthor{Alexander Gasnikov}{mipt,iitp,hse}
    
\end{center}

\icmlaffiliation{mipt}{Moscow Institute of Physics and Technology, Dolgoprudny, Russia}
\icmlaffiliation{iitp}{Institute for Information Transmission Problems RAS, Moscow, Russia}
\icmlaffiliation{wias}{Weierstrass Institute for Applied Analysis and Stochastics, Berlin, Germany}
\icmlaffiliation{hse}{HDI Lab @ National Research University Higher School of Economics, Russian Federation}

\icmlcorrespondingauthor{Nazarii Tupitsa}{tupitsa@phystech.edu}
\icmlcorrespondingauthor{Pavel Dvurechensky}{pavel.dvurechensky@wias-berlin.de}

\icmlkeywords{convex optimization, non-convex optimization, optimal transport}

\vskip 0.3in



\printAffiliationsAndNotice{}  

\begin{abstract}
Alternating minimization (AM) procedures are practically efficient in many applications for solving convex and non-convex optimization problems. On the other hand, Nesterov's accelerated gradient is theoretically optimal first-order method for convex optimization. In this paper 
we combine AM and Nesterov's acceleration to propose an accelerated alternating minimization algorithm. We prove $1/k^2$ convergence rate in terms of the objective for convex problems and $1/k$ in terms of the squared gradient norm for non-convex problems, where $k$ is the iteration counter. Our method does not require any knowledge of neither convexity of the problem nor function parameters such as Lipschitz constant of the gradient, i.e. it is adaptive to convexity and smoothness and is uniformly optimal for smooth convex and non-convex problems. Further, we develop its primal-dual modification for strongly convex problems with linear constraints and prove the same $1/k^2$ for the primal objective residual and constraints feasibility.
\end{abstract}

\section{Introduction}

Alternating minimization (AM) optimization algorithms have been known for a long time \cite{ortega1970iterative,bertsekas1989parallel}. These algorithms assume that the decision variable is divided into several blocks and minimization in each block can be done explicitly. AM algorithms have a number of applications in machine learning problems. For example, iteratively reweighted least squares can be seen as an AM algorithm. Other applications include robust regression \cite{mccullagh1989generalized} and sparse recovery \cite{daubechies2010iteratively}. The famous Expectation-maximization (EM) algorithm can also be seen as an AM algorithm \cite{mclachlan1996algorithm,andresen2016convergence}.

The initial motivation for this paper was accelerating algorithms for optimal transport (OT) applications, which are widespread in the machine learning community \cite{cuturi2013sinkhorn,cuturi2014fast,arjovsky2017wasserstein}. The ubiquitous Sinkhorn's algorithm can be seen as an alternating minimization algorithm for the dual to the entropy-regularized optimal transport problem. Recent Greenkhorn algorithm \cite{altschuler2017near-linear}, which is a greedy version of Sinkhorn's algorithm, is a greedy modification of an AM algorithm. For the Wasserstein barycenter \cite{agueh2011barycenters} problem, the extension of the Sinkhorn's algorithm is known as the Iterative Bregman Projections (IBP) algorithm \cite{benamou2015iterative}, which can be seen as an alternating minimization procedure \cite{kroshnin19a}. This motivated us to have a wider look on alternating minimization algorithms and try to accelerate general AM algorithm.

Sublinear $1/k$ convergence rate was proved for general AM algorithm for $n=2$ in \cite{beck2015convergence}.
Despite the same convergence rate as for the gradient method, AM-algorithms converge faster in practice as they are free of the choice of the step-size and are adaptive to the local smoothness of the problem. At the same time, there are accelerated gradient methods (AGM) which use a momentum term to have a faster convergence rate of $1/k^2$ \cite{nesterov1983method} and use gradient steps rather than block minimization. Our goal in this paper is to combine the idea of alternating minimization and momentum acceleration to propose an accelerated alternating minimization method. As applications of our general approach, we develop accelerated alternating least squares algorithm and apply it to a non-convex collaborative filtering problem, and propose accelerated Sinkhorn's algorithm for OT distances and accelerated Iterative Bregman Projections algorithm for Wasserstein barycenters.





\textbf{Related work.} Besides mentioned above works on AM algorithms, we mention \cite{beck2013convergence,saha2013nonasymptotic,sun2015improved}, where non-asymptotic convergence rates for AM algorithms for convex problems were proposed and their connection with cyclic coordinate descent was discussed, but the analyzed algorithms are not accelerated. Accelerated versions are known for random coordinate descent methods \cite{nesterov2012efficiency,lee2013efficient,shalev-shwartz2014accelerated,lin2014accelerated,fercoq2015accelerated,allen2016even,nesterov2017efficiency,alacaoglu2017smooth}, cyclic block coordinate descent \cite{beck2013convergence}, greedy coordinate descent \cite{lu2018accelerating}. 
These ACD methods are designed for convex problems and use momentum term, but 
they require knowledge of block-wise Lipschitz constants, i.e. are not parameter-free. 
A hybrid accelerated random block-coordinate method (AAR-BCD) with exact minimization in the last block was proposed in \cite{diakonikolas2018alternating} for convex problems. 
Unlike our greedy choice of the updated block they use random choice and the parameters of the algorithm depend on the block Lipschitz constants, meaning that AAR-BCD algorithm is not parameter-free.
An extension providing a two-block accelerated alternating minimization algorithm is available in the updated version \cite{diakonikolas2018alternatingPr} for the convex case. This method is deterministic and it is explained how to make it parameter-free.
At the same time neither of two algorithms from \cite{diakonikolas2018alternatingPr} have an analysis for non-convex problems or problems with linear constraints, yet it seems that such extensions are possible for their methods. We also underline that our definition of the algorithm parameters, in particular, the sequence $a_k$ in Algorithm \ref{AAM-2}, is different from theirs. 

The summary of the related works on alternating minimization and coordinate methods is presented in the Table 1, where P-F stands for parameter-free, Acc. for accelerated, N-C for non-convex, P-D for primal-dual and B-N for number of blocks.

\renewcommand{\arraystretch}{1}

\begin{table}[ht]

\caption{Summary of the related works}
\label{T:AM}
\centering
\begin{tabular}{lcccccc} 
\toprule
& P-F & Acc. & N-C & P-D & B-N \\ 
\midrule
AM 
\footnote{
\cite{beck2013convergence, beck2015convergence} }& $\surd$ & $\times$ & $\times$ & $\times$ & 2\\ 
AM 
\footnote{
\cite{saha2013nonasymptotic, sun2015improved} }& $\surd$ & $\times$ & $\times$ & $\times$ & any\\ 
ACD \footnote{\cite{nesterov2012efficiency,lee2013efficient,fercoq2015accelerated, shalev-shwartz2014accelerated,allen2016even, nesterov2017efficiency, beck2013convergence, lu2018accelerating,lin2014accelerated,alacaoglu2017smooth}}
&$\times$ &  $\surd$ & $\times$ & $\surd$&any\\

AAR-BCD\footnote{ \cite{diakonikolas2018alternating}} & $\times$ & $\surd$ & $\times$ & $\times$ & $\,\,\,\,\,\,$any$\,\,\,\,$\\ 

AAM\footnote{ \cite{diakonikolas2018alternatingPr}} & $\surd$ & $\surd$ & $\times$ & $\times$ & $\,\,\,\,\,\,$2$\,\,\,\,$\\ 

This paper   & $\surd$ & $\surd$ & $\surd$ & $\surd$ & any\\ 
\bottomrule
\end{tabular}
\end{table}





\begin{table}[ht]

    \centering
   \caption{Summary of OT algorithms}
\label{T:OT}   
\begin{tabular}{cc} 
\toprule
Algorithm & Complexity  \\ 
\midrule
Sinkhorn \cite{cuturi2013sinkhorn,dvurechensky2018computational}& ${N^2\|C\|^2_\infty}/{\varepsilon^2}$ \\
Greenkhorn \cite{altschuler2017near-linear,lin2019efficient}& ${N^2\|C\|^2_\infty}/{\varepsilon^2}$   \\ 
Randkhorn \cite{lin2019efficiency} & ${N^{7/3}\|C\|_\infty^{4/3}}/{\varepsilon}$ \\
APDA(G/M)D \cite{dvurechensky2018computational,lin2019efficient}& ${N^{5/2}\|C\|_\infty}/{\varepsilon}$   \\ 
Mirror-Prox \cite{jambulapati2019direct} & ${N^{2}\|C\|_\infty}/{\varepsilon}$  \\
This paper & ${N^{5/2}\|C\|_\infty}/{\varepsilon}$ \\ 
\bottomrule
\end{tabular}
\end{table}

Concerning the OT problem, the most used algorithm is  Sinkhorn's algorithm \cite{sinkhorn1974diagonal,cuturi2013sinkhorn}. Its complexity for the OT problem was first analyzed in \cite{altschuler2017near-linear} and improved in \cite{dvurechensky2018computational}.
An accelerated gradient descent method in application to OT problem was also analyzed in \cite{dvurechensky2018computational} with a better dependence on $k$ in the rate, but worse dependence on the dimension of the problem, see also \cite{lin2019efficient}. \cite{altschuler2017near-linear} propose a greedy variant called Greenkhorn together with complexity analysis, which was improved in \cite{lin2019efficient}. 
In an unpublished preprint \cite{lin2019efficiency} the authors propose a randomized accelerated version of Sinkhorn's algorithm. We summarize the complexity of existing methods for OT in the Table 2. $N$ is the number of points in the histogram, $C$ is the transportation cost matrix, $\varepsilon$ desired accuracy.
The complexity of approximating Wasserstein barycenter was analyzed in \cite{kroshnin19a}, where the complexity by Iterative Bregman Projections algorithm and a variant of accelerated gradient method was obtained.
Previous works \cite{cuturi2014fast,benamou2015iterative,staib2017parallel,claici2018stochastic} did not give an explicit complexity bounds for approximating barycenter. 
But there are plenty of algorithms for approximating WB including accelerated gradient method plus Sinkhorn's algorithm \cite{cuturi2014fast}, gradient-type methods \cite{cuturi2016smoothed}, accelerated primal-dual gradient descent~\cite{dvurechensky2018decentralize,krawtschenko2020distributed}, stochastic gradient descent~\cite{claici2018stochastic,tiapkin2020stochastic}, distributed and parallel gradient descent~\cite{staib2017parallel,uribe2018distributed,rogozin2021decentralized}, alternating direction method of multipliers (ADMM)\cite{Ye-2017-Fast,Yang-2018-ADMM} and interior-point algorithm~\cite{Ge-2019-Interior}. Only recently the question of complexity got some answers. Namely, two approaches for approximating Wasserstein barycenter based on entropic regularization \cite{cuturi2013sinkhorn} were analyzed. The first approach is based on Iterative Bregman Projection (IBP) algorithm \cite{benamou2015iterative}, which can be considered as a general alternating projections algorithm and also as a generalization of the Sinkhorn's algorithm \cite{sinkhorn1974diagonal}. The second approach Primal-Dual Accelerateg Gradient Descent (PDAGD) is based on constructing a dual problem and solving it by primal-dual accelerated gradient descent. For both approaches, it was shown, how the regularization parameter should be chosen in order to approximate the original, non-regularized barycenter.
In \cite{2020arXiv200204783L} the authors proposed a variant of the Iterative Bregman Projection (IBP) algorithm, which they called FastIBP.
Very recently \cite{dvinskikh2021improved} provided two algorithms to compute Wasserstein barycenter, one of them has the best theoretical convergence guarantees.

We summarize the known complexity bounds from the literature in Table \ref{T:IBP}. We underline that despite many advantages of the entropic regularization, in some situations other reguarizations provide more robust results \cite{blondel2017smooth}. Our proposed method is flexible enough to allow efficient computations with regularizers other than entopic both for OT and WB problems.

\begin{table}[ht]

    \centering
   \caption{Summary of algorithms for Wasserstein barycenters}
   \label{T:IBP}
\begin{tabular}{cc} 
\toprule
Algorithm & Complexity  \\ 
\midrule
IBP \cite{benamou2015iterative,kroshnin19a}& ${mN^{2}\|C\|^2_\infty}/{\varepsilon^2}$ \\
PDAGD \cite{kroshnin19a}& ${mN^{5/2}\|C\|_\infty}/{\varepsilon}$  \\
FastIBP \cite{2020arXiv200204783L} & ${mN^{7/3}\|C\|_\infty^{4/3}}/{\varepsilon^{4/3}}$ \\
Area Convexity \cite{dvinskikh2021improved} & ${mN^{2}\|C\|_\infty/\varepsilon}$ \\
Mirror-Prox \cite{dvinskikh2021improved} & ${mN^{5/2}\|C\|_\infty}/{\varepsilon}$ \\

This paper & ${mN^{5/2}\|C\|_\infty}/{\varepsilon}$ \\ 

\bottomrule
\end{tabular}
\end{table}

\textbf{Our contributions.} For objectives with $n$ blocks of variables we introduce an accelerated alternating minimization method with $O(n/k^2)$ convergence rate for the objective values in smooth unconstrained convex problems and $O(n/k)$ convergence rate in terms of the squared norm of the gradient both for convex and non-convex smooth unconstrained problems.
Thus, in terms of the dependence on the iteration counter $k$ our algorithm achieves uniformly the best possible rates in convex case (same as for AGM) and in non-convex case (same as for gradient descent (GD)). Moreover, the algorithm automatically adapts to convexity and smoothness: it is completely the same for convex and non-convex settings and does not need to know in advance whether the problem is convex or not, i.e. is uniform for smooth convex and non-convex problems; it does not need to know the Lipschitz constant of the gradient, i.e. is parameter-free. Parameter-free versions exist also for AGM and GD (see, e.g. \cite{nesterov2013gradient}), but they are based on a different idea of backtracking line-search and do not explore the block structure of the problem and block minimization for acceleration in practice. 
%

The main idea of our algorithm is to combine block-wise minimization and the extrapolation (also known as momentum) step which is usually used in accelerated gradient methods. We also show that in the convex setting the proposed method is primal-dual, meaning that if we apply it to a dual problem for a linearly constrained strongly convex problem, we can reconstruct the solution of the primal problem with the same convergence rate. In the follow-up work \cite{tupitsa2021alternating} a modification of AAM is proposed and analyzed for strongly convex problems.

To highlight the new properties of our method, the proven convergence rate for non-convex problems and the primal-dual analysis, we consider two particular applications. First, we consider a non-convex collaborative filtering problem and show empirically that our algorithm outperforms the standard alternating least squares algorithm. Second, we apply it to the dual entropy-regularized OT problem to obtain the Accelerated Sinkhorn's algorithm. The Primal-dual analysis is crucial here since the goal is to find the transportation plan, i.e. the primal variable, by solving the dual problem. Our method has complexity comparable to the existing methods and in the experiments, we show that our general method outperforms specific baselines for this problem, including Sinkhorn's algorithm. Importantly, we use a non-standard formulation of the dual entropy-regularized OT problem in the form of minimization of a softmax function. Moreover, \textit{our algorithm is more flexible since it can solve OT problems with other types of regularization, e.g. by squared Euclidean norm.}  
Finally, in the supplementary, we apply our accelerated primal-dual AM algorithm to the Wasserstein Barycenter (WB) problem and propose an accelerated Iterative Bregman Projection algorithm with the complexity $\widetilde{O}\left(mN^{2.5}/\eps\right)$ to find a barycenter of $m$ histograms of dimension $N$. This bound is better than the complexity bound for the standard Iterative Bregman Projection algorithm \cite{kroshnin19a} $\widetilde{O}\left(mN^{2}/\eps^2\right)$ in terms of $\eps$. 
In the follow-up paper \cite{tupitsa2020multimarginal} the AAM method is applied to a more general multimarginal optimal transport problem and complexity estimates are obtained that are better in some regimes than the ones in the literature.

\textbf{Paper organization.} In Sect. \ref{S:AAM} we consider the general setting of minimizing a smooth objective function using block minimization. We introduce our uniform accelerated alternating minimization (AAM) method for convex and non-convex problems together with its primal-dual modification for convex linearly constrained problems. In Sect. \ref{S:primal-dual} we study the primal-dual properties of the method. In Sect.\ref{S:als} we discuss the application of our method to the collaborative filtering problem and provide experiments on the \href{https://www.upf.edu/web/mtg/lastfm360k}{Last.fm dataset 360K} for the collaborative filtering problem. In Sect. \ref{S:ASA} we describe the OT and the WB problems and their entropy-regularized versions, together with the dual for the latters, that are non-standard. Then, we propose the Accelerated versions of Sinkhorn's algorithm and IBP algorithm and obtain their theoretical complexity, and provide the results of numerical experiments on MNIST dataset for both problems and additionally provide experiments for WB problem with Gaussian measures. The proofs of all stated results, the explicit form of algorithms and the application of the proposed methods to the regularized Wasserstein Barycenter problem may be found in the supplement.
In Section~\ref{S:LS} we provide numerical experiment for least squares problem for linear regression.\footnote{Code for all presented algorithms is available at \url{https://github.com/nazya/AAM}}

\section{Accelerated Alternating Minimization}
\label{S:AAM}

In this section we consider the minimization problem 
$\min\limits_{x\in \mathbb{R}^N}f(x),$
where $f(x)$ is continuously differentiable and, in general \textit{non-convex},
 $L$-smooth function, the latter meaning that its gradient is $L$-Lipschitz, i.e. $\forall\ x,\ y\in \mathbb{R}^N\quad \|\nabla f(x)-\nabla f(y)\|_2\leqslant L\|x-y\|_2$.
We assume that the space is equipped with the Euclidean norm $\|\cdot\|_2$ and that the problem has at least one solution, denoted by $x^*$. The set $\{1,\ldots, N\}$ of indices of the basis vectors $\{e_i\}_{i=1}^N$ is divided into $n$ disjoint subsets (blocks) $I_p$, $p\in\{1,\ldots,n\}$. Let $S_p(x)=x+\spn\{e_i:\ i\in I_p\}$, i.e. the affine subspace containing $x$ and all the points differing from $x$ only over the block $p$. We use $x_i$ to denote the components of $x$ corresponding to the block $i$ and $\nabla_i f(x)$ to denote the gradient corresponding to the block $i$. We will further require that for any $p\in\{1,\ldots,n\}$ and any $z\in\mathbb{R}^N$ the problem $\min\limits_{x\in S_p(z)} f(x)$ has a solution, and this solution is easily computable.  



\begin{algorithm}[H]
\caption{Accelerated Alternating Minimization (AAM)}
\label{AAM-2}
\begin{algorithmic}[1]
   \REQUIRE Starting point $x_0$.
    \ENSURE $x^k$
   \STATE Set $A_0=0$, $x^0 = v^0$.
   \FOR{$k \geqslant 0$}
	\STATE Set $\beta_k = \argmin\limits_{\beta\in [0,1]} f\left(x^k + \beta (v^k - x^k)\right)$
	\STATE Set $y^k = x^k + \beta_k (v^k - x^k)\quad $
    \STATE Choose $i_k=\argmax\limits_{i\in\{1,\ldots,n\}} \|\nabla_i f(y^k)\|_2^2$
\STATE Set $x^{k+1}=\argmin\limits_{x\in S_{i_k}(y^k)} f(x)$\quad 
\STATE Find $a_{k+1}$, $A_{k+1} = A_{k} + a_{k+1}$ from \[f(y^k)-\frac{a_{k+1}^2}{2A_{k+1}}\|\nabla f(y^k)\|_2^2=f(x^{k+1})\]\\
\STATE Set $v^{k+1} = v^{k}-a_{k+1}\nabla f(y^k)$ 
\ENDFOR
\end{algorithmic}
\end{algorithm}

Our accelerated alternating minimization method is listed as Algorithm~\ref{AAM-2}. 
This algorithm combines AM and Nesterov’s momentum and, thus, a full-gradient step 8 is inherited and AM updates are used for faster empirical convergence than AGD. In some sense this is similar to AM compared to gradient descent: theoretical rates are the same, but AM has practical benefits. At the same time, full gradient step 8 is not more expensive than other steps. For example, in the OT applications, full gradient costs nearly the same as block minimization.  
We underline that Algorithm~\ref{AAM-2} does not require knowledge of whether the function is convex or non-convex and does not require knowledge of any parameters of the function. The latter is in contrast to standard accelerated gradient descent \cite{nesterov2004introduction}, accelerated random coordinate descent \cite{nesterov2012efficiency,lee2013efficient,shalev-shwartz2014accelerated,lin2014accelerated,fercoq2015accelerated,allen2016even,nesterov2017efficiency}, accelerated cyclic block coordinate descent \cite{beck2013convergence}, accelerated greedy coordinate descent \cite{lu2018accelerating}, all of which require the knowledge of either the constant $L$ or block-wise Lipschitz constants. Our method is also different from parameter-free versions of AGM that use a backtracking line-search as, e.g., in \cite{nesterov2013gradient}. Parameter-free nature of our method is achieved by applying steps 3 and 7. In standard methods $a_k$ is defined by an equation containing $L$ and $\beta_k$ is defined based on $a_k$.
We prove that in the case when $f$ is convex and $L$-smooth, our method has the accelerated $O(n/k^2)$ rate for the objective residual and, for a general setting of possibly non-convex $L$-smooth functions it guarantees that the squared norm of the gradient decreases as $O(n/k)$. Importantly, the obtained convergence rate in the convex case is $n$ times better than the rate for accelerated random coordinate descent \cite{nesterov2012efficiency}, which is $O(n^2/k^2)$.
The main convergence rate theorem for Algorithm~\ref{AAM-2} is as follows.

\begin{theorem}
\label{Th:AAM2-main}
a) Assume that $f$ is (possibly non-convex) $L$-smooth function w.r.t. $\|\cdot\|_2$. Then, after $k$ steps of Algorithm \ref{AAM-2},
$$
\min_{i=0,...,k}\| \nabla f(y^i)\|_2^2 \leqslant \frac{2nL(f(x^0) - f(x^*))}{k}.
$$
b) Assume additionally that $f$ is convex. Then, after $k$ steps of Algorithm \ref{AAM-2},
\begin{align*}
    f(x^{k})-f(x^*) \leqslant \frac{2nL\|x^0-x^*\|_2^2}{k^2}.
\end{align*}
\end{theorem}
\begin{proof}[Proof of Theorem~\ref{Th:AAM2-main}, a)] 
$L$-smoothness of $f$ together with the fact that $x^{k+1}=\argmin_{x\in S_{i_k}(y_k)} f(x)$ where $i_k=\argmax_{i} \|\nabla_i f(y^k)\|_2^2$
implies
\[
    f(y^k)-\frac{1}{2L}\|\nabla_{i_k} f(y^k)\|_2^2\geqslant f(x^{k+1}).
\]
    
Since $i_k=\argmax_{i} \|\nabla_i f(y^k)\|_2^2$ we have that $$\|\nabla_{i_k} f(y^k)\|_2^2\geqslant \frac{1}{n}\|\nabla f(y^k)\|_2^2$$ and

\begin{equation*}
    f(x^{k+1})\leqslant f(y^k)-\frac{1}{2nL}\|\nabla f(y^k)\|_2^2
    \\
    \leqslant f(x^k)-\frac{1}{2nL}\|\nabla f(y^k)\|_2^2.
\end{equation*}
Summing this up for $i=0,\ldots, k$, we obtain 
\begin{equation*}
    f(x^0)-f(x^*)\geqslant f(x^{0})-f(x^{N+1})
    \\
    \geqslant \frac{k}{2nL}\min_{i=0,\ldots, k}\|\nabla f(y^i)\|^2_2.
\end{equation*}
Consequently, we may guarantee $\min\limits_{i=0,\ldots, k}\|\nabla f(y^i)\|_2^2\leqslant \frac{2nL(f(x^0)-f(x^*))}{k}.$
\end{proof}

To prove the part b) of Theorem~\ref{Th:AAM2-main} we firstly state an auxiliary lemma.
Let us introduce an auxiliary  sequence of functions defined as $\psi_0(x)=\frac{1}{2}\|x-x^0\|^2,\    \psi_{k+1}(x) = \psi_{k}(x) + a_{k+1}\{f(y^k) + \langle \nabla f(y^k), x - y^k \rangle\}.
$
It is easy to see that $v^{k}=\argmin\limits_{x\in\mathbb{R}^N} \psi_{k}(x).$

\begin{lemma} 
\label{AAM-2_Ak_rate}
After $k$ steps of Algorithm \ref{AAM-2} it holds that
\begin{equation}
    \label{eq:main_recurrence}
    A_{k}f(x^{k}) \leqslant \min_{x \in \mathbb{R}^N} \psi_{k}(x) = \psi_{k}(v^{k}).
\end{equation}
Moreover, 
$A_k \geqslant\frac{k^2}{4Ln}$, where $n$ is the number of blocks.
\end{lemma}

\begin{proof}[Proof of Theorem~\ref{Th:AAM2-main} b).]
From the convexity of $f(x)$ we have
$
\frac{1}{A_k}\sum_{i=0}^{k-1} a_{k+1}(f(y^k)+\langle\nabla f(y^k),x-y^k\rangle)\leqslant f(x^*).
$ From Lemma~\ref{AAM-2_Ak_rate}, using the standard argument \cite{nesterov2005smooth}, we have 
\begin{multline*}
    A_k f(x^k)\leqslant\psi_{k}(v^k)\leqslant\psi_k(x^*)
    =\frac{1}{2}\|x^*-x^0\|_2^2 +\sum_{i=0}^{k-1} a_{i+1}(f(y^i)+\langle\nabla f(y^i),x^*-y^i\rangle)
    \\
    \leqslant A_kf(x^*)+\frac{1}{2}\|x^*-x^0\|_2^2.
\end{multline*}
Since $A_k\geqslant \frac{k^2}{4nL},$ we finally obtain the statement of the theorem $f(x^{k})-f(x^*) \leqslant \frac{2nL\|x^*-x^0\|_2^2}{k^2}.$
\end{proof}
The obtained rate leads to complexity $O(\sqrt{n/\eps})$ to achieve accuracy $\eps$ in terms of the objective. As we show below, for the collaborative filtering problem and optimal transport problem $n=2$ and our accelerated method provides acceleration from complexity $O(1/\eps)$ of existing AM methods to the better complexity $O(1/\sqrt{\eps})$.

\section{Primal-Dual Extension}
\label{S:primal-dual}
In this section we consider the primal-dual (up to a sign) pair of minimization problems 
\begin{align*}
(P_1) &  \min_{x\in Q \subseteq E} \left\{ f(x) : \bm{A}x =b \right\},
\\
(P_2) & \min_{\lambda \in \Lambda} \left\{   \phi(\lambda)=\la \lambda, b \ra  + \max_{x\in Q} \left( -f(x) - \la \bm{A}^T \lambda  ,x \ra \right) \right\},
\end{align*}

\begin{algorithm}[ht]
\caption{Primal-Dual AAM}
\label{PDAAM-2}
\begin{algorithmic}[1]
   \STATE $A_0=a_0=0$, $\eta_0=\zeta_0=\lambda_0=0$.
   \FOR{$k \geqslant 0$}
   \STATE Set $\beta_k = \argmin\limits_{\beta\in [0,1]} \phi\left(\eta^k + \beta (\zeta^k - \eta^k)\right)$
\STATE Set $\lambda^{k}=\beta_k \zeta^k+(1-\beta_k)\eta^k$
\STATE Choose $i_k=\argmax\limits_{i\in\{1,\ldots,n\}} \|\nabla_i \phi(\lambda^k)\|_2^2$
\STATE Set $\eta^{k+1}=\argmin\limits_{\eta\in S_{i_k}(\lambda^k)} \phi(\eta)$
\STATE Find $a_{k+1}$, $A_{k+1} = A_{k} + a_{k+1}$  from  \[\phi(\lambda^k)-\frac{a_{k+1}^2}{2(A_k+a_{k+1})}\|\nabla \phi(\lambda^k)\|_2^2=\phi(\eta^{k+1})\]\\
\STATE Set $\zeta^{k+1} = \zeta^{k} - a_{k+1}\nabla \phi(\lambda^k)$
\STATE Set $\hat{x}^{k+1} = \frac{a_{k+1}x(\lambda^{k})+A_k\hat{x}^{k}}{A_{k+1}}.$

\ENDFOR
\ENSURE The points $\hat{x}^{k+1}$, $\eta^{k+1}$.
\end{algorithmic}
\end{algorithm}

where $E$ is a finite-dimensional real vector space, $Q$ is a simple closed convex set, $f$ is a $\gamma$-strongly convex function, $\bm{A}$ is a given linear operator from $E$ to some finite-dimensional real vector space $H$, $b \in H$ is given, $\Lambda=H^*$ is the conjugate space.

Since $f$ is convex, $\phi(\lambda)$ is a convex function and, by Danskin's theorem, its subgradient is equal to
\begin{equation}
\nabla \phi(\lambda) = b - \bm{A} x (\lambda),
\label{eq:nvp}
\end{equation}
where $x (\lambda)$ is some solution of the convex problem
\begin{equation}
    \max_{x\in Q} \left( -f(x) - \la \bm{A}^T \lambda  ,x \ra \right).
    \label{eq:inner}
\end{equation}

In what follows, we assume that $H$ is equipped with the Euclidean norm, $\phi(\lambda)$ is $L$-smooth and that the problem $(P_2)$ has a solution $\lambda^*$ and there exist some $R>0$ such that $\|\lambda^{*}\|_{2} \leqslant R$. We underline that the quantity $R$ will be used only in the convergence analysis, but not in the algorithm itself.
Our primal-dual algorithm based on Algorithm~\ref{AAM-2} for the pair $(P_1)$-$(P_2)$ is listed as Algorithm \ref{PDAAM-2}.

The key result for this method is that it guarantees convergence in terms of the constraints and the duality gap for the primal problem, provided that the primal objective is strongly convex. The rate of convergence and complexity remain the same as for Algorithm~\ref{AAM-2}.
\begin{theorem}
\label{PD-bounds}
Let the objective $f(x)$ in the problem $(P_1)$ be $\gamma$-strongly convex w.r.t. $\|\cdot\|_E$, and let $\|\lambda^*\| \leqslant R$. Then, for the sequences $\hat{x}^{k},\eta^{k}$, $k\geqslant 0$, generated by Algorithm \ref{PDAAM-2}, 
\begin{align}
    &|\phi(\eta^k) + f(\hat{x}^k)| \leqslant\frac{8n\|\bm{A}\|^2_{E\to H}R^2}{\gamma k^2}, 
    \\
    &\|\bm{A} \hat{x}^k - b \|_2 \leqslant \frac{8n\|\bm{A}\|^2_{E\to H}R}{\gamma k^2},
    \\
    &\|\hat{x}^k-x^*\|_E \leqslant \frac{4n\|\bm{A}\|_{E\to H}R}{\gamma k}
    \label{eq:APDAAM_bound}
\end{align}
where $\|\bm{A}\|_{E\to H}$ is the norm of $\bm{A}$ as a linear operator from $E$ to $H$, i.e. $\|\bm{A}\|_{E \rightarrow H} = \max_{u,v} \left\{{\la Au, v\ra : \|u\|_E = 1, \|v\|_H = 1}\right\}$, and $\|\cdot\|_H = \|\cdot\|_2$.

\end{theorem}

\section{Application to Non-convex Optimization} \label{S:als}
\begin{figure}[!ht]
	\centering
	\includegraphics[width=0.8\columnwidth]{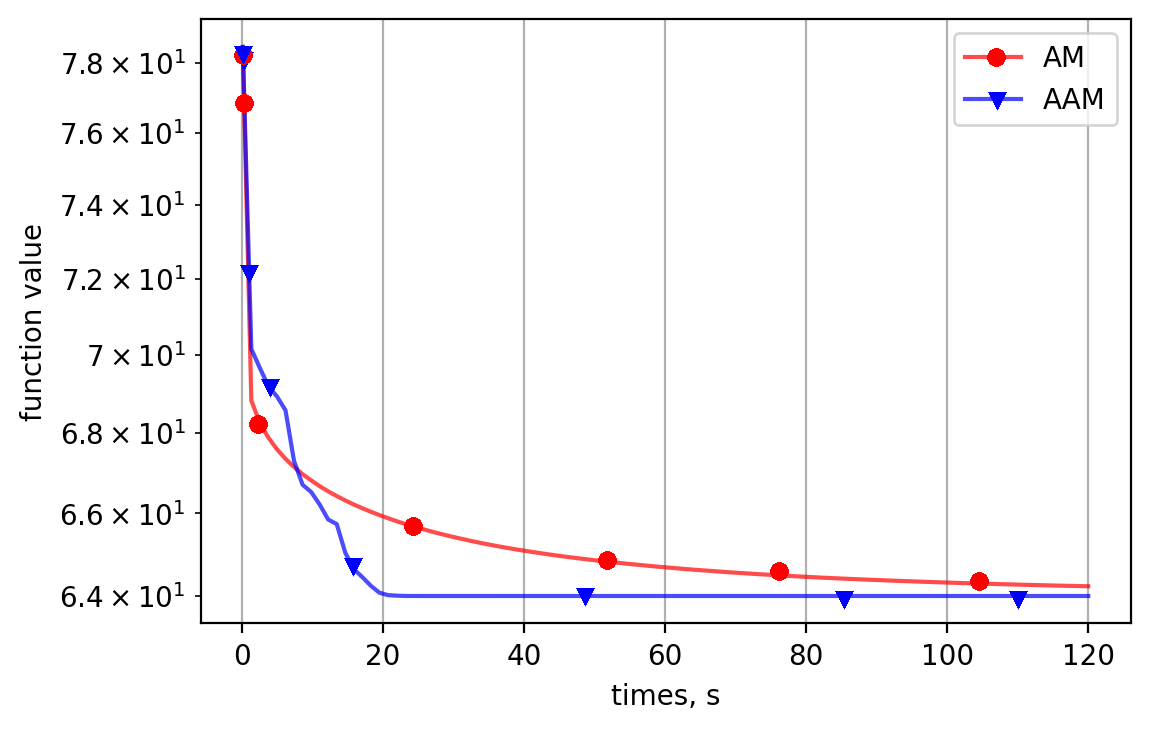}
	\vspace{-0.5cm}
	\caption{Performance of  AM and Algorithm~\ref{AAM-2} on the problem \eqref{eq:mcomp_problem}}
	\label{fig:als}
\end{figure}


In this section we apply our general accelerated AM method to a non-convex collaborative filtering problem. The problem consists of completion of the user-item preferences matrix with estimated values based on a small number of observed ratings made by other users. This is a particular case of the matrix completion problem. The unknown ratings $\hat r_{ui}$ associated with the user $u$ and the item $i$ are sought as a product $x_u^\top y_i$, where the vectors $x_u$ and $y_i$ are the optimized variables. 
We assume that we are given $r_{ui}$ -- observed preference rates associated with some users and items. The confidence $c_{ui}$ for an observation $r_{ui}$ is defined as $c_{ui} = 1 + 5 r_{ui}$, and the binarized rating  $p_{ui}$ is defined as $p_{ui} = 1$ if $r_{ui} > 0$ and $p_{ui} = 0$ if $r_{ui} = 0$.
Following the approach in \cite{hu2008collaborative}, we minimize the data fitting term with a regularizer
\begin{equation}
    \label{eq:mcomp_problem}
    \min\limits_{x, y}~~F(x, y) = \sum_{\text{observed}\,u, i} c_{ui} \left(r_{ui} - x_u^\top y_i\right)^2
    \\
    +  \lambda \left(\sum_{u} ||x_u||_2^2 + \sum_{i} ||y_i||_2^2\right).
\end{equation}

This function can be explicitly minimized over $x$ for fixed $y$ and vice-versa, which motivates the use of alternating minimization procedures.

The considered objective function is not convex, but has Lipchitz continuous gradient (by Theorem 1 from \cite{khenissi2019modeling}), so the minimization via Algorithm~\ref{AAM-2} is possible. We use the standard AM algorithm as a baseline. We generate the matrix $\{r_{ui}\}_{u, i}$ from  \href{https://www.upf.edu/web/mtg/lastfm360k}{Last.fm dataset 360K} with ratings given by listeners to certain artists. There were 70  users and 100 artists observed, and the sparsity coefficient of the matrix was approximately $2\%$. The regularization coefficient was set to $\lambda = 0.1$ 
In Figure~\ref{fig:als} we compare the performance of AM and Algorithm~\ref{AAM-2} applied to the problem \eqref{eq:mcomp_problem}.

\section{Application to Optimal Transport and Wasserstein Barycenter}
\label{S:ASA}
In this section we apply the developed methods to solve the  discrete-discrete optimal transportation problem
\begin{align}
    \label{OT}
    \min_{X\in \mathcal{U}(r,c)} f(X)=\langle C, X\rangle
    \\
    \mathcal{U}(r,c)=\{X\in \mathbb{R}^{N\times N}_+: X\mathbf{1}=r, X^T\mathbf{1}=c\} \notag,
\end{align}
where $X$ is the transportation plan, $C\in\mathbb{R}^{N\times N}_+$ is a given cost matrix, $\mathbf{1} \in \mathbb{R}^N$ is the vector of all ones,  $r,c\in S_N(1):=\{s \in \mathbb{R}^N_+: \la s, \mathbf{1}\ra = 1\}$ are given discrete measures, and $\la A,B\ra$ denotes the Frobenius product of matrices defined as $\la A,B\ra=\sum\limits_{i,j=1}^N A_{ij}B_{ij}$. 

Optimal transport distances lead to the concept of Wasserstein barycenter (WB). Given two probability measures $p, q$ and a cost matrix $C \in \mathbb{R}_+^{N \times N}$ we define optimal
transportation distance between them as
\[W_{C}(p, q) = \min_{X\in\mathcal{U}(p,q)} \langle X, C\rangle.\]

For a given set of probability measures $p_i$ and cost matrices $C_i$ we define their weighted barycenter with weights $w \in S_m(1)$ as a solution of the following convex optimization problem:\[\min_{q \in S_N(1)} \sum_{i=1}^m w_iW_{C_i}(p_i
, q).\]

The key aspect to apply our method is the strong convexity of the function to minimize. To ensure this, we introduce a \textit{general} strongly convex regularizer $\mathcal{R}(X)$, e.g. entropy \cite{cuturi2013sinkhorn} or squared Euclidean norm \cite{essid2018quadratically}. Since the $f(X)$ is strongly convex, we are in the situation of Section \ref{S:primal-dual}. We underline that our method is able to solve OT problems with \textit{general} regularizers, but, next we focus on a special case of entropic regularization as the most used in practice. In this case $\mathcal{R}(X) = \la X, \ln X \ra$ with  $\ln X$ taken elementwise. The detailed derivations and proofs for this subsection can be found in the supplementary.

 Using the entropic regularization we define  the regularized OT-distance for $\gamma>0$: 
 \[
    W_{C,\gamma}(p, q) = \min_{\pi\in\mathcal{U}(p,q)} \langle \pi, C\rangle + \gamma \mathcal{R}(\pi),
 \]
and the regularized barycenter which is the solution to the following problem:\begin{equation}
\label{entropy_bar}
\min _{q \in \in S_N(1)} \sum_{l=1}^{m} w_{l}\mathcal{W}_{C_l,\gamma}\left(p_{l}, q\right).
\end{equation}

Importantly, the entropy $\la X, \ln X \ra$ is \textit{not} strongly convex on $\mathbb{R}^{N\times N}_+$. Thus, if we just take $Q=\mathbb{R}^{N\times N}_+$ in Section \ref{S:primal-dual}, we will get a standard dual problem \cite{altschuler2017near-linear}[Sect. 3.3] in the form of minimization of a sum of exponents. This objective \textit{does not} have Lipschitz-continuous gradient as the gradient grows exponentially. Previous works \cite{dvurechensky2018computational,lin2019efficient,lin2019efficiency} do not take this into account and apply accelerated gradient methods to the dual problem, which makes their complexity results not completely correct. 

To resolve this problem, we note that $\mathcal{U}(r,c) \subset Q:=\{X\in \mathbb{R}^{N\times N}_+: \mathbf{1}^TX\mathbf{1}=1\}$ and the entropy $\la X, \ln X \ra$ is strongly convex on this new set $Q$ in $1$-norm. Thus, we introduce an additional constraint $\mathbf{1}^TX\mathbf{1}=1$ into the problem. Since this constraint is a corollary of the constraint $X \in \mathcal{U}(r,c)$, the solution of the problem remains the same. The gain is that the gradient in the dual now becomes Lipschitz continuous and we can apply our primal-dual AAM.

Introducing the dual variables $y,z$, we derive in the supplementary the dual entropy OT problem
\small
\begin{equation}
    \label{OT_dual}
    \min_{y,z\in\mathbb{R}^N}  
    \gamma\ln\left(
    \sum_{i,j=1}^N \exp \left(\frac{- ({y^i+z^j+C^{i j}})}{\gamma} \right)
    \right)+\la y,r \ra + \la z,c \ra,
\end{equation}
\normalsize
 and the dual (minimization) problem of \eqref{entropy_bar}
 
\begin{equation}
    \label{WB_dual}
       \min_{
       \substack{
        u,v\\
        \sum_{l=1}^m w_lv_l=0
       }
       }
       \gamma\sum\limits_{l=1}^{m} w_{l} \ln \sum_{i,j=1}^N \exp \frac{- ({u_l^i+v_l^j+C_l^{i j}})}{\gamma} 
       \\
       -\gamma\sum\limits_{l=1}^{m} w_{l} \left\langle u_{l}, p_{l}\right\rangle
\end{equation}
The variables in the dual problem \eqref{OT_dual}, \eqref{WB_dual} naturally decompose into two blocks. Moreover, minimization over any one block can be made explicitly and the expressions are the same as for the Sinkhorn's algorithm in the form of \cite{altschuler2017near-linear} and IBP from \cite{kroshnin19a}. The detailed proof of this fact may be found in the corresponding section of the supplement.

\begin{figure}[!ht]
\centering
\includegraphics[width=0.8\columnwidth]{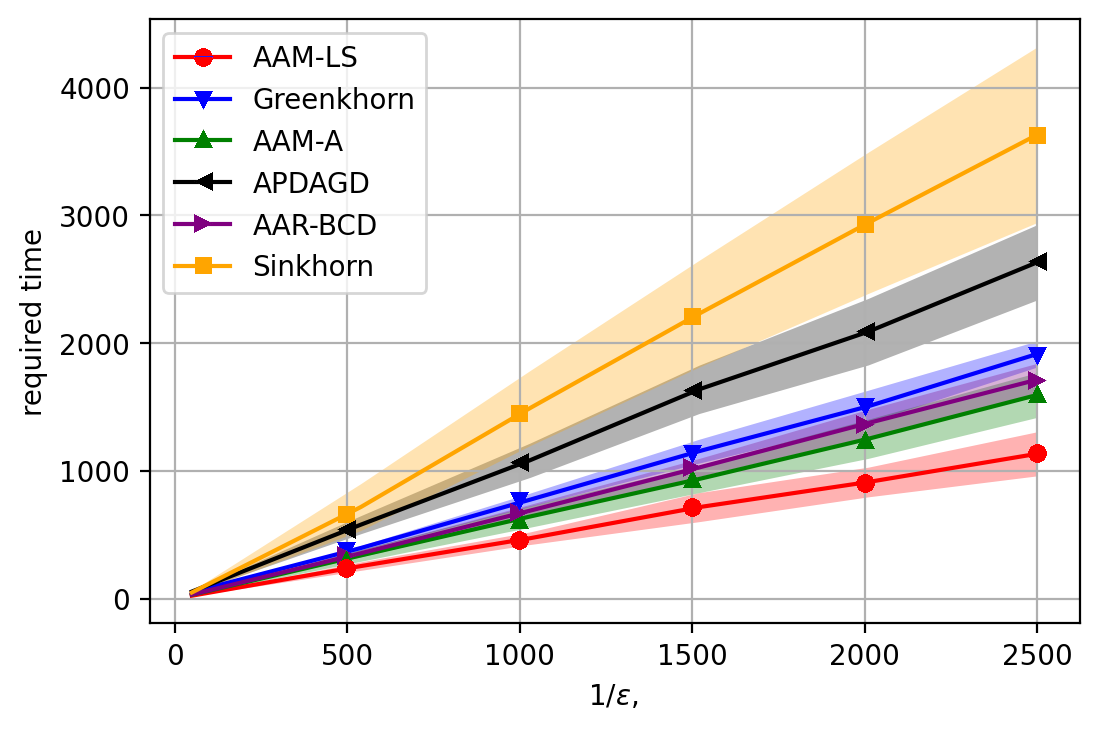}
\caption{Performance comparison on MNIST dataset. Filled in area corresponds to 1 standard deviation.}
\label{icml-historical}
\end{figure}

Concerning \textbf{OT problem}, the goal is to approximate the non-regularized OT distance, the regularization parameter has to be chosen small, which leads to instabilities for the matrix-scaling Sinkhorn's algorithm of \cite{cuturi2013sinkhorn}.

We obtained the final bound of the complexity to find an $\eps$-approximation for the non-regularized OT problem to be $O\left(\frac{N^{5/2}\sqrt{\ln N}\|C\|_\infty}{\varepsilon}\right)$. Compared to the same bound for the Sinkhorn's algorithm, which is $O\left(\frac{N^2\ln N\|C\|^2_\infty}{\varepsilon^2}\right)$, the new result for our accelerated algorithm is better in terms of $\eps$.  Detailed derivations can be found in the supplementary.

In Figure~\ref{icml-historical}, we provide a numerical comparison of our methods with Sinkhorn's algorithm, the AAR-BCD method \cite{diakonikolas2018alternating}, the APDA(G/M)D method \cite{dvurechensky2018computational,lin2019efficient} and with the Greenkhorn algorithm \cite{altschuler2017near-linear}. We do not provide numerical comparison with Area Convexity algorithm from \cite{jambulapati2019direct} because the authors \textbf{did not} implement their algorithm. Instead of this the authors "implemented their algorithm as an instance
of mirror prox". For this instance "there is not a known proof of  convergence with an area-convex regularizer". So it's impossible to know the moment of time when the desired accuracy is reached. The AAM-LS method is the Accelerated Sinkhorn algorithm based on Algorithm \ref{PDAAM-2}, while the AAM-A is the Accelerated Sinkhorn algorithm based on the APDAGD method.  Pseudocode of both these methods may be found in the supplementary. 
We performed experiments using randomly chosen images from MNIST dataset. We slightly modified the smaller values in the measures corresponding to the images as in \cite{dvurechensky2018computational}. We choose several values of accuracy $\eps \in [0.0004, 0.002]$, sampled 5 pairs of images and ran the methods until the desired accuracy was reached, which is ensured using computable stopping criteria \cite{dvurechensky2018computational}. Our AAM algorithms  outperform the other methods and also have much lower variance in performance compared to the Sinkhorn's algorithm. Probably the large variance in the results for Sinkhorn's algorithm is caused by its instability for small $\gamma$, which corresponds to small $\eps$.

\FloatBarrier

For \textbf{WB problem}, we add to the comparison recently presented algorithm from \cite{dvinskikh2021improved}. All presented algorithms have convergence guarantees on the value of non-regularized primal function, e.g. they guarantee that $\sum_{l = 1}^m w_l  W(p_l, \bar{q}^t) - \sum_{l = 1}^m w_l W(p_l, q^*) \leq \varepsilon$ after $t$ number of iterations (see Table~\ref{T:IBP}), where $\bar{q}^t \:= \sum_{l = 1}^m w_l q^t_l$ and $q^t_l = (X^t)^T \one$, $X^t$ is an approximation of a tansportation plan at iteration $t$. But the particular implementation of Area Convexity algorithm from  \cite{dvinskikh2021improved} is supposed to work faster than theoretical analysis allows, because alternating minimization procedure for calculation of a prox-mapping has different stopping criterion, which is more easy to satisfy.  To compare actual convergence, we took on 5 randomly chosen images from MNIST dataset and plotted in Figure~\ref{pfme} and Figure~\ref{pfms} the rate of decay of primal function from a transportation plan, which is projected on the feasible set with  Algorithm~2 from \cite{altschuler2017near-linear}. 
We divided visualisation into two figures because of the scaling issues: Area-Convexity and Mirror-Prox were much slower than the others. IBP appears twice for a reference.
Parameter of entropic regularization $\gamma=5e-4$.

Figure~\ref{ab-ibp} and Figure~\ref{ab-oth} illustrate the results obtained after 500s by the proposed algorithms.

\begin{figure}[H]
\centering
\includegraphics[width=0.5\columnwidth]{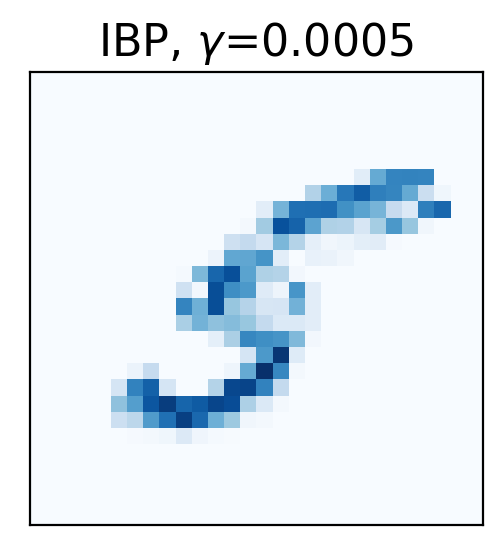}
\caption{Approximate barycenter}
\label{ab-ibp}
\end{figure}

\begin{figure}[H]
\centering
\includegraphics[width=\columnwidth]{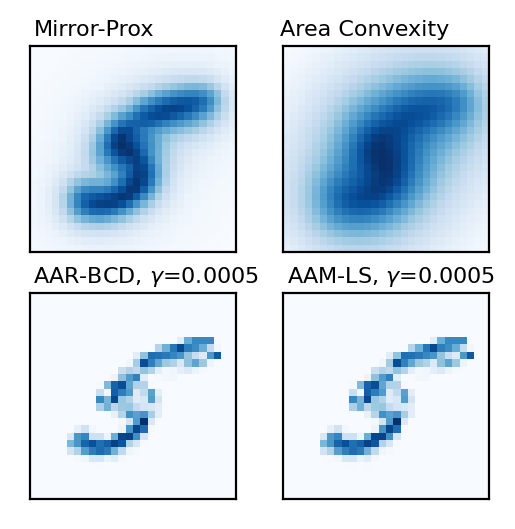}
\caption{Approximate barycenter}
\label{ab-oth}
\end{figure}

\begin{figure}[H]
\centering
\includegraphics[width=0.8\columnwidth]{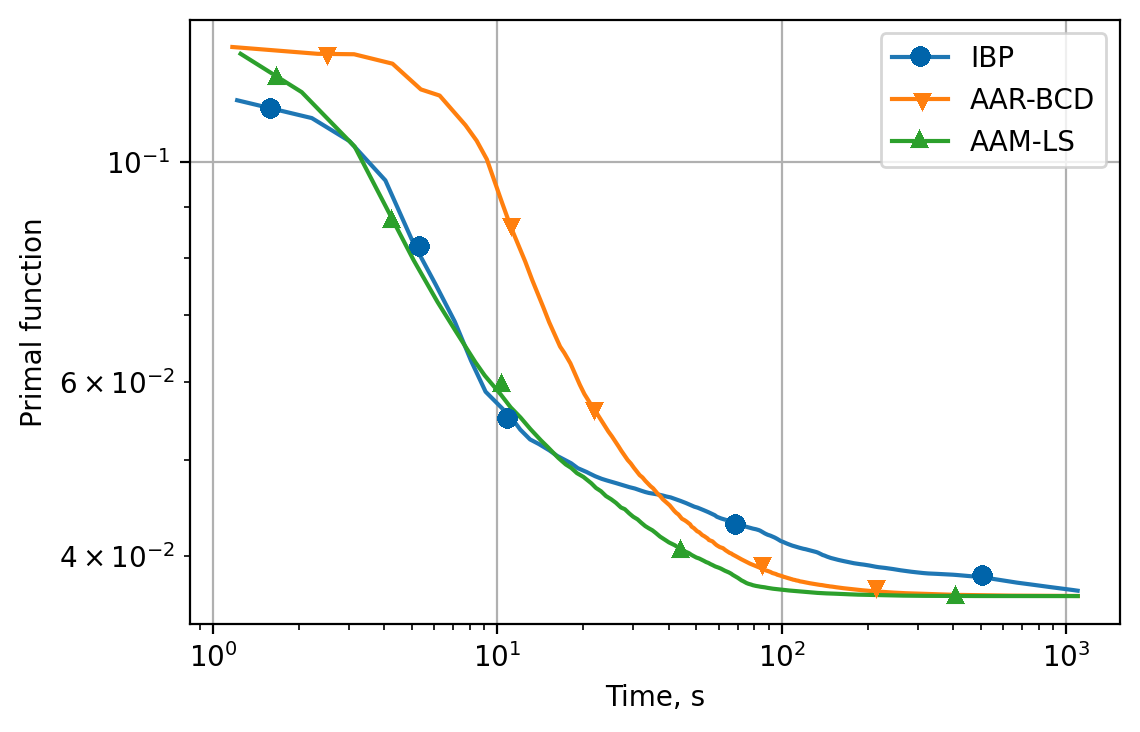}
\caption{Performance comparison on MNIST dataset.}
\label{pfme}
\end{figure}

\begin{figure}[H]
\centering
\includegraphics[width=0.8\columnwidth]{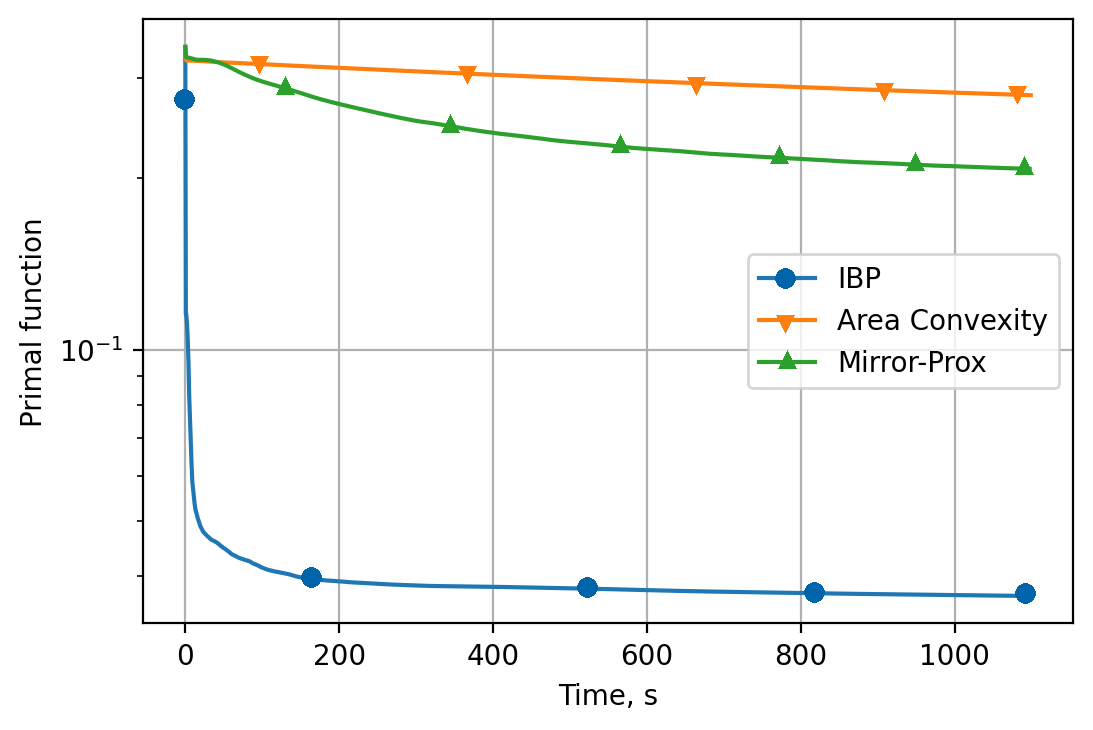}
\caption{Performance comparison on MNIST dataset.}
\label{pfms}
\end{figure}

We also compare the performance of algorithms in terms of $\sum_{l = 1}^m w_l \|q_l^t - \bar{q}^t\|_1$ which is used as stopping criterion for IBP algorithm, in Figure~\ref{oam}.

\begin{figure}[H]
\centering
\includegraphics[width=0.8\columnwidth]{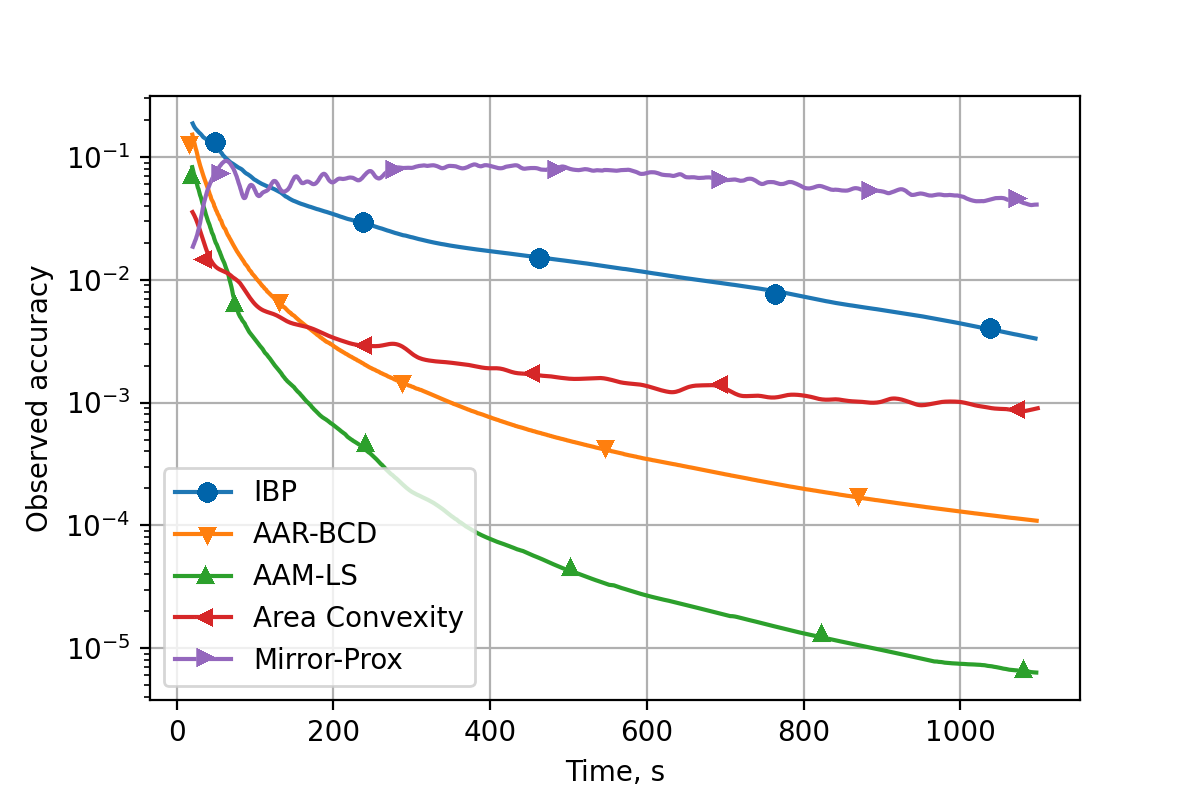}
\caption{Performance comparison on MNIST dataset.}
\label{oam}
\vspace{-0.2cm}
\end{figure}



One may be interested in convergence to a true barycenter. To show the convergence we conducted experiments with random Gaussian measures. For this setup one has analytic expression for a Wasserstein barycenter. 

In Figure~\ref{rag} we compare the performance of algorithms in terms of $\|\bar{q}^t - {q}^*\|_1$,  where ${q}^*$ is a true barycenter. Parameter of entropic regularization $\gamma=5e-5$.

\begin{figure}[H]
\centering
\includegraphics[width=0.83\columnwidth]{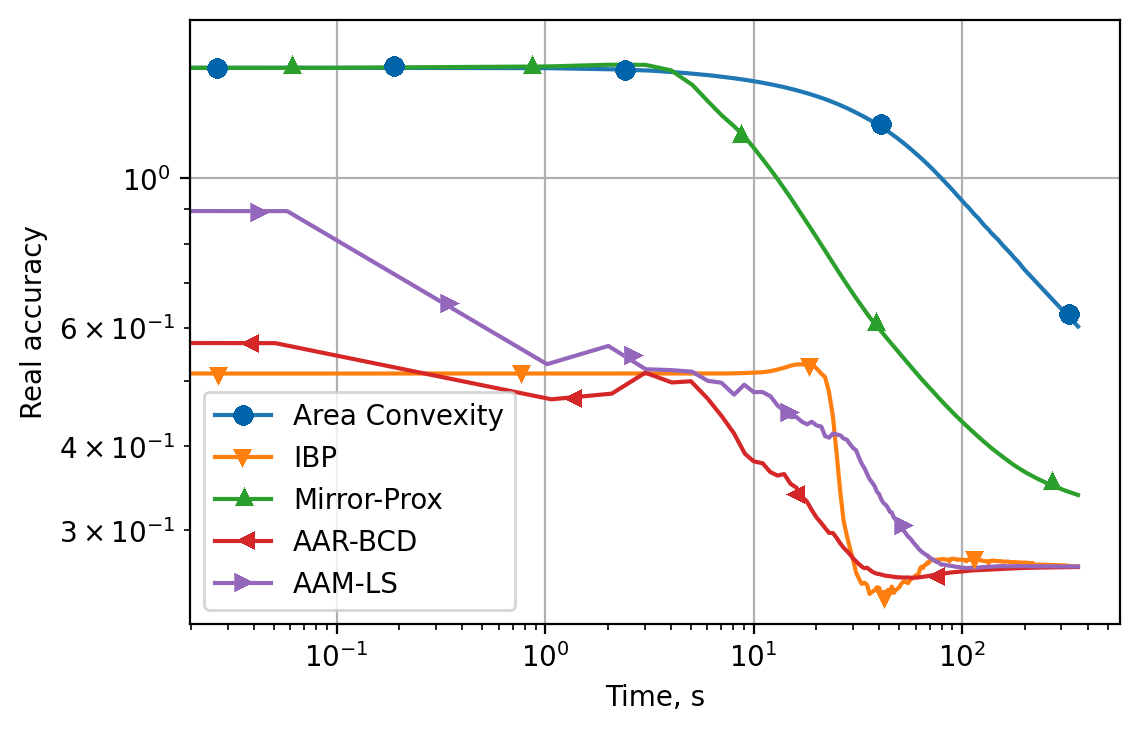}
\caption{Performance comparison on Gaussian measures.}
\label{rag}
\end{figure}


In Figure~\ref{oag} we compare the performance of algorithms in terms of $\sum_{l = 1}^m w_l \|q_l^t - \bar{q}^t\|_1$ in order to show a relation between Real accuracy and Observed accuracy.
\begin{figure}[H]
\centering
\includegraphics[width=0.8\columnwidth]{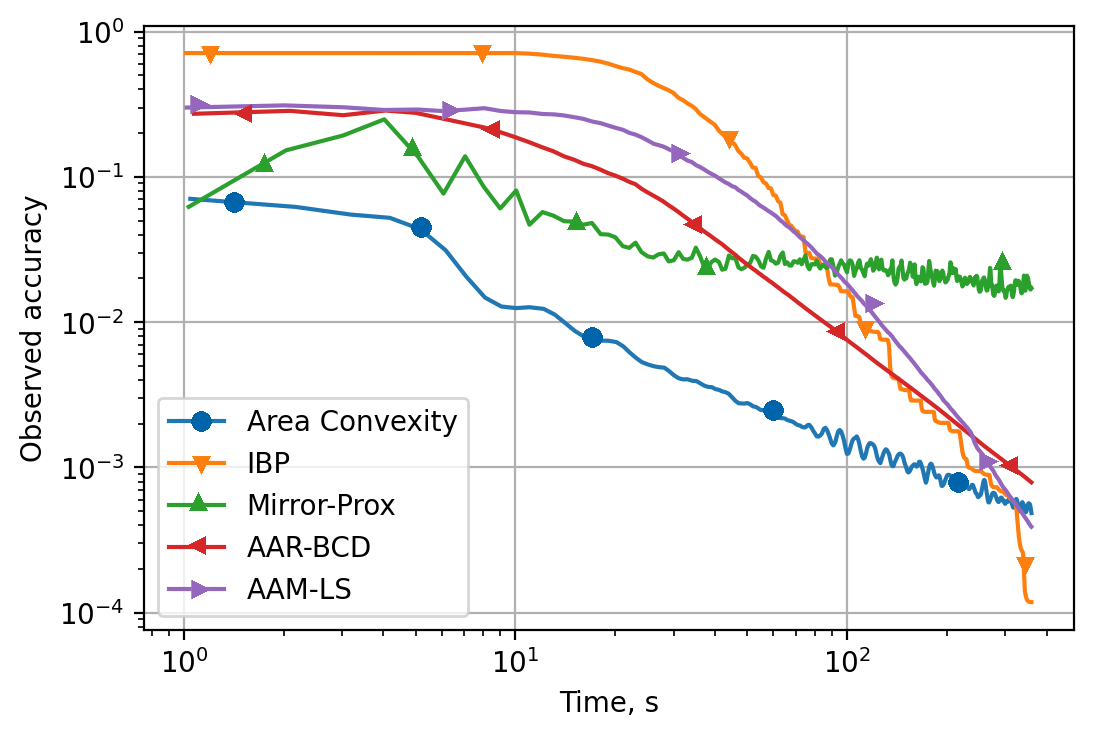}
\caption{Performance comparison on Gaussian measures.}
\label{oag}
\end{figure}

\section{Application to Least Squares}
\label{S:LS}
We also illustrate the results by solving the alternating least squares problem on the Blog Feedback Data Set \cite{buza2014feedback} obtained from \href{http://archive.ics.uci.edu/ml}{UCI Machine Learning Repository}. The data set contains 280 attributes and 52,396 data points. The attributes correspond to various metrics of crawled blog posts. The data is labeled, and the labels correspond to the number of comments that were posted within 24 hours from a fixed basetime. The goal of a regression method is to predict the number of comments that a blog post receives. 

We partition the data into $n$ blocks of the same size sequentially, e.g. we group the first $N/n$ coordinates into the first block, the second $N/n$
coordinates into the second block, and so on. We present comparison with 
block sizes $N/n$ are 5 and 20, corresponding to $n = 56$ and $n = 14$.

The comparison for the linear regression is presented in Figure~\ref{linreg-1} and in Figure~\ref{linreg-2}.

\begin{figure}[H]
	\centering
	\includegraphics[width=0.8\columnwidth]{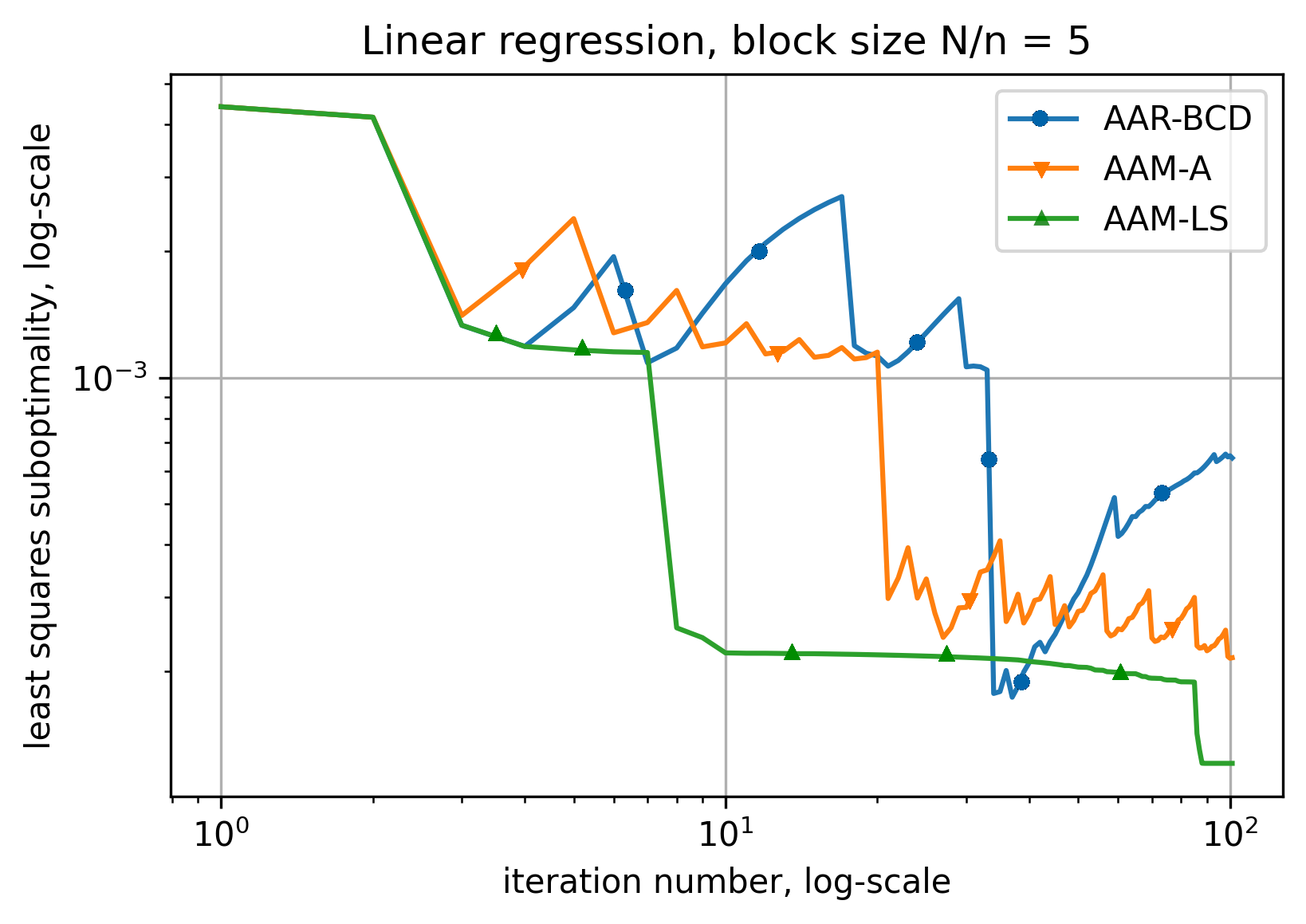}
	\caption{Performance comparison for the linear regression}
	\label{linreg-1}
	
	\includegraphics[width=0.8\columnwidth]{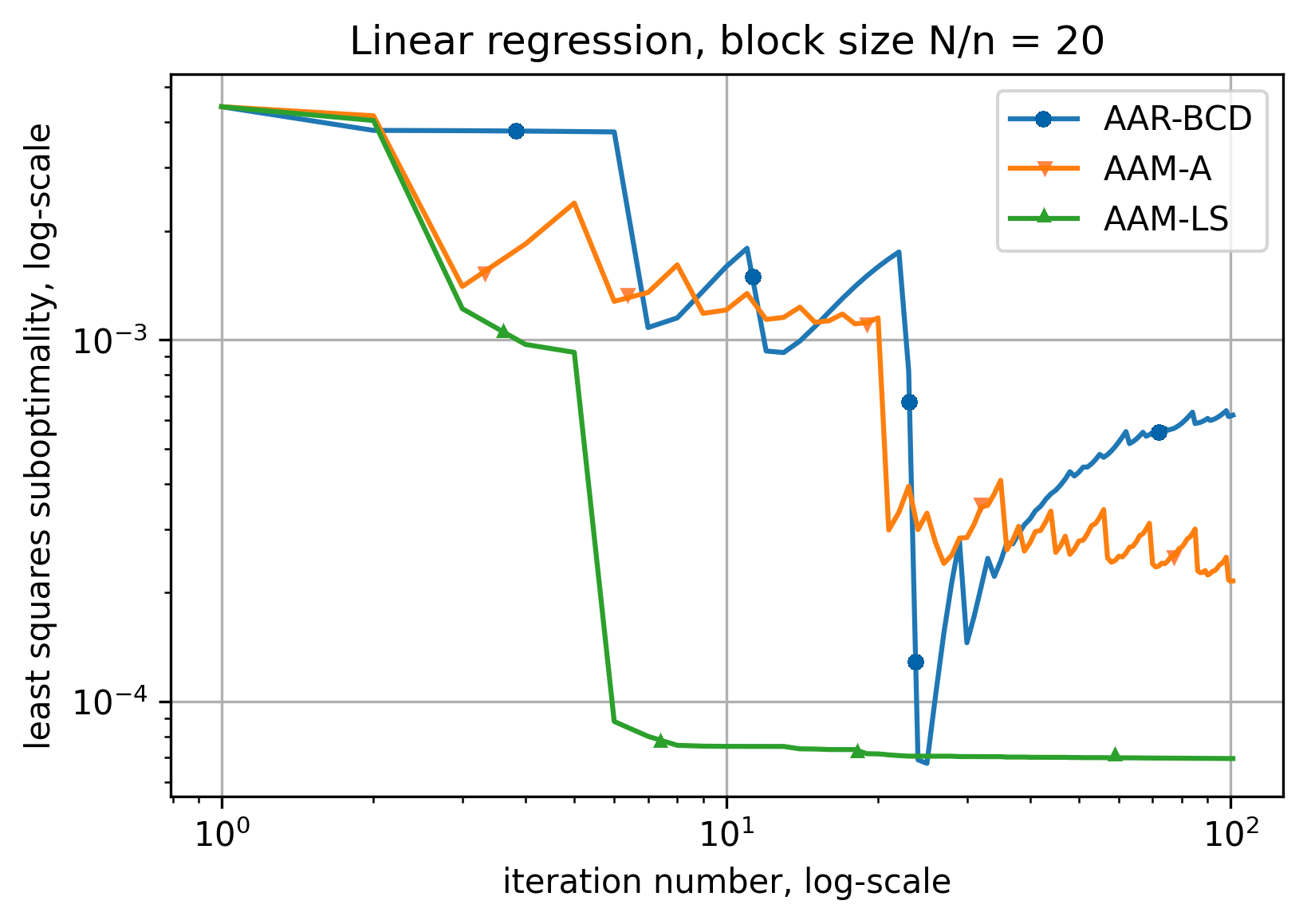}
	\caption{Performance comparison for the linear regression}
	\label{linreg-2}
\end{figure}

\section{Conclusions}
In this paper we propose an accelerated alternating minimization algorithm that combines greedy block-wise updates with full relaxation and Nesterov's moment. The method automatically adapts to the gradient Lipschitz constant and convexity of the problem. It achieves in the convex case  $O(n/k^2)$ convergence rate for the objective and in the non-convex case $O(n/k)$ convergence rate for the squared norm of the gradient. We also propose a primal-dual extension of this algorithm for minimizing strongly convex functions under linear constraints. The practical efficiency of the algorithm is demonstrated by a series of numerical experiments.

\section*{Acknowledgements}
We are grateful to the anonymous referees for their helpful comments and suggestions. We are also grateful to Jelena Diakonikolas for discussions related to this work.
This research was funded by Russian Science Foundation (project 18-71-10108).


\nocite{nesterov2020primal-dual,guminov2019accelerated,chernov2016fast,dvurechensky2016primal-dual,anikin2017dual,dvinskikh2019primal,dvurechensky2020stable}

\bibliography{references}

\begin{thebibliography}{65}
\providecommand{\natexlab}[1]{#1}
\providecommand{\url}[1]{\texttt{#1}}
\expandafter\ifx\csname urlstyle\endcsname\relax
  \providecommand{\doi}[1]{doi: #1}\else
  \providecommand{\doi}{doi: \begingroup \urlstyle{rm}\Url}\fi

\bibitem[Agueh \& Carlier(2011)Agueh and Carlier]{agueh2011barycenters}
Agueh, M. and Carlier, G.
\newblock Barycenters in the wasserstein space.
\newblock \emph{SIAM Journal on Mathematical Analysis}, 43\penalty0
  (2):\penalty0 904--924, 2011.

\bibitem[Alacaoglu et~al.(2017)Alacaoglu, Tran~Dinh, Fercoq, and
  Cevher]{alacaoglu2017smooth}
Alacaoglu, A., Tran~Dinh, Q., Fercoq, O., and Cevher, V.
\newblock Smooth primal-dual coordinate descent algorithms for nonsmooth convex
  optimization.
\newblock In Guyon, I., Luxburg, U.~V., Bengio, S., Wallach, H., Fergus, R.,
  Vishwanathan, S., and Garnett, R. (eds.), \emph{Advances in Neural
  Information Processing Systems 30}, pp.\  5852--5861. Curran Associates,
  Inc., 2017.
\newblock URL
  \url{http://papers.nips.cc/paper/7167-smooth-primal-dual-coordinate-descent-algorithms-for-nonsmooth-convex-optimization.pdf}.

\bibitem[Allen-Zhu et~al.(2016)Allen-Zhu, Qu, Richtarik, and
  Yuan]{allen2016even}
Allen-Zhu, Z., Qu, Z., Richtarik, P., and Yuan, Y.
\newblock Even faster accelerated coordinate descent using non-uniform
  sampling.
\newblock In Balcan, M.~F. and Weinberger, K.~Q. (eds.), \emph{Proceedings of
  The 33rd International Conference on Machine Learning}, volume~48 of
  \emph{Proceedings of Machine Learning Research}, pp.\  1110--1119, New York,
  New York, USA, 20--22 Jun 2016. PMLR.
\newblock URL \url{http://proceedings.mlr.press/v48/allen-zhuc16.html}.
\newblock First appeared in arXiv:1512.09103.

\bibitem[Altschuler et~al.(2017)Altschuler, Weed, and
  Rigollet]{altschuler2017near-linear}
Altschuler, J., Weed, J., and Rigollet, P.
\newblock Near-linear time approxfimation algorithms for optimal transport via
  sinkhorn iteration.
\newblock In Guyon, I., Luxburg, U.~V., Bengio, S., Wallach, H., Fergus, R.,
  Vishwanathan, S., and Garnett, R. (eds.), \emph{Advances in Neural
  Information Processing Systems 30}, pp.\  1961--1971. Curran Associates,
  Inc., 2017.
\newblock arXiv:1705.09634.

\bibitem[Andresen \& Spokoiny(2016)Andresen and
  Spokoiny]{andresen2016convergence}
Andresen, A. and Spokoiny, V.
\newblock Convergence of an alternating maximization procedure.
\newblock \emph{Journal of Machine Learning Research}, 17\penalty0
  (63):\penalty0 1--53, 2016.
\newblock URL \url{http://jmlr.org/papers/v17/15-392.html}.

\bibitem[Anikin et~al.(2017)Anikin, Gasnikov, Dvurechensky, Tyurin, and
  Chernov]{anikin2017dual}
Anikin, A.~S., Gasnikov, A.~V., Dvurechensky, P.~E., Tyurin, A.~I., and
  Chernov, A.~V.
\newblock Dual approaches to the minimization of strongly convex functionals
  with a simple structure under affine constraints.
\newblock \emph{Computational Mathematics and Mathematical Physics},
  57\penalty0 (8):\penalty0 1262--1276, 2017.

\bibitem[Arjovsky et~al.(2017)Arjovsky, Chintala, and
  Bottou]{arjovsky2017wasserstein}
Arjovsky, M., Chintala, S., and Bottou, L.
\newblock {W}asserstein generative adversarial networks.
\newblock In Precup, D. and Teh, Y.~W. (eds.), \emph{Proceedings of the 34th
  International Conference on Machine Learning}, volume~70 of \emph{Proceedings
  of Machine Learning Research}, pp.\  214--223. PMLR, 06--11 Aug 2017.
\newblock URL \url{http://proceedings.mlr.press/v70/arjovsky17a.html}.

\bibitem[Beck(2015)]{beck2015convergence}
Beck, A.
\newblock On the convergence of alternating minimization for convex programming
  with applications to iteratively reweighted least squares and decomposition
  schemes.
\newblock \emph{SIAM Journal on Optimization}, 25\penalty0 (1):\penalty0
  185--209, 2015.

\bibitem[Beck \& Tetruashvili(2013)Beck and Tetruashvili]{beck2013convergence}
Beck, A. and Tetruashvili, L.
\newblock On the convergence of block coordinate descent type methods.
\newblock \emph{SIAM Journal on Optimization}, 23\penalty0 (4):\penalty0
  2037--2060, 2013.

\bibitem[Benamou et~al.(2015)Benamou, Carlier, Cuturi, Nenna, and
  Peyré]{benamou2015iterative}
Benamou, J.-D., Carlier, G., Cuturi, M., Nenna, L., and Peyré, G.
\newblock Iterative bregman projections for regularized transportation
  problems.
\newblock \emph{SIAM Journal on Scientific Computing}, 37\penalty0
  (2):\penalty0 A1111--A1138, 2015.

\bibitem[Bertsekas \& Tsitsiklis(1989)Bertsekas and
  Tsitsiklis]{bertsekas1989parallel}
Bertsekas, D.~P. and Tsitsiklis, J.~N.
\newblock \emph{Parallel and distributed computation: numerical methods},
  volume~23.
\newblock Prentice hall Englewood Cliffs, NJ, 1989.

\bibitem[Blondel et~al.(2018)Blondel, Seguy, and Rolet]{blondel2017smooth}
Blondel, M., Seguy, V., and Rolet, A.
\newblock Smooth and sparse optimal transport.
\newblock In Storkey, A. and Perez-Cruz, F. (eds.), \emph{Proceedings of the
  Twenty-First International Conference on Artificial Intelligence and
  Statistics}, volume~84 of \emph{Proceedings of Machine Learning Research},
  pp.\  880--889. PMLR, 09--11 Apr 2018.
\newblock URL \url{http://proceedings.mlr.press/v84/blondel18a.html}.

\bibitem[Buza(2014)]{buza2014feedback}
Buza, K.
\newblock Feedback prediction for blogs.
\newblock In \emph{Data analysis, machine learning and knowledge discovery},
  pp.\  145--152. Springer, 2014.

\bibitem[Chernov et~al.(2016)Chernov, Dvurechensky, and
  Gasnikov]{chernov2016fast}
Chernov, A., Dvurechensky, P., and Gasnikov, A.
\newblock Fast primal-dual gradient method for strongly convex minimization
  problems with linear constraints.
\newblock In Kochetov, Y., Khachay, M., Beresnev, V., Nurminski, E., and
  Pardalos, P. (eds.), \emph{Discrete Optimization and Operations Research: 9th
  International Conference, DOOR 2016, Vladivostok, Russia, September 19-23,
  2016, Proceedings}, pp.\  391--403. Springer International Publishing, 2016.

\bibitem[Claici et~al.(2018)Claici, Chien, and Solomon]{claici2018stochastic}
Claici, S., Chien, E., and Solomon, J.
\newblock Stochastic {W}asserstein barycenters.
\newblock In Dy, J. and Krause, A. (eds.), \emph{Proceedings of the 35th
  International Conference on Machine Learning}, volume~80 of \emph{Proceedings
  of Machine Learning Research}, pp.\  999--1008. PMLR, 2018.
\newblock URL \url{http://proceedings.mlr.press/v80/claici18a.html}.

\bibitem[Cuturi(2013)]{cuturi2013sinkhorn}
Cuturi, M.
\newblock Sinkhorn distances: Lightspeed computation of optimal transport.
\newblock In Burges, C. J.~C., Bottou, L., Welling, M., Ghahramani, Z., and
  Weinberger, K.~Q. (eds.), \emph{Advances in Neural Information Processing
  Systems 26}, pp.\  2292--2300. Curran Associates, Inc., 2013.

\bibitem[Cuturi \& Doucet(2014)Cuturi and Doucet]{cuturi2014fast}
Cuturi, M. and Doucet, A.
\newblock Fast computation of wasserstein barycenters.
\newblock In Xing, E.~P. and Jebara, T. (eds.), \emph{Proceedings of the 31st
  International Conference on Machine Learning}, volume~32 of \emph{Proceedings
  of Machine Learning Research}, pp.\  685--693, Bejing, China, 22--24 Jun
  2014. PMLR.
\newblock URL \url{http://proceedings.mlr.press/v32/cuturi14.html}.

\bibitem[Cuturi \& Peyr\'e(2016)Cuturi and Peyr\'e]{cuturi2016smoothed}
Cuturi, M. and Peyr\'e, G.
\newblock A smoothed dual approach for variational wasserstein problems.
\newblock \emph{SIAM Journal on Imaging Sciences}, 9\penalty0 (1):\penalty0
  320--343, 2016.

\bibitem[Daubechies et~al.(2010)Daubechies, DeVore, Fornasier, and
  G{\"u}nt{\"u}rk]{daubechies2010iteratively}
Daubechies, I., DeVore, R., Fornasier, M., and G{\"u}nt{\"u}rk, C.~S.
\newblock Iteratively reweighted least squares minimization for sparse
  recovery.
\newblock \emph{Communications on Pure and Applied Mathematics}, 63\penalty0
  (1):\penalty0 1--38, 2010.
\newblock \doi{10.1002/cpa.20303}.
\newblock URL \url{https://onlinelibrary.wiley.com/doi/abs/10.1002/cpa.20303}.

\bibitem[Diakonikolas \& Orecchia(2018{\natexlab{a}})Diakonikolas and
  Orecchia]{diakonikolas2018alternating}
Diakonikolas, J. and Orecchia, L.
\newblock Alternating randomized block coordinate descent.
\newblock In Dy, J. and Krause, A. (eds.), \emph{Proceedings of the 35th
  International Conference on Machine Learning}, volume~80 of \emph{Proceedings
  of Machine Learning Research}, pp.\  1224--1232, Stockholmsmässan, Stockholm
  Sweden, 10--15 Jul 2018{\natexlab{a}}. PMLR.
\newblock URL \url{http://proceedings.mlr.press/v80/diakonikolas18a.html}.

\bibitem[Diakonikolas \& Orecchia(2018{\natexlab{b}})Diakonikolas and
  Orecchia]{diakonikolas2018alternatingPr}
Diakonikolas, J. and Orecchia, L.
\newblock Alternating randomized block coordinate descent.
\newblock \emph{arXiv:1805.09185}, 2018{\natexlab{b}}.

\bibitem[Dvinskikh \& Tiapkin(2021)Dvinskikh and
  Tiapkin]{dvinskikh2021improved}
Dvinskikh, D. and Tiapkin, D.
\newblock Improved complexity bounds in wasserstein barycenter problem.
\newblock In Banerjee, A. and Fukumizu, K. (eds.), \emph{The 24th International
  Conference on Artificial Intelligence and Statistics, {AISTATS} 2021, April
  13-15, 2021, Virtual Event}, volume 130 of \emph{Proceedings of Machine
  Learning Research}, pp.\  1738--1746. {PMLR}, 2021.
\newblock URL \url{http://proceedings.mlr.press/v130/dvinskikh21a.html}.

\bibitem[Dvinskikh et~al.(2019)Dvinskikh, Gorbunov, Gasnikov, Dvurechensky, and
  Uribe]{dvinskikh2019primal}
Dvinskikh, D., Gorbunov, E., Gasnikov, A., Dvurechensky, P., and Uribe, C.~A.
\newblock On primal and dual approaches for distributed stochastic convex
  optimization over networks.
\newblock In \emph{2019 IEEE 58th Conference on Decision and Control (CDC)},
  pp.\  7435--7440, 2019.
\newblock \doi{10.1109/CDC40024.2019.9029798}.
\newblock arXiv:1903.09844.

\bibitem[Dvurechensky et~al.(2016)Dvurechensky, Gasnikov, Gasnikova,
  Matsievsky, Rodomanov, and Usik]{dvurechensky2016primal-dual}
Dvurechensky, P., Gasnikov, A., Gasnikova, E., Matsievsky, S., Rodomanov, A.,
  and Usik, I.
\newblock Primal-dual method for searching equilibrium in hierarchical
  congestion population games.
\newblock In \emph{Supplementary Proceedings of the 9th International
  Conference on Discrete Optimization and Operations Research and Scientific
  School (DOOR 2016) Vladivostok, Russia, September 19 - 23, 2016}, pp.\
  584--595, 2016.
\newblock arXiv:1606.08988.

\bibitem[Dvurechensky et~al.(2018{\natexlab{a}})Dvurechensky, Dvinskikh,
  Gasnikov, Uribe, and Nedić]{dvurechensky2018decentralize}
Dvurechensky, P., Dvinskikh, D., Gasnikov, A., Uribe, C.~A., and Nedić, A.
\newblock Decentralize and randomize: Faster algorithm for {W}asserstein
  barycenters.
\newblock In Bengio, S., Wallach, H., Larochelle, H., Grauman, K.,
  Cesa-Bianchi, N., and Garnett, R. (eds.), \emph{Advances in Neural
  Information Processing Systems 31}, NeurIPS 2018, pp.\  10783--10793. Curran
  Associates, Inc., 2018{\natexlab{a}}.
\newblock URL
  \url{http://papers.nips.cc/paper/8274-decentralize-and-randomize-faster-algorithm-for-wasserstein-barycenters.pdf}.
\newblock arXiv:1802.04367.

\bibitem[Dvurechensky et~al.(2018{\natexlab{b}})Dvurechensky, Gasnikov, and
  Kroshnin]{dvurechensky2018computational}
Dvurechensky, P., Gasnikov, A., and Kroshnin, A.
\newblock Computational optimal transport: Complexity by accelerated gradient
  descent is better than by {S}inkhorn’s algorithm.
\newblock In Dy, J. and Krause, A. (eds.), \emph{Proceedings of the 35th
  International Conference on Machine Learning}, volume~80 of \emph{Proceedings
  of Machine Learning Research}, pp.\  1367--1376, 2018{\natexlab{b}}.
\newblock arXiv:1802.04367.

\bibitem[Dvurechensky et~al.(2020)Dvurechensky, Gasnikov, Omelchenko, and
  Tiurin]{dvurechensky2020stable}
Dvurechensky, P., Gasnikov, A., Omelchenko, S., and Tiurin, A.
\newblock A stable alternative to {S}inkhorn's algorithm for regularized
  optimal transport.
\newblock In Kononov, A., Khachay, M., Kalyagin, V.~A., and Pardalos, P.
  (eds.), \emph{Mathematical Optimization Theory and Operations Research}, pp.\
   406--423, Cham, 2020. Springer International Publishing.
\newblock ISBN 978-3-030-49988-4.

\bibitem[Essid \& Solomon(2018)Essid and Solomon]{essid2018quadratically}
Essid, M. and Solomon, J.
\newblock Quadratically regularized optimal transport on graphs.
\newblock \emph{SIAM Journal on Scientific Computing}, 40\penalty0
  (4):\penalty0 A1961--A1986, 2018.
\newblock arXiv:1704.08200.

\bibitem[Fercoq \& Richt{\'a}rik(2015)Fercoq and
  Richt{\'a}rik]{fercoq2015accelerated}
Fercoq, O. and Richt{\'a}rik, P.
\newblock Accelerated, parallel, and proximal coordinate descent.
\newblock \emph{SIAM Journal on Optimization}, 25\penalty0 (4):\penalty0
  1997--2023, 2015.
\newblock First appeared in arXiv:1312.5799.

\bibitem[Ge et~al.(2019)Ge, Wang, Xiong, and Ye]{Ge-2019-Interior}
Ge, D., Wang, H., Xiong, Z., and Ye, Y.
\newblock Interior-point methods strike back: Solving the wasserstein
  barycenter problem.
\newblock In \emph{Advances in Neural Information Processing Systems 32}, pp.\
  6894--6905. Curran Associates, Inc., 2019.
\newblock URL
  \url{http://papers.nips.cc/paper/8913-interior-point-methods-strike-back-solving-the-wasserstein-barycenter-problem.pdf}.

\bibitem[Guminov et~al.(2019)Guminov, Nesterov, Dvurechensky, and
  Gasnikov]{guminov2019accelerated}
Guminov, S.~V., Nesterov, Y.~E., Dvurechensky, P.~E., and Gasnikov, A.~V.
\newblock Accelerated primal-dual gradient descent with linesearch for convex,
  nonconvex, and nonsmooth optimization problems.
\newblock \emph{Doklady Mathematics}, 99\penalty0 (2):\penalty0 125--128, Mar
  2019.

\bibitem[Hu et~al.(2008)Hu, Koren, and Volinsky]{hu2008collaborative}
Hu, Y., Koren, Y., and Volinsky, C.
\newblock Collaborative filtering for implicit feedback datasets.
\newblock In \emph{Proceedings of the 2008 Eighth IEEE International Conference
  on Data Mining}, ICDM ’08, pp.\  263–272, USA, 2008. IEEE Computer
  Society.
\newblock ISBN 9780769535029.
\newblock \doi{10.1109/ICDM.2008.22}.
\newblock URL \url{https://doi.org/10.1109/ICDM.2008.22}.

\bibitem[Jambulapati et~al.(2019)Jambulapati, Sidford, and
  Tian]{jambulapati2019direct}
Jambulapati, A., Sidford, A., and Tian, K.
\newblock A direct tilde$o(1/\varepsilon)$ iteration parallel algorithm for
  optimal transport.
\newblock In Wallach, H., Larochelle, H., Beygelzimer, A., d'Alch\'{e} Buc, F.,
  Fox, E., and Garnett, R. (eds.), \emph{Advances in Neural Information
  Processing Systems 32}, pp.\  11359--11370. Curran Associates, Inc., 2019.
\newblock URL
  \url{http://papers.nips.cc/paper/9313-a-direct-tildeo1epsilon-iteration-parallel-algorithm-for-optimal-transport.pdf}.

\bibitem[Khenissi \& Nasraoui(2019)Khenissi and Nasraoui]{khenissi2019modeling}
Khenissi, S. and Nasraoui, O.
\newblock Modeling and counteracting exposure bias in recommender systems, Dec
  2019.
\newblock URL \url{https://doi.org/10.18297/etd/3182}.

\bibitem[Krawtschenko et~al.(2020)Krawtschenko, Uribe, Gasnikov, and
  Dvurechensky]{krawtschenko2020distributed}
Krawtschenko, R., Uribe, C.~A., Gasnikov, A., and Dvurechensky, P.
\newblock Distributed optimization with quantization for computing wasserstein
  barycenters.
\newblock \emph{arXiv:2010.14325}, 2020.
\newblock \doi{10.20347/WIAS.PREPRINT.2782}.
\newblock WIAS preprint 2782.

\bibitem[Kroshnin et~al.(2019)Kroshnin, Tupitsa, Dvinskikh, Dvurechensky,
  Gasnikov, and Uribe]{kroshnin19a}
Kroshnin, A., Tupitsa, N., Dvinskikh, D., Dvurechensky, P.~E., Gasnikov, A.,
  and Uribe, C.~A.
\newblock On the complexity of approximating wasserstein barycenters.
\newblock In \emph{Proceedings of the 36th International Conference on Machine
  Learning, {ICML} 2019, 9-15 June 2019, Long Beach, California, {USA}}, pp.\
  3530--3540, 2019.
\newblock URL \url{http://proceedings.mlr.press/v97/kroshnin19a.html}.

\bibitem[Lee \& Sidford(2013)Lee and Sidford]{lee2013efficient}
Lee, Y.~T. and Sidford, A.
\newblock Efficient accelerated coordinate descent methods and faster
  algorithms for solving linear systems.
\newblock In \emph{Proceedings of the 2013 IEEE 54th Annual Symposium on
  Foundations of Computer Science}, FOCS '13, pp.\  147--156, Washington, DC,
  USA, 2013. IEEE Computer Society.
\newblock ISBN 978-0-7695-5135-7.
\newblock \doi{10.1109/FOCS.2013.24}.
\newblock URL \url{http://dx.doi.org/10.1109/FOCS.2013.24}.
\newblock First appeared in arXiv:1305.1922.

\bibitem[Lin et~al.(2014)Lin, Lu, and Xiao]{lin2014accelerated}
Lin, Q., Lu, Z., and Xiao, L.
\newblock An accelerated proximal coordinate gradient method.
\newblock In Ghahramani, Z., Welling, M., Cortes, C., Lawrence, N.~D., and
  Weinberger, K.~Q. (eds.), \emph{Advances in Neural Information Processing
  Systems 27}, pp.\  3059--3067. Curran Associates, Inc., 2014.
\newblock First appeared in arXiv:1407.1296.

\bibitem[Lin et~al.(2019{\natexlab{a}})Lin, Ho, and Jordan]{lin2019efficient}
Lin, T., Ho, N., and Jordan, M.
\newblock On efficient optimal transport: An analysis of greedy and accelerated
  mirror descent algorithms.
\newblock In Chaudhuri, K. and Salakhutdinov, R. (eds.), \emph{Proceedings of
  the 36th International Conference on Machine Learning}, volume~97 of
  \emph{Proceedings of Machine Learning Research}, pp.\  3982--3991, Long
  Beach, California, USA, 09--15 Jun 2019{\natexlab{a}}. PMLR.
\newblock URL \url{http://proceedings.mlr.press/v97/lin19a.html}.

\bibitem[Lin et~al.(2019{\natexlab{b}})Lin, Ho, and Jordan]{lin2019efficiency}
Lin, T., Ho, N., and Jordan, M.~I.
\newblock On the efficiency of the {S}inkhorn and {G}reenkhorn algorithms and
  their acceleration for optimal transport.
\newblock \emph{arXiv:1906.01437}, 2019{\natexlab{b}}.

\bibitem[{Lin} et~al.(2020){Lin}, {Ho}, {Chen}, {Cuturi}, and
  {Jordan}]{2020arXiv200204783L}
{Lin}, T., {Ho}, N., {Chen}, X., {Cuturi}, M., and {Jordan}, M.~I.
\newblock {Computational Hardness and Fast Algorithm for Fixed-Support
  Wasserstein Barycenter}.
\newblock \emph{arXiv e-prints}, art. arXiv:2002.04783, February 2020.

\bibitem[Lu et~al.(2018)Lu, Freund, and Mirrokni]{lu2018accelerating}
Lu, H., Freund, R., and Mirrokni, V.
\newblock Accelerating greedy coordinate descent methods.
\newblock In Dy, J. and Krause, A. (eds.), \emph{Proceedings of the 35th
  International Conference on Machine Learning}, volume~80 of \emph{Proceedings
  of Machine Learning Research}, pp.\  3257--3266, Stockholmsmässan, Stockholm
  Sweden, 10--15 Jul 2018. PMLR.
\newblock URL \url{http://proceedings.mlr.press/v80/lu18b.html}.

\bibitem[McCullagh \& Nelder(1989)McCullagh and
  Nelder]{mccullagh1989generalized}
McCullagh, P. and Nelder, J.
\newblock \emph{Generalized Linear Models, Second Edition}.
\newblock Chapman and Hall/CRC Monographs on Statistics and Applied Probability
  Series. Chapman \& Hall, 1989.
\newblock ISBN 9780412317606.

\bibitem[McLachlan \& Krishnan(1996)McLachlan and
  Krishnan]{mclachlan1996algorithm}
McLachlan, G. and Krishnan, T.
\newblock \emph{The EM Algorithm and Extensions}.
\newblock Wiley Series in Probability and Statistics. Wiley, 1996.

\bibitem[Nesterov(1983)]{nesterov1983method}
Nesterov, Y.
\newblock A method of solving a convex programming problem with convergence
  rate $o(1/k^2)$.
\newblock \emph{Soviet Mathematics Doklady}, 27\penalty0 (2):\penalty0
  372--376, 1983.

\bibitem[Nesterov(2004)]{nesterov2004introduction}
Nesterov, Y.
\newblock \emph{Introductory Lectures on Convex Optimization: a basic course}.
\newblock Kluwer Academic Publishers, Massachusetts, 2004.

\bibitem[Nesterov(2005)]{nesterov2005smooth}
Nesterov, Y.
\newblock Smooth minimization of non-smooth functions.
\newblock \emph{Mathematical Programming}, 103\penalty0 (1):\penalty0 127--152,
  2005.

\bibitem[Nesterov(2012)]{nesterov2012efficiency}
Nesterov, Y.
\newblock Efficiency of coordinate descent methods on huge-scale optimization
  problems.
\newblock \emph{SIAM Journal on Optimization}, 22\penalty0 (2):\penalty0
  341--362, 2012.
\newblock \doi{10.1137/100802001}.
\newblock URL \url{https://doi.org/10.1137/100802001}.
\newblock First appeared in 2010 as CORE discussion paper 2010/2.

\bibitem[Nesterov(2013)]{nesterov2013gradient}
Nesterov, Y.
\newblock Gradient methods for minimizing composite functions.
\newblock \emph{Mathematical Programming}, 140\penalty0 (1):\penalty0 125--161,
  2013.
\newblock First appeared in 2007 as CORE discussion paper 2007/76.

\bibitem[Nesterov \& Stich(2017)Nesterov and Stich]{nesterov2017efficiency}
Nesterov, Y. and Stich, S.~U.
\newblock Efficiency of the accelerated coordinate descent method on structured
  optimization problems.
\newblock \emph{SIAM Journal on Optimization}, 27\penalty0 (1):\penalty0
  110--123, 2017.
\newblock \doi{10.1137/16M1060182}.
\newblock URL \url{https://doi.org/10.1137/16M1060182}.
\newblock First presented in May 2015
  \url{http://www.mathnet.ru:8080/PresentFiles/11909/7_nesterov.pdf}.

\bibitem[Nesterov et~al.(2020)Nesterov, Gasnikov, Guminov, and
  Dvurechensky]{nesterov2020primal-dual}
Nesterov, Y., Gasnikov, A., Guminov, S., and Dvurechensky, P.
\newblock Primal-dual accelerated gradient methods with small-dimensional
  relaxation oracle.
\newblock \emph{Optimization Methods and Software}, pp.\  1--28, 2020.
\newblock \doi{10.1080/10556788.2020.1731747}.
\newblock URL \url{https://doi.org/10.1080/10556788.2020.1731747}.
\newblock arXiv:1809.05895.

\bibitem[Ortega \& Rheinboldt(1970)Ortega and Rheinboldt]{ortega1970iterative}
Ortega, J. and Rheinboldt, W.
\newblock \emph{Iterative Solution of Nonlinear Equations in Several
  Variables}.
\newblock Classics in Applied Mathematics. Society for Industrial and Applied
  Mathematics, 1970.
\newblock ISBN 9780898714616.
\newblock URL \url{https://books.google.de/books?id=UEFRfUZBpUEC}.

\bibitem[Rogozin et~al.(2021)Rogozin, Beznosikov, Dvinskikh, Kovalev,
  Dvurechensky, and Gasnikov]{rogozin2021decentralized}
Rogozin, A., Beznosikov, A., Dvinskikh, D., Kovalev, D., Dvurechensky, P., and
  Gasnikov, A.
\newblock Decentralized distributed optimization for saddle point problems.
\newblock \emph{arXiv:2102.07758}, 2021.

\bibitem[Saha \& Tewari(2013)Saha and Tewari]{saha2013nonasymptotic}
Saha, A. and Tewari, A.
\newblock On the nonasymptotic convergence of cyclic coordinate descent
  methods.
\newblock \emph{SIAM Journal on Optimization}, 23\penalty0 (1):\penalty0
  576--601, 2013.

\bibitem[Shalev-Shwartz \& Zhang(2014)Shalev-Shwartz and
  Zhang]{shalev-shwartz2014accelerated}
Shalev-Shwartz, S. and Zhang, T.
\newblock Accelerated proximal stochastic dual coordinate ascent for
  regularized loss minimization.
\newblock In Xing, E.~P. and Jebara, T. (eds.), \emph{Proceedings of the 31st
  International Conference on Machine Learning}, volume~32 of \emph{Proceedings
  of Machine Learning Research}, pp.\  64--72, Bejing, China, 22--24 Jun 2014.
  PMLR.
\newblock URL \url{http://proceedings.mlr.press/v32/shalev-shwartz14.html}.
\newblock First appeared in arXiv:1309.2375.

\bibitem[Sinkhorn(1974)]{sinkhorn1974diagonal}
Sinkhorn, R.
\newblock Diagonal equivalence to matrices with prescribed row and column sums.
  {II}.
\newblock \emph{Proc. Amer. Math. Soc.}, 45:\penalty0 195--198, 1974.

\bibitem[Staib et~al.(2017)Staib, Claici, Solomon, and
  Jegelka]{staib2017parallel}
Staib, M., Claici, S., Solomon, J.~M., and Jegelka, S.
\newblock Parallel streaming wasserstein barycenters.
\newblock In Guyon, I., Luxburg, U.~V., Bengio, S., Wallach, H., Fergus, R.,
  Vishwanathan, S., and Garnett, R. (eds.), \emph{Advances in Neural
  Information Processing Systems 30}, pp.\  2647--2658. Curran Associates,
  Inc., 2017.
\newblock URL
  \url{http://papers.nips.cc/paper/6858-parallel-streaming-wasserstein-barycenters.pdf}.

\bibitem[Su et~al.(2016)Su, Boyd, and Cand{{\`e}}s]{JMLR:v17:15-084}
Su, W., Boyd, S., and Cand{{\`e}}s, E.~J.
\newblock A differential equation for modeling nesterov's accelerated gradient
  method: Theory and insights.
\newblock \emph{Journal of Machine Learning Research}, 17\penalty0
  (153):\penalty0 1--43, 2016.
\newblock URL \url{http://jmlr.org/papers/v17/15-084.html}.

\bibitem[Sun \& Hong(2015)Sun and Hong]{sun2015improved}
Sun, R. and Hong, M.
\newblock Improved iteration complexity bounds of cyclic block coordinate
  descent for convex problems.
\newblock In \emph{Proceedings of the 28th International Conference on Neural
  Information Processing Systems - Volume 1}, NIPS'15, pp.\  1306--1314,
  Cambridge, MA, USA, 2015. MIT Press.
\newblock URL \url{http://dl.acm.org/citation.cfm?id=2969239.2969385}.

\bibitem[Tiapkin et~al.(2020)Tiapkin, Gasnikov, and
  Dvurechensky]{tiapkin2020stochastic}
Tiapkin, D., Gasnikov, A., and Dvurechensky, P.
\newblock Stochastic saddle-point optimization for wasserstein barycenters.
\newblock \emph{arXiv:2006.06763}, 2020.

\bibitem[Tupitsa et~al.(2020)Tupitsa, Dvurechensky, Gasnikov, and
  Uribe]{tupitsa2020multimarginal}
Tupitsa, N., Dvurechensky, P., Gasnikov, A., and Uribe, C.~A.
\newblock Multimarginal optimal transport by accelerated alternating
  minimization.
\newblock In \emph{2020 59th IEEE Conference on Decision and Control (CDC)},
  pp.\  6132--6137, 2020.
\newblock \doi{10.1109/CDC42340.2020.9304010}.
\newblock URL \url{https://ieeexplore.ieee.org/document/9304010}.
\newblock arXiv:2004.02294.

\bibitem[Tupitsa et~al.(2021)Tupitsa, Dvurechensky, Gasnikov, and
  Guminov]{tupitsa2021alternating}
Tupitsa, N., Dvurechensky, P., Gasnikov, A., and Guminov, S.
\newblock Alternating minimization methods for strongly convex optimization.
\newblock \emph{Journal of Inverse and Ill-posed Problems}, 2021.
\newblock \doi{doi:10.1515/jiip-2020-0074}.
\newblock URL \url{https://doi.org/10.1515/jiip-2020-0074}.
\newblock WIAS Preprint No. 2692, arXiv:1911.08987.

\bibitem[Uribe et~al.(2018)Uribe, Dvinskikh, Dvurechensky, Gasnikov, and
  Nedi\'c]{uribe2018distributed}
Uribe, C.~A., Dvinskikh, D., Dvurechensky, P., Gasnikov, A., and Nedi\'c, A.
\newblock Distributed computation of {W}asserstein barycenters over networks.
\newblock In \emph{2018 IEEE 57th Annual Conference on Decision and Control
  (CDC)}, 2018.
\newblock Accepted, arXiv:1803.02933.

\bibitem[{Yang} et~al.(2018){Yang}, {Li}, {Sun}, and {Toh}]{Yang-2018-ADMM}
{Yang}, L., {Li}, J., {Sun}, D., and {Toh}, K.-C.
\newblock {A Fast Globally Linearly Convergent Algorithm for the Computation of
  Wasserstein Barycenters}.
\newblock \emph{arXiv e-prints}, art. arXiv:1809.04249, September 2018.

\bibitem[Ye et~al.(2017)Ye, Wu, Wang, and Li]{Ye-2017-Fast}
Ye, J., Wu, P., Wang, J.~Z., and Li, J.
\newblock Fast discrete distribution clustering using wasserstein barycenter
  with sparse support.
\newblock \emph{Trans. Sig. Proc.}, 65\penalty0 (9):\penalty0 2317–2332, May
  2017.
\newblock ISSN 1053-587X.
\newblock \doi{10.1109/TSP.2017.2659647}.
\newblock URL \url{https://doi.org/10.1109/TSP.2017.2659647}.

\end{thebibliography}
\bibliographystyle{icml2021}

\part*{Supplementary materials}

\section{Omitted proofs in Section~\ref{S:AAM}: Accelerated Alternating Minimization}




\begin{proof}[Proof of Lemma~\ref{AAM-2_Ak_rate}]
Let us introduce an auxiliary  sequence of functions defined as \begin{equation*}
    \psi_0(x)=\frac{1}{2}\|x-x^0\|^2,\quad
    \psi_{k+1}(x) = \psi_{k}(x) + a_{k+1}\{f(y^k) + \langle \nabla f(y^k), x - y^k \rangle\}.
\end{equation*}

It is easy to see that $v^{k}=\argmin_{x\in\mathbb{R}^N} \psi_{k}(x).$

Now, we prove inequality \eqref{eq:main_recurrence} by induction over $k$. For $k=0$, the inequality holds. Assume that 
\[A_{k}f(x^{k}) \leqslant \min_{x \in \mathbb{R}^N} \psi_{k}(x) = \psi_{k}(v^{k}).\]
Then 
\begin{multline*}
    \psi_{k+1}(v^{k+1}) =\min_{x \in \mathbb{R}^N} \left\{ \psi_{k}(x) + a_{k+1}\{f(y^k) + \langle \nabla f(y^k), x - y^k \rangle\} \right\}\geqslant
    \\
    \geqslant \min_{x \in \mathbb{R}^N} \bigg\{ \psi_{k}(v^k) +\frac{1}{2}\|x-v^k\|_2^2+ a_{k+1}\{f(y^k) + \langle \nabla f(y^k), x - y^k \rangle\} \bigg\} \geqslant
    \\
    \geqslant \psi_{k}(v^k)+a_{k+1}f(y^k) -\frac{a_{k+1}^2}{2}\|\nabla f(y^k)\|_2^2 + a_{k+1}\langle \nabla f(y^k), v^k - y^k\rangle\geqslant
    \\
    \geqslant A_kf(x^k) +a_{k+1}f(y^k)-\frac{a_{k+1}^2}{2}\|\nabla f(y^k)\|_2^2+ a_{k+1}\langle \nabla f(y^k), v^k - y^k\rangle 
    \\
    \geqslant A_{k+1} f(y^k) - \frac{a_{k+1}^2}{2}\|\nabla f(y^k)\|_2^2+ a_{k+1}\langle \nabla f(y^k), v^k - y^k\rangle.
\end{multline*}
Here we used that $\psi_{k}$ is a strongly convex function with minimum at $v^k$ and that $f(y^k)\leqslant f(x^k)$. By the optimality conditions for the problem $\min\limits_{\beta\in [0,1]} f\left(x^k + \beta (v^k - x^k)\right)$, there are three possibilities
\begin{enumerate}
    \item $\beta_k = 1$, $\langle \nabla f(y^k),x^k - v^k \rangle \geqslant 0$, $y^k = v^k$;
    \item $\beta_k \in (0,1)$ and $\langle \nabla f(y^k),x^k - v^k \rangle = 0$, $y^k = v^k + \beta_k (x^k - v^k)$;
    \item $\beta_k = 0$ and $\langle \nabla f(y^k),x^k - v^k \rangle \leqslant 0$, $y^k = x^k$ .
\end{enumerate}In all three cases,  $\langle \nabla f(y^k), v^k - y^k \rangle \geqslant 0$.

Using the rule for choosing $a_{k+1}$ in the method, we finish the proof of the induction step: $$\psi_{k+1}(v^{k+1})\geqslant A_{k+1}f(x^{k+1}).$$ It remains to show that the equation
\begin{equation}
    \label{eq:Th:Main_proof_1}
    f(y^k) - \frac{a_{k+1}^2}{2A_{k+1}} \|\nabla f(y^k) \|_2^2 = f(x^{k+1}).
\end{equation}
has a solution $a_{k+1} > 0$.
By the $L$-smoothness of the objective, we have, for all $i \geq 0$,
\[f(y^k)-\frac{1}{2L}\|\nabla_i f(y^k)\|_2^2\geqslant f(x_i^{k+1}),\] where $x_i^{k+1}=\argmin_{x\in S_{i}} f(x)$. Since $A_{k+1}=A_k + a_{k+1}$, we can rewrite \eqref{eq:Th:Main_proof_1} as 
\begin{equation*}
    \frac{a_{k+1}^2}{2} \|\nabla f(y^k) \|_2^2  + a_{k+1} (f(x^{k+1}) - f(y^k))+  A_k(f(x^{k+1}) - f(y^k)) = 0.
\end{equation*}
    
Since $f(x^{k+1}) - f(y^k) < 0$ (otherwise $\|\nabla f(y^k)\| = 0$ and $y_k$ is a solution to the problem), there exists solution $ a_{k+1}>0$.

Let us estimate the rate of the growth for $A_k$.  Since $i_k=\argmax_{i} \|\nabla_i f(y^k)\|_2^2$, 
\[\|\nabla_{i_k} f(y^k)\|_2^2\geqslant \frac{1}{n}\|\nabla f(y^k)\|_2^2.\]
As a consequence, we have
\begin{equation*}
    f(y^k)-\frac{1}{2Ln}\|\nabla f(y^k)\|_2^2\geqslant
     f(y^k)-\frac{1}{2L}\|\nabla_{i_k} f(y^k)\|_2^2\geqslant f(x^{k+1}).
\end{equation*}
This in combination with our rule for choosing $a_{k+1}$ implies $\frac{a_{k+1}^2}{2A_{k+1}}\geqslant \frac{1}{2Ln}$. Since $A_1 = a_1 \geqslant \frac{1}{Ln}$, we prove by induction that $a_k \geqslant \frac{k}{2Ln}$ and $A_{k} \geqslant \frac{(k+1)^2}{4nL} \geqslant \frac{k^2}{4nL}.$ Indeed,
\begin{equation}
	a_{k+1}  \geqslant \frac{1 + \sqrt{1 + 4A_k Ln}}{2Ln} = \frac{1}{2Ln} + \sqrt{\frac{1}{4L^2n^2} + \frac{A_{k}}{Ln}} \notag
	\geqslant 
	\frac{1}{2Ln} + \sqrt{\frac{A_{k}}{Ln}} \notag 
	\geqslant 
	\frac{1}{2Ln} + \frac{1}{\sqrt{L}}\frac{k+1}{2\sqrt{Ln}} =
	\frac{k+2}{2Ln}. \notag
\end{equation}

Hence,
\begin{equation*}
  A_{k+1} = A_k + a_{k+1} \geqslant \frac{(k+1)^2}{4Ln} + \frac{k+2}{2Ln} \geqslant \frac{(k+2)^2}{4Ln}.  
\end{equation*}

\end{proof}
\section{Omitted proofs in Section~\ref{S:primal-dual}: Primal-Dual Extension}



To prove  Theorem~\ref{PD-bounds}, we first prove a slightly more general result.

\begin{theorem}
\label{Th:PD_rate}
Let the objective $\phi$ in the problem $(P_2)$ be $L$-smooth w.r.t. $\|\cdot\|_2$ and the solution of this problem be bounded, i.e. $\|\lambda^*\|_2 \leqslant R$. Then, for the sequences $\hat{x}_{k+1},\eta_{k+1}$, $k\geqslant 0$, generated by Algorithm \ref{PDAAM-2}, 
\begin{equation*}
    \hspace{-1em}\|\bm{A} \hat{x}^k - b \|_2 \leqslant \frac{8nLR}{k^2},\  |\phi(\eta^k) + f(\hat{x}^k)| \leqslant\frac{8nLR^2}{k^2},\ \|\hat{x}^k-x^*\|_E \leqslant \frac{4}{k}\sqrt{\frac{2nLR^2}{\gamma}}.
\end{equation*}
\end{theorem}
\begin{proof}

Applying Lemma~\ref{AAM-2_Ak_rate} to problem $(P_2)$, we obtain 
\begin{equation}
    A_k \phi(\eta^k) \leqslant \min_{\lambda \in \Lambda} \left\{\sum_{j=0}^{k-1}\{a_{j+1} ( \phi(\lambda^j)\right.+\left. \langle \nabla \phi(\lambda^j), \lambda-\lambda^j \rangle) 
    + \frac{1}{2} \|\lambda\|_2^2 \vphantom{\sum_{j=0}^{k-1}}\right\},\label{eq:FGM_compl}
\end{equation}

Let us introduce the set $\Lambda_R =\{\lambda:  \|\lambda\|_2 \leqslant 2R \}$ where $R$ is such that $\|\lambda^{*}\|_{2} \leqslant R$.
Then, from \eqref{eq:FGM_compl}, we obtain for $h(\lambda) = \sum_{j=0}^{k-1}{a_{j+1} \left( \phi(\lambda^j) + \langle \nabla \phi(\lambda^j), \lambda-\lambda^j \rangle \right) } + \frac{1}{2} \|\lambda\|_2^2$
\begin{equation}
    A_k \phi(\eta^k) \leqslant \min_{\lambda \in \Lambda}  h(\lambda) \leqslant \min_{\lambda \in \Lambda_R}   h(\lambda)  \leqslant 
    2R^2+\min_{\lambda \in \Lambda_R} \left\{  \sum_{j=0}^{k-1}{a_{j+1}(\phi(\lambda^j) + \langle \nabla \phi(\lambda^j), \lambda-\lambda^j \rangle } \right\}.
    \label{eq:proof_st_1}
\end{equation}
On the other hand, from the definition 
$(P_2)$ of $\phi(\lambda)$, we have
\begin{equation*}
    \phi(\lambda^i)   = \langle \lambda^i, b \rangle + \max_{x\in Q} \left( -f(x) - \langle \bm{A}^T \lambda^i ,x \rangle \right) \notag 
    = \langle \lambda^i, b \rangle  - f(x(\lambda^i)) - \langle \bm{A}^T \lambda^i, x(\lambda^i) \rangle . \notag
\end{equation*}
Combining this equality with \eqref{eq:nvp}, we obtain
\begin{equation*}
    \phi(\lambda^i) - \langle \nabla \phi (\lambda^i), \lambda^i \rangle  
    =
    \langle \lambda^i, b \rangle - f(x(\lambda^i))
    - \langle \bm{A}^T \lambda^i, x(\lambda^i) \rangle - \langle b-\bm{A} x(\lambda^i),\lambda^i \rangle  = - f(x(\lambda^i)).
\end{equation*}
Summing these equalities from $i=0$ to $i=k-1$ with the weights $\{a_{i+1}\}_{i=0,...k-1}$, we get, using the convexity of $f$
\begin{multline*}
\sum_{i=0}^{k-1}a_{i+1}( \phi(\lambda^i) + \langle \nabla \phi(\lambda^i), \lambda-\lambda^i \rangle ) 
    \\
    =-\sum_{i=0}^{k-1} a_{i+1} f(x(\lambda^i)) + \sum_{i=0}^{k-1} a_{i+1} \langle (b-\bm{A} x(\lambda^i), \lambda \rangle 
    \leqslant -A_k f(\hat{x}^k) + A_{k} \langle b-\bm{A} \hat{x}^{k}, \lambda \rangle .
\end{multline*}
Substituting this inequality into \eqref{eq:proof_st_1}, we obtain
\begin{equation*}
    A_k \phi(\eta^k)  \leqslant -A_kf(\hat{x}^k)
    + \min_{\lambda \in \Lambda_R} \left\{  A_k\langle b-\bm{A} \hat{x}^k, \lambda\rangle \right\} + 2R^2 \notag
\end{equation*}
Finally, since 
$\max\limits_{\lambda \in \Lambda_R} \left\{  \langle -b+\bm{A} \hat{x}^k, \lambda \rangle \right\} = 2 R \|\bm{A} \hat{x}^k - b \|_2$,
we obtain
\begin{equation}
    A_k(\phi(\eta^k) + f(\hat{x}^k)) +2 RA_k \|\bm{A} \hat{x}^k - b\|_2 \leqslant 2R^2.
\label{eq:vpmfxh}
\end{equation}

Since  $\lambda^*$ is an optimal solution of Problem $(D_1)$, we have, for any $x \in Q$
$$
Opt[P_1]\leqslant f(x) + \langle  \lambda^*, \bm{A} x -b \rangle .
$$
Using the assumption that $\|\lambda^{*}\|_{2} \leqslant R$ 
, we get
\begin{equation}
f(\hat{x}^k) \geqslant Opt[P_1]- R \|\bm{A} \hat{x}^k - b \|_2 .
\label{eq:fxhat_est}
\end{equation}
Hence,
\begin{multline}
    \phi(\eta^k) + f(\hat{x}^k) =\phi(\eta^k) - Opt[P_2]+Opt[P_2]+ Opt[P_1]  - Opt[P_1] + f(\hat{x}^k)=  \\
    =\phi(\eta^k)  - Opt[P_2]  - Opt[P_1] + f(\hat{x}^k)  
    \geqslant  - Opt[P_1] + f(\hat{x}^k) \stackrel{\eqref{eq:fxhat_est}}{\geqslant} - R \|\bm{A} \hat{x}^k - b \|_2.
\label{eq:aux1}        
\end{multline}

This and \eqref{eq:vpmfxh} give
$R \|A_k(\bm{A} \hat{x}^k - b) \|_2 \leqslant 2R^2.$ Hence, from \eqref{eq:aux1} we obtain
$A_k(\phi(\eta^k) + f(\hat{x}^k)) \geqslant -2R^2.$ On the other hand, from \eqref{eq:vpmfxh} we have 
$A_k(\phi(\eta^k) + f(\hat{x}^k)) \leqslant 2R^2.$ Combining all of these results, we conclude
\begin{equation}
    A_k\|\bm{A} \hat{x}^k - b \|_2 \leqslant 2R,  \quad A_k|\phi(\eta^k) + f(\hat{x}^k)| \leqslant2R^2.\label{eq:untileq_PDAAM-1}
\end{equation}
From \ref{AAM-2_Ak_rate}, for any $k\geqslant 0$, $A_k \geqslant \frac{k^2}{4Ln}$. Combining this and \eqref{eq:untileq_PDAAM-1}, we obtain the first two inequalities of the statement:

\begin{equation}
     \|\bm{A} \hat{x}^k - b \|_2 \leqslant \frac{8nLR}{k^2},\  |\phi(\eta^k) + f(\hat{x}^k)| \leqslant\frac{8nLR^2}{k^2}.\notag
\end{equation}
It remains to prove the third inequality. By the optimality condition for Problem $(P_1)$, we have
$$
\la \nabla f(x^*) + \bm{A}^T \lambda^*, \hat{x}_k - x^* \ra \geq 0, \quad \bm{A} x^* = b.
$$
 Then
\begin{equation}
    \la \nabla f(x^*) , \hat{x}_k - x^* \ra  \geq - \la \bm{A}^T \lambda^*, \hat{x}_k - x^* \ra
     =  - \la \lambda^{*}, \bm{A} \hat{x}_k - b  \ra
     \geq - R \|\bm{A} \hat{x}_k - b \|_2\geq - \frac{8nLR^2}{k^2},
\end{equation}
where we used the same reasoning as while deriving \eqref{eq:fxhat_est}. 
Using this inequality and the $\gamma$-strong convexity of $f$, we obtain
\begin{multline*}
\frac{\gamma}{2}\|\hat{x}_k-x^*\|_E^2 \leq f(\hat{x}_k)-Opt[P_1] - \la \nabla f(x^*) , \hat{x}_k - x^* \ra\\\leq f(\hat{x}_k)+\phi(\eta^k) + \la \nabla f(x^*) , \hat{x}_k - x^* \ra\leqslant \frac{8nLR^2}{k^2}+\frac{8nLR^2}{k^2}=\frac{16nLR^2}{k^2}, 
\end{multline*} or
\[\|\hat{x}_k-x^*\|_E\leq \frac{4}{k}\sqrt{\frac{2nLR^2}{\gamma}}.\]
	
\end{proof}

\subsection*{Proof of Theorem~\ref{PD-bounds}}
\begin{proof}
The result follows from the previous theorem and the bound $L \leq \frac{\|\bm{A}\|^2_{E \rightarrow H}}{\gamma}$, which is shown in  \cite{nesterov2005smooth}.
\end{proof}

\section{Fixed-Step Accelerated Alternating Minimization}
In this section we introduce another variant of accelerated alternating minimization method. Algorithm 2 in the main text uses full relaxation on a segment to find the next iterate $y^k$. On the contrary, the method which we introduce in this section tries to adaptively find an approximation for the constant $L$ -- Lipschitz constant of the gradient. Based on this approximation, a fixed stepsize is used to find $y^k$.  
Thus, compared to the AAM algorithm presented in Section~\ref{S:AAM} of the main paper, this algorithm does not require solving any one-dimensional minimization problems during each iteration, but instead requires adapting to the smoothness parameter of the problem. This typically results in repeating each iteration twice. In our experience, which of the two method turns out to be more efficient significantly depends on the problem being solved (generally, the more difficult the function is to compute, the more taxing the line-search becomes) and the implementation of the line-search procedure. We also point out that we can not guarantee the convergence of Algorithm~\ref{AAM-1} to a stationary point for non-convex objectives.
In the experiments for the OT problem we use this algorithm and the result is denoted by AAM-A.

\begin{algorithm}[h]
 \caption{Fixed-Step Accelerated Alternating Minimization}
 \label{AAM-1}
 {\small
 \begin{algorithmic}[1]
   \REQUIRE starting point $x_0$, initial estimate of the Lipschitz constant $L_0$.
     \ENSURE $x^k$
   \STATE $x^0 = y^0= v^0$.
   \FOR{$k \geqslant 0$}
   \STATE Set $L_{k+1}=L_{k}/2$
   \WHILE{True}
   \STATE Set $a_{k+1}=\frac{1}{2L_{k+1}}+\sqrt{\frac{1}{4L^2_{k+1}}+a_k^2\frac{L_k}{L_{k+1}}}$
   Find $a_{k+1}$ s.t. $A_{k+1}:=a_{k+1}^2L_{k+1}=a_k^2L_k+a_{k+1}$.
   \STATE Set $\tau_k=\frac{1}{a_{k+1}L_{k+1}}$
 \STATE Set $y^{k}=\tau_k v^k+(1-\tau_k)x^k$\quad \COMMENT{Extrapolation step}
 \STATE Choose $i_k=\argmax\limits_{i\in\{1,\ldots,n\}} \|\nabla_i f(y^k)\|_2^2$
 \STATE Set $x^{k+1}=\argmin\limits_{x\in S_{i_k}(y^k)} f(x)$\quad  
 \STATE Set $v^{k+1} = v^{k} - a_{k+1}\nabla f(y^k)$\quad 
 \IF{$f(x^{k+1})\leqslant f(y^{k})-\frac{\|\nabla f(y^k)\|_2^2}{2L_{k+1}}$}
 \STATE \textbf{break}
 \ENDIF
 \STATE Set $L_{k+1}=2L_{k+1}$.
 \ENDWHILE
 \STATE $k = k + 1$
 \ENDFOR
 \end{algorithmic}
 }
 \end{algorithm}
The convergence rate of Algorithm \ref{AAM-1} is given by the following theorem

\begin{theorem} 
\label{AAM-1_convergence}
Let the objective $f$ be convex and $L$-smooth. If $L_0\leqslant 4nL$, then after $k$ steps of Algorithm \ref{AAM-1} it holds that
\begin{equation}
    \label{eq:main_result}
    f(x^{k})-f(x^*) \leqslant \frac{4nL\|x^0-x^*\|_2^2}{k^2}.
\end{equation}
\end{theorem}
Unlike the AM algorithm, this method requires computing the whole gradient of the objective, which makes the iterations of this algorithm considerably more expensive. Also, even when the number of blocks is 2, the convergence rate of Algorithm~\ref{AAM-1} depends on the smoothness parameter $L$ of the whole objective, and not on the Lipschitz constants of each block on its own, which is the case for the AM algorithm \cite{beck2015convergence}. On the other hand, if we compare the Algorithm~\ref{AAM-1} algorithm to an adaptive accelerated gradient method, we will see that the  theoretical worst-case time complexity of Algorithm~\ref{AAM-1} method is only  $\sqrt{n}$ times worse, while in practice block-wise minimization steps may perform much better than gradient descent steps simply because they directly use some specific structure of the objective.

This convergence rate is $n$ times worse than that of an adaptive accelerated gradient method \cite{dvurechensky2018computational}, or, equivalently, this means that in the worst case it may take $\sqrt{n}$ times more iterations to guarantee accuracy $\eps$ compared to an adaptive accelerated gradient method.
To prove the convergence rate of the method, we will need a technical result.

\begin{lemma}
\label{AAM-1_lemma_1}
For any $u\in\mathbb{R}^N$
\begin{equation*}
    a_{k+1}\langle\nabla f(y^{k}),v^k-u\rangle
    \leqslant a_{k+1}^2L_{k+1}\left(f(y^{k})-f(x^{k+1})\right)
    +\frac{1}{2}\|v^k-u\|_2^2-\frac{1}{2}\|v^{k+1}-u\|_2^2. 
\end{equation*}
\end{lemma}

\begin{proof}
\begin{multline*}
    a_{k+1} \langle\nabla f(y^{k}),v^k-u\rangle 
    = a_{k+1}\langle\nabla f(y^{k}),v^k-v^{k+1}\rangle
    + a_{k+1}\langle\nabla f(y^{k}),v^{k+1}-u\rangle
    \\
    = a_{k+1}^2\|\nabla f(y^k)\|_2^2
    + \la v^k-v^{k+1},v^{k+1}-u\ra 
    \\
    = a_{k+1}^2\|\nabla f(y^k)\|_2^2
    + \frac{1}{2}\|v^k-u\|_2^2
    - \frac{1}{2}\|v^{k+1}-u\|_2^2 
    - \frac{1}{2}\|v^{k+1}-v^k\|_2^2
    \\
    \leqslant a_{k+1}^2L_{k+1} \Big(f(y^{k}) 
    - f(x^{k+1})\Big) 
    + \frac{1}{2}\|v^k-u\|_2^2 
    - \frac{1}{2}\|v^{k+1}-u\|_2^2.
\end{multline*}
Here the last inequality follows from line 11 of Algorithm~\ref{AAM-1}.
\end{proof}

\begin{lemma}
\label{AAM-1_lemma_2}
For any $u\in\mathbb{R}^N$ and any $k\geqslant 0$
\begin{equation*}
    a_{k+1}^2L_{k+1}f(x^{k+1})-\left(a^2_{k+1}L_{k+1}-a_{k+1}\right)f(x^k)
    +\frac{1}{2}\|v^{k}-u\|_2^2-\frac{1}{2}\|v^{k+1}-u\|_2^2\leqslant a_{k+1}f(u). 
\end{equation*}
\end{lemma}
\begin{proof}
\begin{multline}
    a_{k+1}(f(y^{k})-f(u)) 
    \leq a_{k+1}\langle \nabla f(y^{k}), y^{k} - u\rangle 
    \\
    = a_{k+1}\langle \nabla f(y^k), y^k - v^k\rangle + a_{k+1}\langle \nabla f(y^k), v^k - u\rangle
    \\
    \stackrel{\scriptsize{\circled{1}}}{=} \frac{(1-\tau_k)a_{k+1}}{\tau_k}\langle \nabla f(y^k), x^k-y^k\rangle
    +a_{k+1}\langle \nabla f(y^k), v^k-u\rangle    
    \\
    \stackrel{\scriptsize{\circled{2}}}{\leqslant} \frac{(1-\tau_k)a_{k+1}}{\tau_k}\left(f(x^k)-f(y^k)\right)
    +a^2_{k+1}L_{k+1}\left(f(y^k)-f(x^{k+1})\right) 
    +\frac{1}{2}\|v^{k}-u\|_2^2
    -\frac{1}{2}\|v^{k+1}-u\|_2^2
    \\
    \stackrel{\scriptsize{\circled{3}}}{=} (a^2_{k+1}L_{k+1}-a_{k+1})f(x^k)
    -a_{k+1}^2L_{k+1}f(x^{k+1})
    +a_{k+1}f(y^k)
    +\frac{1}{2}\|v^{k}-u\|_2^2
    -\frac{1}{2}\|v^{k+1}-u\|_2^2.
\end{multline}
Here, $\circled{1}$ uses the fact that our choice of $y^k$ satisfies $\tau_k(y^k-v^k)=(1-\tau_k)(x^k-y^k)$. $\circled{2}$ is by convexity of $f(\cdot)$ and Lemma~\ref{AAM-1_lemma_1} , while $\circled{3}$ uses the choice of $\tau_k=\frac{1}{a_{k+1}L_{k+1}}$.
\end{proof}

\begin{proof}[\textbf{\textup{Proof of Theorem~\ref{AAM-1_convergence}}}.]
Note that 
$$a_{k+1} = \frac{1}{2L_{k+1}} + \sqrt{\frac{1}{4L^2_{k+1}} + a^2_k\frac{L_k}{L_{k+1}}}$$ 
satisfies the equation
$a_{k+1}^2L_{k+1}=a_k^2L_k+a_{k+1}$. We also have $a_1=\frac{1}{L_{k+1}}$. With that in mind, we sum up the inequality in the statement of  Lemma~\ref{AAM-1_lemma_2} for $k=0,\dots, T-1$ and set $u=x^*$:
\[
L_Ta^2_Tf(x^{T}) +\frac{1}{2}\|v^{0}-x^*\|_2^2-\frac{1}{2}\|v^{T}-x^*\|_2^2
    \leq \sum_{k=0}^{T-1} a_k f(x^*)=L_Ta^2_T f(x^*).
\]
Denote $A_k=a_k^2L_k$. Since $v^0=x^0$, we now have that for any $T\geq 1$
\[f(x^T)-f(x^*)\leqslant\frac{\|x^0-x^*\|_2^2}{2A_T}.\]
It remains to estimate $A_T$ from below. We will now show by induction that $A_k\geqslant\frac{nk^2}{8L}$.
From the $L$-smoothness of the objective we have 

\[
f(x^{k+1})=\argmin_{x\in S_{i_k}(y^k)} f(x) \leqslant f(y^k-\frac{1}{L}\nabla_{i_k} f(y^k))
    \leqslant f(y^k)-\frac{1}{2L}\|\nabla_{i_k} f(y^k)\|_2^2.
\]

Also, since $i_k$ is chosen by the Gauss--Southwell rule, it is true that \[\|\nabla_{i_k} f(y^k)\|_2^2\geqslant\frac{1}{n}\|\nabla f(y^k)\|_2^2.\] As a result,
\[f(x^{k+1})\leqslant f(y^k)-\frac{1}{2nL}\|\nabla f(y^k)\|_2^2.\]
This implies that the condition in line 11 of Algorithm~\ref{AAM-1} is automatically satisfied if $L_{k+1}\geqslant nL$. Combined with the fact that we multiply $L_{k+1}$ by 2 if this condition is not met, this means that if $L_{k+1}\leqslant 2Ln$ at the beginning of the while loop during iteration $k$, then it is sure to hold at the end of the iteration too. This is guaranteed by our assumption that $L_0\leqslant 4Ln$.

We have just shown that  $L_{k}\leqslant 2Ln$ for $k\geqslant1$.The base case $k=0$ is trivial. Now assume that $A_k \geq \frac{k^2}{8nL}$ for some k. Note that $A_{k+1} = L_k a^2_k + a_{k+1} = A_k + a_{k+1}$ 
and $L_{k+1} = \frac{A_{k+1}}{a^2_{k+1}}$.
\begin{equation*}
    a_{k+1} = \frac{1}{2L_{k+1}} + \sqrt{\frac{1}{4L^2_{k+1}} + a^2_k\frac{L_k}{L_{k+1}}}
    \geqslant\frac{1}{4nL}+\sqrt{\frac{1}{16n^2L^2} + a^2_k\frac{L_k}{2nL}}
    \geqslant\frac{1}{4nL}\left(1+\sqrt{1+8A_knL}\right)\geqslant\frac{k+1}{4nL}.
\end{equation*}
Finally, \[A_{k+1}=A_{k}+a_{k+1}\geqslant\frac{k^2+2(k+1)}{8nL}\geqslant\frac{(k+1)^2}{8nL}.\]
By induction, we have $\forall k\geqslant 1$ 
\begin{equation}
    A_k\geqslant\frac{k^2}{8nL}\quad  \label{A_k-AAM1}
\end{equation}
and 
\[f(x^k)-f(x^*)\leqslant \frac{4nL\|x^0-x^*\|_2^2}{k^2}.\]

\end{proof}

We also note that the assumption $L_0\leqslant 4nL$ is not really crucial. In fact, if $L_0>4nL$, then after $O(\log_2\frac{L_0}{4L})$ iterations $L_k$ is surely lesser than $4L$, so overestimating $L$ only results in a logarithmic in $\frac{L_0}{L}$ amount of additional iterations needed to converge. 

\subsection{Primal-Dual Extension for Fixed Step Accelerated Alternating Minimization}

Our primal-dual algorithm based on Algorithm~\ref{AAM-1} for Problem $(P_1)$ is listed below as Algorithm \ref{PDAAM-1}.

\begin{algorithm}[H]
\caption{Primal-Dual Accelerated Alternating Minimization}
\label{PDAAM-1}

\begin{algorithmic}[1]
   \REQUIRE initial estimate of the Lipschitz constant $L_0$.
   \STATE $A_0=a_0=0$, $\eta_0=\zeta_0=\lambda_0=0$.
   \FOR{$k \geqslant 0$}
   \STATE Set $L_{k+1}=L_{k}/2$
   \WHILE{True}
   \STATE Set $a_{k+1}=\frac{1}{2L_{k+1}}+\sqrt{\frac{1}{4L^2_{k+1}}+a_k^2\frac{L_k}{L_{k+1}}}$
   \STATE Set $\tau_k=\frac{1}{a_{k+1}L_{k+1}}$
\STATE Set $\lambda^{k}=\tau_k \zeta^k+(1-\tau_k)\eta^k$
\STATE Choose $i_k=\argmax\limits_{i\in\{1,\ldots,n\}} \|\nabla_i \varphi(\lambda^k)\|_2^2$
\STATE Set $\eta^{k+1}=\argmin\limits_{\eta\in S_{i_k}(\lambda^k)} \varphi(\eta)$
\STATE Set $\zeta^{k+1} = \zeta^{k} - a_{k+1}\nabla f(\lambda^k)$
\IF{$\varphi(\eta^{k+1})\leqslant \varphi(\lambda^{k})-\frac{\|\nabla \varphi(\lambda^k)\|_2^2}{2L_{k+1}}$}
\STATE $\hat{x}^{k+1} = \frac{a_{k+1}x(\lambda^{k})+L_ka^2_k\hat{x}^{k}}{L_{k+1}a^2_{k+1}}.$
\STATE \textbf{break}
\ENDIF

\STATE Set $L_{k+1}=2L_{k+1}$.
\ENDWHILE
\ENDFOR
\ENSURE The points $\hat{x}^{k+1}$, $\eta^{k+1}$.
\end{algorithmic}

\end{algorithm}

The key result for this method is that it guarantees convergence in terms of the constraints and the duality gap for the primal problem, provided that it is strongly convex.

\begin{theorem}
\label{Th:PD_rate-fixed}
Let the objective $\vp$ in the problem $(P_2)$ be $L$-smooth and the solution of this problem be bounded, i.e. $\|\lambda^*\|_2 \leqslant R$. Then, for the sequences $\hat{x}_{k+1},\eta_{k+1}$, $k\geqslant 0$, generated by Algorithm \ref{PDAAM-1}, 
\begin{equation*}
    \hspace{-1em}\|\bm{A} \hat{x}^k - b \|_2 \leqslant \frac{16nLR}{k^2},  |\vp(\eta^k) + f(\hat{x}^k)| \leqslant\frac{16nLR^2}{k^2}.
\end{equation*}
\end{theorem}

\begin{proof}
Once again, denote $A_k=a_k^2L_{k}$ and note that $A_{k+1}=A_k+a_{k+1}$. From the proof of Lemma~\ref{AAM-1_lemma_2} we have for all $\lambda\in H$
\begin{equation*}
    a_{j+1}\la\nabla \varphi(\lambda^j),\lambda^j-\lambda\ra
    \leqslant A_j\varphi(\eta^j)-A_{j+1}\varphi(\eta^{j+1})+a_{j+1}\varphi(\lambda^j)+\frac{1}{2}\|\zeta^{j}-\lambda\|_2^2-\frac{1}{2}\|\zeta^{j+1}-\lambda\|_2^2.
\end{equation*}
We take a sum of these inequalities for $j=0,\ldots, k-1$ and rearrange the terms:
\begin{equation*}
    A_k\varphi(\eta^k)\leqslant  \sum_{j=0}^{k-1}\left\{a_{j+1} \left( \vp(\lambda^j) + \langle \nabla \vp(\lambda^j), \lambda-\lambda^j \rangle \right) \right\} + \frac{1}{2} \|\zeta^0-\lambda\|_2^2-\frac{1}{2} \|\zeta^k-\lambda\|_2^2.
\end{equation*}

If we drop the last negative term and notice that this inequality holds for all $\lambda\in H$, we arrive at
\begin{equation}
    A_k \vp(\eta^k) \leqslant \min_{\lambda \in \Lambda} \left\{\sum_{j=0}^{k-1}\{a_{j+1} ( \vp(\lambda^j)\right.\notag
    +\left. \langle \nabla \vp(\lambda^j), \lambda-\lambda^j \rangle) 
    + \frac{1}{2} \|\lambda\|_2^2 \vphantom{\sum_{j=0}^{k-1}}\right\},\label{eq:FGM_compl-fixed}
\end{equation}

From this point onwards, the proof mimics the proof of Theorem~\ref{Th:PD_rate}  word-for-word. The only difference is the different bound on $A_k$, which is $A_k \geqslant \frac{k^2}{8Ln}$ as in Theorem~\ref{AAM-1_convergence}.
\end{proof}

\section{Details for Section~\ref{S:ASA}: Application to Optimal Transport and Wasserstein Barycenter}

\subsection{Derivation of the dual entropy-regularized OT problem}

The dual problem is constructed as follows.
\begin{multline}
    \min _{X \in Q\cap\mathcal{U}(r,c)}\langle C, X\rangle+\gamma\langle X, \ln X\rangle 
    \\
    =\min _{X \in Q} \max _{y, z \in \mathbb{R}^{N}} \Big\{\langle C, X\rangle\notag
    +\gamma\langle X, \ln X\rangle+\langle y, X \mathbf{1}-r\rangle+\left\langle z, X^{T} \mathbf{1}-c\right\rangle \Big\}\notag
    \\
    = \max _{y, z \in \mathbb{R}^{N}}\Big\{-\langle y, r\rangle-\langle z, c\rangle\notag
    +\min _{X\in Q} \sum_{i, j=1}^{N} X^{i j}\left(C^{i j}+\gamma \ln X^{i j}+y^{i}+z^{j}\right) \Big\}
\end{multline}
Since the derivative of the entropy grows exponentially as $X^{ij} \to 0$, the objective under $\min_{X\in Q}$ grows as $X^{ij} \to 0$. This means that at the minimum point all the components $X^{ij} > 0$. 
Our next goal is to find $\min_{X\in Q}$. Using Lagrange multipliers for the constraint $\mathbf{1}^TX\mathbf{1}=1$, we obtain the problem
\begin{equation*}
    \min _{X^{ij}>0}\max_\nu \Bigg\{\sum_{i, j=1}^{N}\left[ X^{i j}\left(C^{i j}+\gamma \ln X^{i j}+y^{i}+z^{j}\right)\right]
    -\nu\bigg[\sum_{i, j=1}^{N}X^{i j}-1\bigg]\Bigg\},
\end{equation*}
we obtain that the solution to this problem is 
\[X^{ij}=\frac{\exp \left(-\frac{1}{\gamma}\left(y^{i}+z^{j}+C^{i j}\right)\right)}{\sum_{i,j=1}^n\exp \left(-\frac{1}{\gamma}\left(y^{i}+z^{j}+C^{i j}\right)\right)}\]

This allows us to write the dual problem as \begin{equation}
    \min_{y,z\in\mathbb{R}^N}  
    \phi(y,z)=\gamma\ln\left(
    \sum_{i,j=1}^N \exp \left(- ({y^i+z^j+C^{i j}})/{\gamma} \right)
    \right)+\la y,r \ra + \la z,c \ra.
\end{equation}
By performing a change of variables $u=-y / \gamma, v=-z / \gamma$  in~\eqref{OT_dual} we arrive at an equivalent, but possibly more well-known formulation
\begin{equation}
    \label{OT_dual_2}
    \min_{u,v\in\mathbb{R}^N}  \vp(u,v)=\gamma(\ln\left(\mathbf{1}^TB(u,v)\mathbf{1}\right)-\la u,r \ra - \la v,c \ra),
\end{equation}
\begin{equation}
    [B(u,v)]^{ij} = \exp \left(u^{i}+v^{j} -\frac{C^{i j}}{\gamma}\right).
\end{equation}
Note that to distinguish between the dual problem in terms of variables $(y,z)$ and its reformulation in terms of variables $(u,v)$ we use $\phi(y,z)$ in the first case and $\vp(u,v)$ in the second. This also means that $\phi(y,z)=\vp(-y/\gamma,-z/\gamma)$ by definition.

\subsection{Deriving Sinkhorn's algorithm as AM for the dual problem}
\begin{lemma}
\label{Sinkhorn=AM}
The iterations \[u^{k+1} \in \argmin_{u\in\mathbb{R}^N}\vp(u,v^k),\, v^{k+1}\in \argmin_{v\in\mathbb{R}^N}\vp(u^{k+1},v),\] can be written explicitly as
\[u^{k+1}=u^{k}+\ln r-\ln \left(B\left(u^k, v^k\right) \mathbf{1}\right), \]
\[v^{k+1}=v^{k}+\ln c-\ln \left(B\left(u^{k+1}, v^k\right)^T \mathbf{1}\right). \]  
\end{lemma}
\begin{proof}
From optimality conditions, for $u$ to be optimal, it is sufficient to have $\nabla_u\varphi(u,v)=0$, or
\begin{equation}
    r-(\mathbf{1}^TB(u,v^k)\mathbf{1})^{-1}B(u,v^k)\mathbf{1}=0.
\end{equation}
Now we check that it is, indeed, the case for $u=u^{k+1}$ from the statement of this lemma. We manually check that
\begin{multline*}
    B(u^{k+1},v^k)\mathbf{1}=\diag(e^{(u^{k+1}-u^k)})B(u^k,v^k)\mathbf{1}=\diag(e^{\ln r -\ln(B(u^k,v^k)\mathbf{1})})B(u^k,v^k)\mathbf{1}=
    \\
    =\diag(r)\diag(B(u^k,v^k)\mathbf{1})^{-1}B(u^k,v^k)\mathbf{1}=\diag(r)\mathbf{1}=r
\end{multline*}
and the  conclusion then follows from the fact that \[\mathbf{1}^TB(u^{k+1},v^k)\mathbf{1}=\mathbf{1}^Tr=1.\] The optimality of $v^{k+1}$ can be proven in the same way.
\end{proof}

\subsection{Complexity bound for the non-regularized optimal transport}

Next we describe how to apply our Algorithm \ref{PDAAM-2} and Theorem \ref{PD-bounds} to find the \textit{non-regularized} OT distance with accuracy $\eps$, i.e. find $\widehat{X} \in \mathcal{U}(r,c)$ s.t. $\la C,\widehat{X}\ra - \la C,X^* \ra \leq \eps$. Algorithm~\ref{Alg:OTbyAccS} is the pseudocode of our new algorithm for approximating the \emph{non-regularized} OT distance.

\begin{algorithm}[!ht]
  \caption{Accelerated Sinkhorn for OT}
  \label{Alg:OTbyAccS}
  \begin{algorithmic}[1]
    \REQUIRE{Accuracy $\eps$.}
    \STATE Set $\gamma = \frac{\eps}{3 \ln N}$, $\eps' = \frac{\eps}{8 \|C\|_\infty}$.
	\STATE 	Set $(\tilde{r},\tilde{c})=\left(1 - \frac{\eps'}{8}\right)\left((r,c) + \frac{\eps'}{8N} (\one,\one) \right)$
    \FOR{$k=1,2,...$}
	\STATE Perform an iteration of Algorithm~\ref{PDAAM-2} for the OT problem with marginals $\tilde{r},\tilde{c}$ and calculate $\widehat{X}_k$ and $\eta_k$.
	\STATE Find $\widehat{X}$ as the projection of $\widehat{X}_k$ on $\mathcal{U}(r,c)$ by Algorithm 2 of \cite{altschuler2017near-linear}.
  \STATE {\textbf{if} $\la C,\widehat{X}-\widehat{X}_k\ra \leq \frac{\eps}{6}$ and $f(\hat{x}_k)+\phi(\eta_k) \leq \frac{\eps}{6}$}
  \STATE \textbf{then} Return $\widehat{X}$.
   \ENDFOR
  \end{algorithmic}
\end{algorithm}

Taking the bounds in \eqref{eq:APDAAM_bound} instead of bounds in  \cite{dvurechensky2018computational}[Theorem 3] and repeating the proof steps in \cite{dvurechensky2018computational}[Theorem 4] together with \cite{dvurechensky2018computational}[Theorem 2], we obtain the final bound of the complexity to find an $\eps$-approximation for the non-regularized OT problem to be $O\left(\frac{N^{5/2}\sqrt{\ln N}\|C\|_\infty}{\varepsilon}\right)$. To show this, we equip the primal space $E$ with 1-norm and the dual space $H$ with 2-norm. We define $\bm{A}:\R^{N\times N}\rightarrow\R^{2N}$ as the linear operator defining the linear constraints of the problem \eqref{OT}, which is in this case defined as
$\bm{A}\vectr{X} = ((X\one)^T, (X^T\one))^T$. Then, $\|\bm{A}\|^2_{1 \rightarrow 2}=2$.
Besides the Lipschitz constant, we need to bound the norm of the solution to the dual problem \eqref{OT_dual} since that norm enters the convergence rate in Theorem \ref{PD-bounds}.
To obtain the bound we need two following lemmas.

\begin{lemma}
\label{OT_dual_norm}
Denote $\nu=\min\limits_{i,j} K^{ij}=e^\frac{-\|C\|_\infty}{\gamma}.$ Any solution $(u^*,v^*)$ of the dual problem \eqref{OT_dual_2} satisfies
\begin{equation}
    \max u_i^*-\min u_i^*\leqslant -\ln\nu\min_{i}r_i,\quad \max v_i^*-\min v_i^*\leqslant -\ln\nu\min_{i}c_i. \notag
\end{equation}
\end{lemma}
\begin{proof}
Taking the derivative of the dual objective with respect to $u$ and denoting $\Sigma=\mathbf{1}^TB(u^*,v^*)\mathbf{1}$, we obtain that
\[\nabla_u \vp(u^*,v^*)=r-\Sigma^{-1}B(u^*,v^*)\mathbf{1}.\] From the first order optimality conditions we have $\nabla_u \vp(u^*,v^*)=0$. Then we have
\[1\geqslant r_i =\Sigma^{-1}[B(u^*,v^*)\mathbf{1}]_i\geqslant\Sigma^{-1}e^{u^*_i}\nu\langle\mathbf{1},e^{v^*}\rangle.\] From this for all $i$ we get an upper bound \[ u^*_i\leqslant\ln\Sigma-\ln\nu\langle\mathbf{1},e^{v^*}\rangle.\]
On the other hand, since $C^{ij}>0$, we have $K^{ij}\leqslant 1$ and 
\[r_i =\Sigma^{-1}[B(u^*,v^*)\mathbf{1}]_i\leqslant\Sigma^{-1}e^{u_i^*}\langle\mathbf{1},e^{v^*}\rangle, \quad u_i^*\geqslant \ln\Sigma+\ln r_i -\ln\langle\mathbf{1},e^{v^*}\rangle.\] Combining the two above results, we obtain \[\max u_i^*-\min u_i^*\leqslant -\ln\nu\min_{i}r_i.\] The result for $v_i^*$ holds by the same exact argument.
\end{proof}

\begin{lemma}
\label{L:R}
There exists a solution $(y^*,z^*)$ of \eqref{OT_dual} such that
\begin{equation*}
    \|(y^*,z^*)\|_2\leqslant R:= \sqrt{N/2}\left(\|C\|_\infty-\frac{\gamma}{2}\ln{\min\limits_{i,j}\{r_i,c_j\}}\right).
\end{equation*}
\end{lemma}

\begin{proof}
We begin by deriving an upper bound on $\|(u^*,v^*)\|_2$. Using the results of the previous lemma, it remains to notice that the objective $\vp(u,v)$ is invariant under transformations $u\to u+t_u\mathbf{1}$, 
$u\to u+t_v\mathbf{1}$, with 
$t_u,t_v\in\mathbb{R}$, so there must exist some solution with
$\max_i u^*_i=-\min_i u^*_i=\|u^*\|_\infty$, $\max_i v_i=-\min_i v_i=\|v^*\|_\infty$, so
\begin{equation*}
    \|u^*\|_\infty\leqslant-\frac{1}{2}\ln\nu\min_ir_i, \;
    \|v^*\|_\infty\leqslant-\frac{1}{2}\ln\nu\min_ic_i.
\end{equation*}
As a consequence,
\begin{multline*}
    \|(u^*,v^*)\|_2\leqslant\sqrt{2N}\|(u^*,v^*)\|_\infty \leqslant-\sqrt{N/2}\ln\nu\min_{i,j}\{r_i,c_j\}
    \\
    \leqslant\sqrt{N/2}\left(\frac{\|C\|_\infty}{\gamma}-\frac{1}{2}\ln{\min\limits_{i,j}\{r_i,c_j\}}\right).
\end{multline*}
By definition, $u=-\frac{1}{\gamma}y-\frac{1}{2}\mathbf{1}$, $v=-\frac{1}{\gamma}z-\frac{1}{2}\mathbf{1}$, so we have the inverse transformation $y=-\gamma u -\frac{\gamma}{2}\mathbf{1}$, $z=-\gamma v -\frac{\gamma}{2}\mathbf{1}.$ Finally,
\begin{multline*}
    R=\|(y^*,z^*)-(y^0,z^0)\|_2=\left\Vert(-\gamma u^* -\frac{\gamma}{2}\mathbf{1},-\gamma v^* -\frac{\gamma}{2}\mathbf{1})-(-\frac{\gamma}{2}\mathbf{1}, -\frac{\gamma}{2}\mathbf{1})\right\Vert_2
    \\
    =\|-\gamma(u^*,v^*)\|_2=\gamma\|(u^*,v^*)\|_2\leqslant
    \sqrt{N/2}\left(\|C\|_\infty-\frac{\gamma}{2}\ln{\min\limits_{i,j}\{r_i,c_j\}}\right)
\end{multline*}
\end{proof}




Next, consider the non-regularized OT problem
\begin{equation}
    \min _{X \in Q\cap\mathcal{U}(r,c)}\langle C, X\rangle.
    \label{eq:OT}
\end{equation}

Let $X^*$ be the solution of the problem \eqref{eq:OT} and $X_{\gamma}^*$ be the solution of the regularized problem 
\begin{equation}
    \min _{X \in Q\cap\mathcal{U}(r,c)}\langle C, X\rangle+\gamma\langle X, \ln X\rangle .
\end{equation} 

Then, we have
\begin{equation}
\label{eq:OTbyGDPr1}
    \la C, \widehat{X} \ra =  \la C, X^* \ra  + \la C, X_{\gamma}^* - X^* \ra 
    +\la C, \widehat{X}_k - X_{\gamma}^* \ra+ \la C, \widehat{X} - \widehat{X}_k \ra . 
\end{equation} 
Now we estimate the second and third term in the r.h.s. 

\begin{multline}
    \label{eq:OTbyGDPr2}
    \la C, X_{\gamma}^* - X^* \ra =   \la C, X_{\gamma}^*\ra - \gamma H(X_{\gamma}^*)  + \gamma H(X_{\gamma}^*) - \min_{X \in \U(r,c)} \la C, X\ra
    \\
    = \min_{X \in \U(r,c)} \{\la C, X\ra - \gamma H(X)\} + \gamma H(X_{\gamma}^*) - \min_{X \in \U(r,c)} \la C, X\ra 
\end{multline}


Furthermore, since our algorithm solves problem $(P_1)$ with $f(x) = \la C,X\ra - \gamma H(X)$ and $X_{\gamma}^*$ is the solution, we have
\begin{multline}
    \label{eq:OTbyGDPr3}
    \la C, \widehat{X}_k - X_{\gamma}^* \ra 
    =  \la C, \widehat{X}_k \ra - \gamma H(\widehat{X}_k)) 
     - (\la C, X_{\gamma}^* \ra - \gamma H(X_{\gamma}^*)) + \gamma (H(\widehat{X}_k)-H(X_{\gamma}^*) 
     \\
     \stackrel{\circled{1}}{\leq} f(\hat{x}_k) + \vp(\eta_k) + \gamma (H(\widehat{X}_k)-H(X_{\gamma}^*)),
\end{multline}
where $\circled{1}$ follows from the duality gap bound $f(\hat{x}_k) - f^* \leq f(\hat{x}_k) + \vp(\eta_k)$. 

Then by \eqref{eq:OTbyGDPr3} and \eqref{eq:OTbyGDPr2} we have
\begin{multline*}
    \la C, X_{\gamma}^* - X^* \ra + \la C, \widehat{X}_k - X_{\gamma}^* \ra
    \\
    \leq  
    \min_{X \in \U(r,c)} \{\la C, X\ra - \gamma H(X)\} + \gamma H(X_{\gamma}^*) - \min_{X \in \U(r,c)} \la C, X\ra 
    + f(\hat{x}_k) + \vp(\eta_k) + \gamma (H(\widehat{X}_k)-H(X_{\gamma}^*)).
\end{multline*}
Next we use that $-H(X) \in [-2\ln n,0]$ for any $X \in \U(r,c)$, which implies
\begin{equation}
     \min_{X \in \U(r,c)} \{\la C, X\ra - \gamma H(X)\}  - \min_{X \in \U(r,c)} \la C, X\ra  \leq 0.
\end{equation}
and finally implies 
\begin{equation}
    \label{eq:OTbyGDPr23}
    \la C, X_{\gamma}^* - X^* \ra + \la C, \widehat{X}_k - X_{\gamma}^* \ra
    \leq  f(\hat{x}_k) + \vp(\eta_k) + 2\gamma\ln n.
\end{equation}

Combining \eqref{eq:OTbyGDPr1} and \eqref{eq:OTbyGDPr23}, we obtain
\begin{equation}
\label{eq:OTbyGDPr4}
\la C, \widehat{X} \ra \leq  \la C, X^* \ra  + \la C, \widehat{X} - \widehat{X}_k \ra  + f(\hat{x}_k) + \vp(\eta_k) + 2 \gamma \ln n. 
\end{equation} 
We immediately see that, when the stopping criterion in step 6 of Algorithm \ref{Alg:OTbyAccS} is fulfilled, the output $\widehat{X} \in \mathcal{U}(r,c)$ satisfies $\la C,\widehat{X}\ra - \la C,X^* \ra \leq \eps$.

It remains to obtain the complexity bound. First, we estimate the number of iterations in Algorithm \ref{Alg:OTbyAccS} to guarantee $\la C, \widehat{X} - \widehat{X}_k \ra \leq \frac{\varepsilon}{6}$ and, after that, estimate the number of iterations to guarantee $f(\hat{x}_k) + \vp(\eta_k) \leq \frac{\e}{6}$. By H\"older's inequality, we have $\la C, \widehat{X} - \widehat{X}_k \ra \leq \|C\|_{\infty} \|\widehat{X} - \widehat{X}_k\|_1$. By Lemma 7 in \cite{altschuler2017near-linear}, 
\begin{equation}
\label{eq:OTbyGDPr5}
\|\widehat{X} - \widehat{X}_k\|_1 \leq 2 \left(\| \widehat{X}_k \one - r\|_1+ \| \widehat{X}_k^T \one - c\|_1\right).
\end{equation}
Next, we obtain two estimates for the r.h.s of this inequality. First, by the definition of the operator $A$ and the vector $b$,
\begin{equation}
\label{eq:OTbyGDPr6}
    \| \widehat{X}_k \one - r\|_1 + \| \widehat{X}_k^T \one - c\|_1  \leq \sqrt{2N} \|\bm{A} { \vectr}(\widehat{X}_k) - b\|_2 
     \stackrel{}{\leq} \frac{16R\|\bm{A}\|_{E \to H}^2\sqrt{2N}}{\gamma k^2} \leq \frac{32R\sqrt{2N}}{\gamma k^2}.
\end{equation}
Where we used Theorem~\ref{PD-bounds} and the bound for $R$ defined in Lemma~\ref{L:R}.
Note that the statement of Theorem~\ref{PD-bounds} involves $n$, the number of blocks, which in this case is simply equal to 2. Here we used the choice of the norm $\|\cdot\|_1$ in $E = \R^{n^2}$ and the norm $\|\cdot\|_2$ in $H = \R^{2n}$. Indeed, in this setting $\|\bm{A}\|_{E \to H}$ is equal to the maximum Euclidean norm of a column of $A$. By definition, each column of $A$ contains only two non-zero elements, which are equal to one. Hence, $\|A\|_{E \to H} = \sqrt{2}$.

Combining \eqref{eq:OTbyGDPr5} and \eqref{eq:OTbyGDPr6}  we obtain
$$
\la C, \widehat{X} - \widehat{X}_k \ra \leq 2\|C\|_{\infty} \frac{32R\sqrt{2N}}{\gamma k^2}.
$$
Setting $\gamma = \frac{\e}{3 \ln N}$, we have that, to obtain $\la C, \widehat{X}~-~\widehat{X}_k \ra~\leq~\frac{\e}{6}$, it is sufficient to choose
\begin{equation}
\label{eq:OTbyGDPr8}
k = O\left(\frac{N^{1/4}\sqrt{R\|C\|_{\infty}\ln N}}{\e}\right).
\end{equation}
At the same time, since $\|\bm{A}\|_{E \to H} = \sqrt{2}$, by Theorem~\ref{PD-bounds}, 
$$
f(\hat{x}_k) + \vp(\eta_k) \stackrel{ }{\leq} \frac{32R^2}{\gamma k^2}.
$$
Since we set $\gamma = \frac{\e}{3 \ln N}$, we conclude that in order to obtain $f(\hat{x}_k) + \vp(\eta_k) \leq \frac{\e}{6}$, it is sufficient to choose
\begin{equation}
\label{eq:OTbyGDPr9}
k = O\left(\frac{R\sqrt{\ln N}}{\e}\right).
\end{equation}

To estimate the number of iterations required to reach the desired accuracy, we should take maximum of \eqref{eq:OTbyGDPr8} and \eqref{eq:OTbyGDPr9}. We return to the bound established in Lemma~\ref{L:R}:

\[R\leqslant \sqrt{N/2}\left(\|C\|_\infty-\frac{\gamma}{2}\ln{\min\limits_{i,j}\{r_i,c_j\}}\right).\]

In Algorithm 3 of the main part of the paper we modify the marginals $r,\ c$ to have $\min\limits_{i,j}\{r_i,c_j\}\geqslant \frac{\varepsilon}{64N\|C\|_\infty}$. As it was shown in the proof of Theorem 1 of \cite{altschuler2017near-linear}, the optimal value of this problem differs from the optimal value of the original problem by no more than $2\ln{N}\gamma+\frac{\varepsilon}{2}=\frac{7}{6}\varepsilon$. For the modified problem we hence have the bound \[R\leqslant \sqrt{N/2}\left(\|C\|_\infty-\frac{\varepsilon}{2\ln{N}}\ln{\frac{\varepsilon}{64N\|C\|_\infty}}\right)=O\left(\sqrt{N}\|C\|_\infty\right).\] 
The ratio of the bounds \eqref{eq:OTbyGDPr8} and \eqref{eq:OTbyGDPr9} is equal to $\frac{\sqrt{R}}{N^{1/4}\sqrt{\|C\|_\infty}}$, so from our estimate of $R$ we can see that these bounds are of the same order. Hence, we finally obtain the estimate on the number of iterations 

\[O\left(\frac{N^{1/2}\sqrt{\ln N}\|C\|_\infty}{\varepsilon}\right).\]

Since each iteration requires $O(N^2)$ arithmetic operations, which is the same as in the Sinkhorn's algorithm, we get the total complexity

\[O\left(\frac{N^{5/2}\sqrt{\ln N}\|C\|_\infty}{\varepsilon}\right).\]

We would also like to note that the additional factor $N^{1/2}$ compared to the complexity of the Sinkhorn's algorithm seems to be the result of the very rough estimate of $\|\bm{A} { \vectr}(\widehat{X}_k) - b\|_2$ in \eqref{eq:OTbyGDPr6}, and in our experiments our method scales approximately in the same way as the Sinkhorn's algorithm when increasing the size of the problem $N$. Figure~\ref{fig:N} should illustrate it.

\begin{figure}[!ht]
	\centering
	\includegraphics[width=\linewidth]{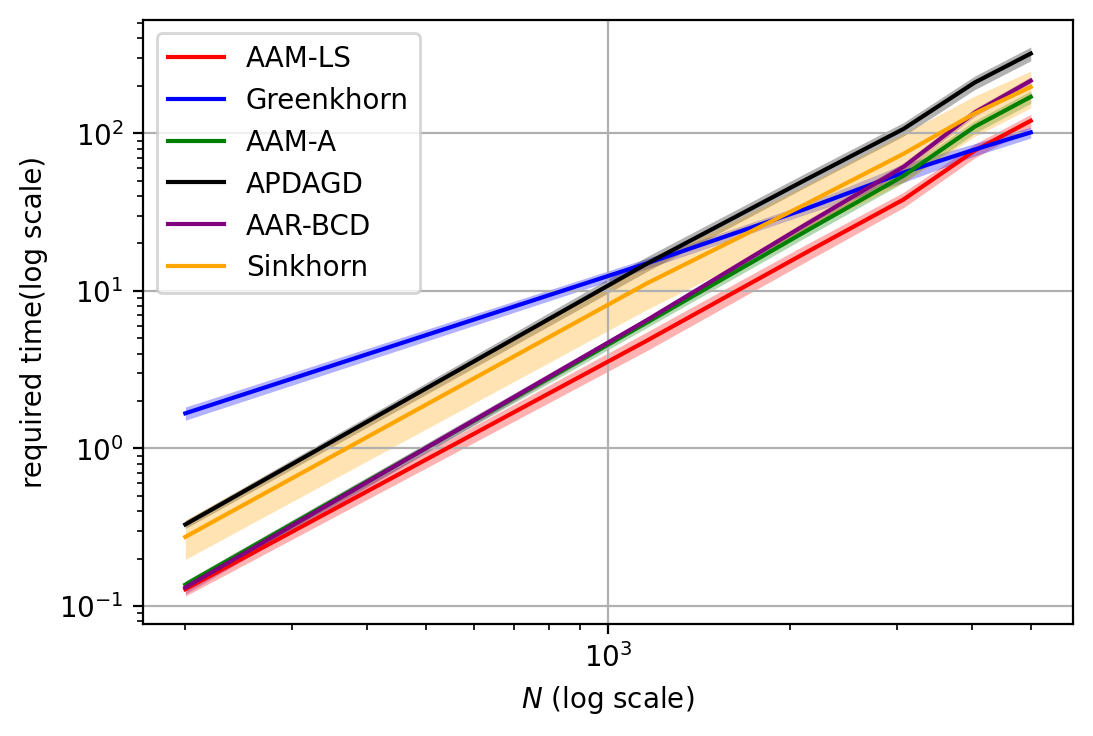}
	\caption{Experiments for OT with $\varepsilon=0.04$ and varying dimension $N$}
	\label{fig:N}
\end{figure}

We also add to comparison the rate of decay of the dual objective in Figure~\ref{fig:dual}.
\begin{figure}[!ht]
	\centering
	\includegraphics[width=\linewidth]{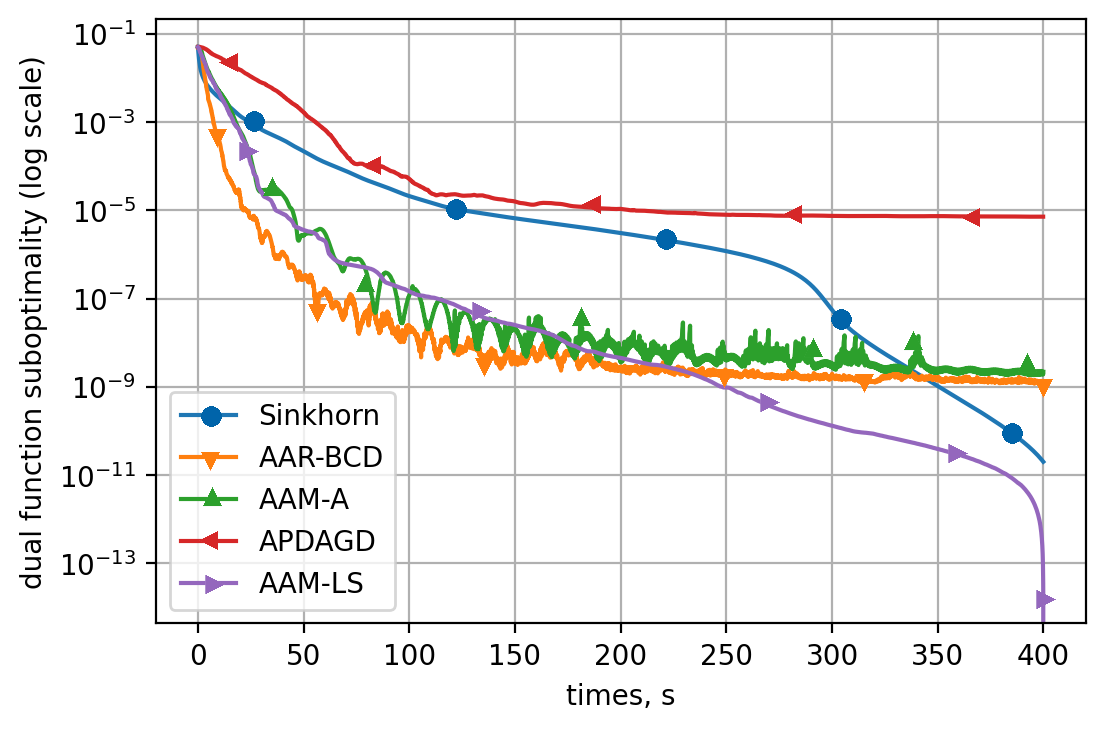}
	\caption{Decrease of the dual objective for $\varepsilon=0.004$, $N=1568$ }
	\label{fig:dual}
\end{figure}

Numerical experiments in \cite{jambulapati2019direct} were performed with an instance of Mirror-prox algorithm. Authors shared their code, and now the python implementation of the method is available at \url{https://github.com/kumarak93/numpy_ot}.
We compared the rate of decay of primal non-regularized function from a transportation plan, which is projected on the feasible set with  Algorithm~2 from \cite{altschuler2017near-linear}. The results is presented in Figure~\ref{mirr}. For AAM-LS algorithm $\e=4e-4$.
\begin{figure}[!ht]
	\centering
	\includegraphics[width=\linewidth]{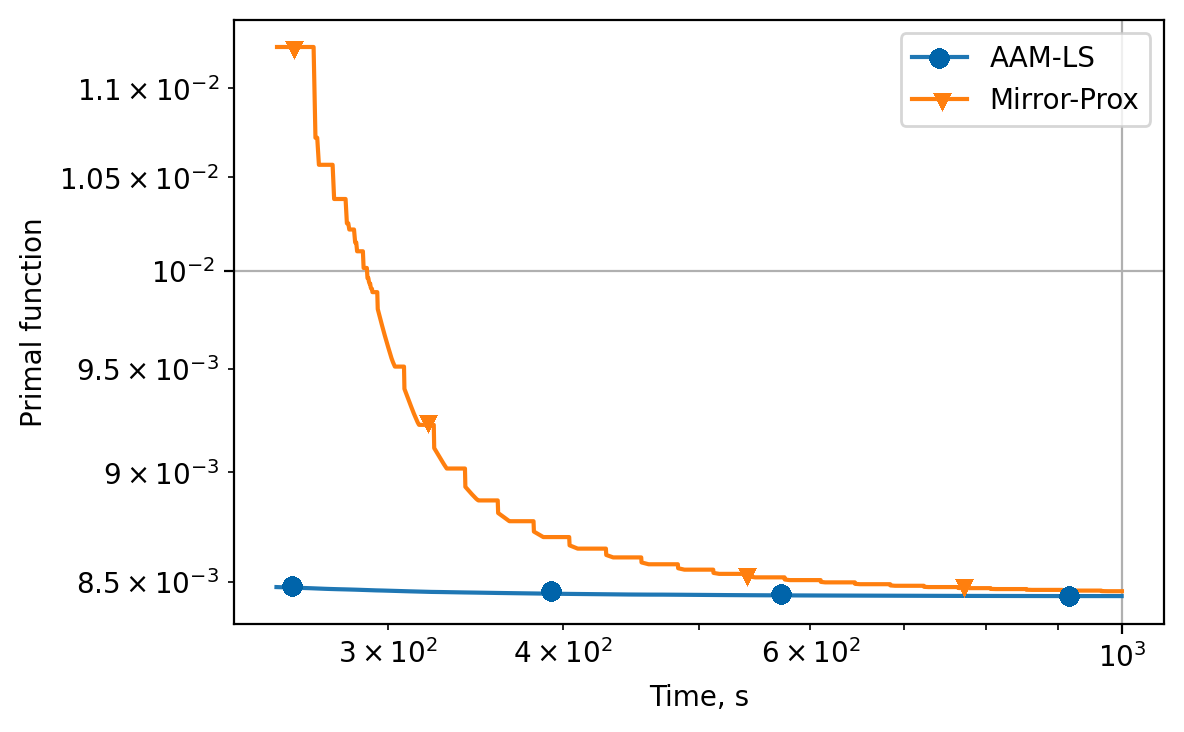}
	\caption{Decrease of the primal non-regularized objective for $\varepsilon=0.00004$, $N=1568$ }
	\label{mirr}
\end{figure}

\section{Accelerating IBP}
\label{S:AIBP}

\subsection{Derivation of the dual entropy-regularized WB problem}

The Iterative Bregman Projections algorithm for solving the regularized Wasserstein Barycenter problem is also an instance of an alternating minimizations procedure \cite{benamou2015iterative,kroshnin19a}. Hence, our accelerated alternating minimizations method may also be used for this problem. 
Denote by $\Delta^N$ the $N$-dimensional probability simplex.
Given two probability measures $p, q$ and a cost matrix $C \in \mathbb{R}_+^{N \times N}$ we define optimal
transportation distance between them as
\[W_{C}(p, q) = \min_{\pi\in\Pi(p,q)} \langle \pi, C\rangle.\]
For a given set of probability measures $p_i$ and cost matrices $C_i$
we define their weighted barycenter with weights $w \in \Delta^m$ as a solution of the following convex optimization problem:\[\min_{q\in\Delta^N} \sum_{i=1}^m w_iW_{C_i}(p_i
, q).\]
We use $c$ to denote $\max\limits_{i=1,\ldots,m} \|C_i\|_{\infty}$. We will also be using the notation $p=[p_1,\ldots,p_m]$. Using the entropic regularization we define  the regularized OT-distance for $\gamma>0$: \[
    W_{C,\gamma}(p, q) = \min_{\pi\in\Pi(p,q)} \langle \pi, C\rangle + \gamma H(\pi),
\]
where $H(\pi) :=\sum_{i, j=1}^{N} \pi_{i j}\ln \pi_{i j}=\left\langle\pi, \ln \pi\right\rangle$. 
One may also consider the regularized barycenter which is the solution to the following problem:
\begin{equation}
    \label{entropy_bar_supp}
    \min _{q \in \Delta^N} \sum_{l=1}^{m} w_{l}\mathcal{W}_{C_l,\gamma}\left(p_{l}, q\right)
\end{equation}

The following lemma is referring to Lemma 1 from  \cite{kroshnin19a}.
\begin{lemma}
The dual (minimization) problem of \eqref{entropy_bar_supp} is 
\begin{equation}
    \label{entropy_bar_dual}
    \min _{\sum_{l} w_{l}v_l=0} \vp(\mathrm{u}, \mathrm{v}),
    \end{equation}
where
\begin{equation}
    \label{WB_dual_supp}
       \min_{
       \substack{
        u,v\\
        \sum_{l=1}^m w_lv_l=0
       }
       }
       \gamma\sum\limits_{l=1}^{m} w_{l}\left\{ \ln\left(\mathbf{1}^T B_{l}\left(u_{l}, v_{l}\right)\mathbf{1}\right)-\left\langle u_{l}, p_{l}\right\rangle\right\}
\end{equation}
$u = [u_1,\ldots , u_m], v = [v_1, \ldots , v_m], u_l,v_l\in\mathbb{R}^N$, and
\[B_l(u_l, v_l) := \diag (e^{u_l}) K_l \diag (e^{v_l})\]
\[ K_{l} =\exp \left(-\frac{C_{l}}{\gamma}\right)\]
Moreover, the solution $\pi^*_\gamma$ to \eqref{entropy_bar_supp} is given by the formula
\[\left[\boldsymbol{\pi}_{\gamma}^{*}\right]_{l}=B_{l}\left(u_{l}^{*}, v_{l}^{*}\right)/\left(\mathbf{1}^TB_l(u^*_l,v^*_l)\mathbf{1}\right),\] where $(u^*, v^*)$ is a solution to the problem \eqref{entropy_bar_dual}.

\end{lemma}
\begin{proof}

Set $Q=\{X\in \mathbb{R}^{N\times N}_+: \mathbf{1}^TX\mathbf{1}=1\}$. In its expanded form, the primal problem takes the following form:
\begin{equation}
    \min_{%
       \substack{%
         \pi_l\in Q\\
         \pi_l\mathbf{1}=p_l\\
         \mathbf{1}^T\pi_1=\ldots=\mathbf{1}^T\pi_m=q\\
       }
     }
     \sum_{l=1}^mw_l\left\{\langle \pi_l,C_l\rangle+\gamma \langle \pi_l,\ln{ \pi_l}\rangle\right\}
     \label{WB}
\end{equation}


The above problem is equivalent to the problem 

\begin{equation}
     \min_{\pi_l\in Q}\max_{\lambda_l,\mu_l}
    \sum_{l=1}^m \left(w_l\left\{\langle \pi_l,C_l\rangle+\gamma \langle \pi_l,\ln{ \pi_l}\rangle\right\}
    + \langle\lambda_l,\pi_l\mathbf{1}-p_l\rangle
    \right)+\sum_{l=1}^{m-1}\langle\mu_l,\mathbf{1}^T\pi_l - \mathbf{1}^T\pi_m\rangle,
    \label{WB-dual-init}
\end{equation}

\[
    \min_{\pi_l\in Q}\max_{\lambda_l,\mu_l}
    \sum_{l=1}^mw_l\left\{\langle \pi_l,C_l\rangle+\gamma \langle \pi_l,\ln{ \pi_l}\rangle\right\}
    + \langle\lambda_l,\pi_l\mathbf{1}-p_l\rangle
    +\langle\mu_l,\mathbf{1}^T\pi_l\rangle
\]

where $\mu_m = -\sum_{l=1}^{m-1}\mu_l$.

We introduce new variables $u_l=-\frac{\lambda_l}{\gamma w_l},\ v_l=-\frac{\mu_l}{\gamma w_l}$, $l=1,...,m$.
We can now manipulate each term in the sum above exactly as we did for the optimal transportation problem. This way we arrive at the following problem.

\begin{equation}
   \min_{
   \substack{
    u,v
    \\
    v_m = -\frac{1}{w_m}\sum_{l=1}^{m-1}w_l v_l
  }
   }
   \gamma\sum\limits_{l=1}^{m} w_{l}\left\{ \ln\left(\mathbf{1}^T B_{l}\left(u_{l}, v_{l}\right)\mathbf{1}\right)-\left\langle u_{l}, p_{l}\right\rangle \right\}.
\end{equation}

The constraints $v_m = -\frac{1}{w_m}\sum_{l=1}^{m-1}w_l v_l$ is equivalent to $\sum_{l=1}^m w_lv_l=0$, that leads to final dual minimization problem:
   
\begin{equation}
       \min_{
       \substack{
        u,v\\
        \sum_{l=1}^m w_lv_l=0
       }
       }
       \gamma\sum\limits_{l=1}^{m} w_{l}\left\{ \ln\left(\mathbf{1}^T B_{l}\left(u_{l}, v_{l}\right)\mathbf{1}\right)-\left\langle u_{l}, p_{l}\right\rangle\right\}.
\end{equation}

\end{proof}

\subsection{Deriving IBP algorithm as AM for the dual problem}

The next result is well-known, but we include its proof in here for the sake of completeness: the objective can also be minimized exactly over the variables $u,v$.
\begin{lemma}
Iterations \[u^{k+1}=\argmin_{u}\vp(u,v^k),\, v^{k+1}=\argmin_{v}\vp(u^k,v),\]may be written explicitly as
\[u^{k+1}_l=u^{k}_l+\ln p_l-\ln \left(B_l\left(u_l, v_l\right) \mathbf{1}\right), \]
\[v^{k+1}_l=v^{k}_l+\sum_{j=1}^m w_j\ln(B_j(u^k_j,v^k_j)^T1)-\ln B_l(u_l,v_l)^T\mathbf{1}. \]  

\end{lemma}

\begin{proof}
Since each term in the sum in the objective only depends on one pair of vectors $(u_l,v_l)$, minimizing over $u$ equivalent to minimizing over each $u_l$. We now have to find a solution of 
\[\min_{u_l} \ln{(\mathbf{1}B_l(u_l,v^k_l)\mathbf{1}) -\langle u_l,p_l\rangle}.\]

This is the same problem as in Lemma~\ref{Sinkhorn=AM} with $p_l$ instead of $r$, so the solution has the same form.

To minimize over $v$ we will use Lagrange multipliers:
\begin{multline*}
    L(u,v,\tau) =\gamma\sum\limits_{l=1}^{m} w_{l}\left\{ \ln\left(\mathbf{1}^T B_{l}\left(u_{l}, v_{l}\right)\mathbf{1}\right)-\left\langle u_{l}, p_{l}\right\rangle\right\}
    +\langle\tau,\sum_{l=1}^m w_lv_l\rangle 
    \\
    = \gamma\sum\limits_{l=1}^{m} w_{l}\left\{ \ln\left(\mathbf{1}^T B_{l}\left(u_{l}, v_{l}\right)\mathbf{1}\right)-\left\langle u_{l}, p_{l}\right\rangle-\langle v_l,\frac{1}{\gamma}\tau\rangle\right\}.
\end{multline*}

Again, we can minimize this Lagrangian independently over each $v_l$. By the results from Lemma~\ref{Sinkhorn=AM}, we have
\[v^{k+1}_l=v^{k}_l+\ln\frac{1}{\gamma}\tau-\ln B_l(u_l,v_l)^T\mathbf{1}. \] 
This iterate needs to satisfy the constraint $\sum\limits_{l=1}^m w_lv^{k+1}_l=0$. Assuming that the previous iterate satisfies this constraint, we have an equation for $\tau$: \[\sum_{l=1}^m w_l\ln\frac{1}{\gamma}\tau=\sum_{l=1}^mw_l\ln B_l(u_l,v_l)^T\mathbf{1}.\]
Since $\sum_{l=1}^m w_l=1$, we have \[\ln\frac{1}{\gamma}\tau=\sum_{l=1}^mw_l\ln B_l(u_l,v_l)^T\mathbf{1}.\]
By plugging this into the formula for $v^{k+1}_l$ we obtain the explicit form of the alternating minimization iteration from the statement of the lemma.
\end{proof}

This result allows us to immediately apply our acceleration scheme to this problem. The resulting method is presented as Algorithm~\ref{LSIBP}. We also adopt problem-specific notation: here $\varphi(\cdot)$ denotes the dual objective \eqref{WB_dual_supp}, the first $mN$ coordinates of the dual points $\eta^k,\zeta^k,\lambda^k$ correspond to the coordinate block $u$, the other coordinates -- to the block $v$. For example, $\eta^k_1$ denotes the vector of variables $u_1$ corresponding to the point $\eta^k$, $\eta^k_{m+2}$ denotes the vector of variables $v_2$ corresponding to the point $\eta^k$. The map $x(\lambda)$ defined previously also takes the explicit form $x_l(u,v)=(\mathbf{1}^TB_l(u,v)\mathbf{1})^{-1}B_l(u,v)$ for $l=1,\ldots, m$. 

\begin{algorithm}[H]
\caption{Accelerated Iterative Bregman Projection (Line Search)}
\label{LSIBP}
\begin{algorithmic}[1]
   \STATE $A_0=\alpha_0=0$, $\eta_0=\zeta_0=\lambda_0=0$.
   \FOR{$k \geqslant 0$}
        \STATE Set $\beta_k = \argmin\limits_{\beta\in [0,1]} \vp\left(\eta^k + \beta (\zeta^k - \eta^k)\right)$
        \STATE Set $\lambda^{k}=\beta_k \zeta^k+(1-\beta_k)\eta^k$
        \STATE Choose $i_k=\argmax\limits_{i\in\{1,2\}} \|\nabla_i \vp(\lambda^k)\|^2$
        \IF {$i_k=1$}
        \FOR {$l=1,\ldots,m$}
        \STATE $\eta^{k+1}_l=\lambda^{k}_l+\ln p_l-\ln \left(B_l\left(\lambda^{k}_1, \lambda^{k}_2\right) \mathbf{1}\right)$
                \STATE $\eta^{k+1}_{m+l}=\lambda^{k}_{m+l}$
        
        \ENDFOR
        \ELSE 
        \FOR {$l=1,\ldots,m$}
        \STATE$\eta^{k+1}_l=\lambda^{k}_l$
                \STATE $\eta^{k+1}_{m+l}=\lambda^{k}_{m+l}+\sum_{j=1}^m w_j\ln(B_j(u^k_j,v^k_j)^T1)-\ln B_l(u_l,v_l)^T\mathbf{1}$
        \ENDFOR
        \ENDIF
        \STATE Find $a_{k+1}$, $A_{k+1} = A_{k} + a_{k+1}$  from  \[\vp(\lambda^k)-\frac{a_{k+1}^2}{2(A_k+a_{k+1})}\|\nabla \vp(\lambda^k)\|_2^2=\vp(\eta^{k+1})\]\\
        \STATE Set $\zeta^{k+1} = \zeta^{k} - a_{k+1}\nabla \vp(\lambda^k)$
        \STATE Set $\hat{x}^{k+1} = \frac{a_{k+1}x(\lambda^{k})+A_k\hat{x}^{k}}{A_{k+1}}.$
\ENDFOR
\ENSURE Transportation matrices $x_l^{k+1}$, dual point $\eta^{k+1}$.
\end{algorithmic}
\end{algorithm}

Note that on each iteration of this method we take a block-wise minimization step over $mN$ variables out of the whole $2mN$ variables, i.e. we are applying our accelerated Alternating Minimization scheme with the number of blocks $n=2$. Since in this case our method has the exact same primal-dual properties as the accelerated method used in \cite{kroshnin19a}, while the complexity of our method only differs by a value dependent only on $n$, which in this case is simply equal to 2, the same complexity analysis applies and our method has the same complexity $O\left(\frac{mN^{5/2}\sqrt{\ln N}\max_l\|C_l\|_\infty}{\varepsilon}\right)$ as the PDAGD method in \cite{kroshnin19a}.

\subsection{Complexity bound for the non-regularized WB problem}

Next we describe how to apply our Algorithm \ref{PDAAM-2} and Theorem \ref{PD-bounds} to find the \textit{non-regularized} WB distance with accuracy $\eps$, i.e. find $\widehat{X} \in \mathcal{U}(r,c)$ s.t. $\la C,\widehat{X}\ra - \la C,X^* \ra \leq \eps$. Algorithm~\ref{Alg:WBbyAccS} is the pseudocode of our new algorithm for approximating the \emph{non-regularized} WB distance.

\begin{algorithm}[!ht]
  \caption{Accelerated IBP}
  \label{Alg:WBbyAccS}
  \begin{algorithmic}[1]
    \REQUIRE{Accuracy $\eps$.}
    \STATE Set $\gamma = \frac{\eps}{2 \ln N}$, $\eps' = \frac{\eps}{8\max_l \|C_l\|_\infty}$.
    
	\STATE 	Set $\tilde{p_l}=\left(1 - \frac{\eps'}{4}\right)\left(p_l + \frac{\eps'}{4N} \one \right)$
    \FOR{$k=1,2,...$}
	\STATE Perform an iteration of Algorithm~\ref{PDAAM-2} for the WB problem with marginals $\tilde{p}$ and calculate $\widehat{X}_l^k$, $l=1,\cdots,m$ and $\eta^k$.
	\STATE Find $\bar q = \sum_{l = 1}^m w_l (\widehat{X}_l^k)^T \one$
	\STATE Calculate $\widehat{X}_l$ as the projection of $\widehat{X}_l^k$ on $\mathcal{U}(\tilde p, \bar q)$ by Algorithm 2 of \cite{altschuler2017near-linear}.
  \STATE {\textbf{if} $\sum_{l = 1}^m w_l \left\{\la C,\widehat{X}_l\ra - \la C,X_l^* \ra \right\} \leq  \frac{\eps}{4}$ and $f(\hat{x}_k)+\phi(\eta_k) \leq \frac{\eps}{4}$}
  \STATE \textbf{then} Return $\widehat{X}$.
   \ENDFOR
  \end{algorithmic}
\end{algorithm}

Taking the bounds in \eqref{eq:APDAAM_bound} instead of bounds in  \cite{dvurechensky2018computational}[Theorem 3] and repeating the proof steps in \cite{dvurechensky2018computational}[Theorem 4] together with \cite{dvurechensky2018computational}[Theorem 2], we obtain the final bound of the complexity to find an $\eps$-approximation for the non-regularized WB problem to be $O\left(\frac{N^{5/2}\sqrt{\ln N}\|C\|_\infty}{\varepsilon}\right)$. 
We need to bound the norm of the solution to the dual problem \eqref{WB-dual-init} since that norm enters the convergence rate in Theorem \ref{PD-bounds}.
 The bound is given by the two following lemmas.
\begin{lemma}
\label{WB_dual_solution_bounds}
Any solution $(u^*,v^*)$ of the problem \eqref{WB_dual_supp} satisfies 
\[\max [u^*_l]_i-\min [u_l^*]_i\leqslant \frac{\|C_l\|\infty}{\gamma}-\ln\min_{i}[p_l]_i,\]
\[\max [v^*_l]_i-\min [v_l^*]_i\leqslant \frac{\|C_l\|\infty}{\gamma}+\sum_{k=1}^m w_k\frac{\|C_k\|_\infty}{\gamma}.\]

\end{lemma}

\begin{proof}
The proof of the first inequality is the same as in Lemma~\ref{OT_dual_norm}, since the derivatives of the objective in the problem \eqref{WB_dual_supp} with respect to $u_l$ have the same for as in the problem \eqref{OT_dual}.

For the dual iterates $v^{k+1}$ we have the formula 
\begin{multline*}
v^{k+1}_l=v^{k}_l+\sum_{j=1}^m w_j\ln(B_j(u^k_j,v^k_j)^T1)-\ln B_l(u_l,v_l)^T\mathbf{1}=\\
=v^{k}_l+\sum_{j=1}^m w_j \ln e^{v^k_j}+\sum_{j=1}^m w_j\ln(K_j^Te^{u^k})-\ln e^{v^k_l}-\ln K_l^Te^{u_l^k}=\\
=\sum_{j=1}^m w_j\ln(K_j^Te^{u_j^k})-\ln K_l^Te^{u_l^k}.
\end{multline*} Since this was derived from the equality of the gradient to zero and holds for any $u^k$, which from now on we will denote as simply $u$, it must also hold for $v^*_l$.
Denote $\nu_j=e^{-\frac{\|C_j\|_\infty}{\gamma}}.$ We then have 
\[\ln\nu_j\langle\mathbf{1},e^{u_j}\rangle\leqslant[\ln(K_j^Te^{u_j})]_i\leqslant \ln\langle\mathbf{1},e^{u_j}\rangle.\]

Then \[\sum_{j=1}^m w_j\ln\nu_j\langle\mathbf{1},e^{u_j}\rangle-\ln\langle\mathbf{1},e^{u_l}\rangle\leqslant[v^*_l]_i\leqslant\sum_{j=1}^m w_j\ln\langle\mathbf{1},e^{u_j}\rangle-\ln\nu_l\langle\mathbf{1},e^{u_l}\rangle.\]

Finally,

\[\max [v^*_l]_i-\min [v_l^*]_i\leqslant-\sum_{j=1}^m w_j\ln\nu_j-\ln\nu_l=\frac{\|C_l\|_\infty}{\gamma}+\sum_{j=1}^m w_j\frac{\|C_j\|_\infty}{\gamma}.\]
\end{proof}

Set $(u^0,v^0)$. Once again, we know the exact value of the smoothness parameter of the dual problem in terms of variables $\lambda_i,\mu_l$, where $i\in\{1,\ldots,m\},l\in\{1,\ldots,m-1\}$. Using the above Lemma we will now derive the bound on the distance to the dual solution in these variables.

\begin{lemma}
With $(\lambda^0,\mu^0)=(0,0)$ there exists a solution of the dual problem~\eqref{WB-dual-init}in the coordinate space $(\lambda,\mu)$ such that
\[
    R^2=\|(\lambda^*,\mu^*)\|^2_2\leqslant N\left(\left(\max_{l}\|C_l\|_\infty-\frac{\gamma}{2}\min_{l,i} [p_l]_i\right)^2+\max_{l}\|C_l\|_\infty^2\right).
\]
\end{lemma}

\begin{proof}
The coordinates $(\lambda,\mu)$ and $(u,v)$ are connected by the transformation
$\lambda_l=-\gamma w_l u_l$, $l\leq m$, $\mu_i=- \gamma w_i v_i$, $i < m$.

As a function of $(u,v)$ the dual objective $\phi(u,v)$ is invariant under transformations of the form $u_l\to u_l+t_l\mathbf{1}$ with arbitrary $t_l\in\mathbb{R}$, and $v_l\to v_l+s_l\mathbf{1}$ with $s_l$ such that $\sum_{l=1}^m w_ls_l=0$. Hence, there exists a solution $(u^*,v^*)$ such that for $l\in{1,\ldots,m}$ 
\[\max [u^*_l]_i=-\min [u_l^*]_i=\|u^*_l\|_\infty,\] and for $j\in{1,\ldots,m-1}$
\[\max [v^*_j]_i=-\min [v_j^*]_i=\|v^*_j\|_\infty.\]

Using the result of the previous Lemma, we have now guaranteed the existence of a solution $(u^*,v^*)$ such that \[\|u^*_l\|_\infty\leqslant\frac{\|C_l\|_\infty}{2\gamma}-\frac{1}{2}\ln\min_{i}[p_l]_i,\]

\[\|v^*_l\|_\infty\leqslant\frac{\|C_l\|_\infty}{2\gamma}+\sum_{k=1}^m w_k\frac{\|C_k\|_\infty}{2\gamma}.\]

\begin{equation*}
    \|\lambda^*_l\|_\infty=\gamma w_l\|u^*_l\|_\infty\leqslant w_l\left(\frac{\|C_l\|_\infty}{2}-\frac{\gamma}{2}\ln\min_{i}[p_l]_i\right)\leqslant
    \leqslant w_l\left(\max_{l}\|C_l\|_\infty-\frac{\gamma}{2}\min_{l,i} [p_l]_i\right),
\end{equation*}

\begin{equation*}
    \|\mu^*_l\|_\infty=
    \gamma w_l\|v^*_l\|_\infty
    \leqslant w_l \max_{l}\|C_l\|_\infty, \; l \in \{1, \dots, m-1\}
\end{equation*}

Finally, \begin{multline*}
    \|(\lambda^*,\mu^*)\|_2^2
    =\sum_{l=1}^m \|\lambda_l\|_2^2+\sum_{j=1}^{m-1} \|\mu^*_j\|_2^2
    \leqslant N\left(\sum_{l=1}^m \|\lambda_l\|_\infty^2+\sum_{j=1}^{m-1} \|\mu^*_j\|_\infty^2\right)
    \\
    \leqslant N\left(\left(\max_{l}\|C_l\|_\infty-\frac{\gamma}{2}\min_{l,i} [p_l]_i\right)^2+\max_{l}\|C_l\|_\infty^2\right)
\end{multline*}
\end{proof}

Next, consider the non-regularized WB problem
\begin{equation}
    \min _{ \substack{X \in Q\\ \bm{A} { \vectr(X)} = b}}\sum_{l = 1}^m w_l \langle C_l, X_l\rangle,
    \label{eq:WB}
\end{equation}
where $\bm{A}  \vectr(X) = (X_1\one, \cdots, X_m\one, (X_1^T\one - X_{m}^T\one), (X_2^T\one - X_{m}^T\one), \cdots,  (X_{m-1}^T\one - X_{m}^T\one))^T$ and $b = (p_1, \cdots, p_m, 0, \cdots, 0)^T$

Let $X^*$ be the solution of the problem \eqref{eq:WB} and $X_{\gamma}^*$ be the solution of the regularized problem 
\begin{equation}
    \min _{ \substack{X \in Q\\ \bm{A} { \vectr(X)} = b}}\sum_{l = 1}^m w_l \langle C_l, X_l\rangle+\gamma\langle X_l, \ln X_l\rangle .
\end{equation} Then, we have

\begin{equation}
    \label{eq:WBbyGDPr1}
    \sum_{l = 1}^m w_l \la C_l, \widehat{X}_l \ra 
    =  \sum_{l = 1}^m w_l \left\{\la C_l, X_l^* \ra  + \la C_l, {X_l}_{\gamma}^* - X_l^* \ra  +\la C_l, \widehat{X}_l^k - {X_l}_{\gamma}^* \ra+ \la C_l, \widehat{X}_l - \widehat{X}_l^k \ra \right\}. 
\end{equation}
Now we estimate the second and third term in the r.h.s. 

\begin{multline}
    \label{eq:WBbyGDPr2}
    \sum_{l = 1}^m w_l \la C_l, {X_l}_{\gamma}^* - X_l^* \ra 
    =   \sum_{l = 1}^m w_l \left\{ \la C_l, {X_l}_{\gamma}^*\ra - \gamma H({X_l}_{\gamma}^*)  + \gamma H({X_l}_{\gamma}^*) \right\}
    - \min _{ \substack{X \in Q\\ \bm{A} { \vectr(X)} = b}}\sum_{l = 1}^m w_l \la C_l, X_l\ra
    \\
    = \min _{ \substack{X \in Q\\ \bm{A} { \vectr(X)} = b}}\sum_{l = 1}^m w_l \left\{\la C_l, X_l\ra 
    - \gamma H(X_l)\right\} 
    - \min _{ \substack{X \in Q\\ \bm{A} { \vectr(X)} = b}}\sum_{l = 1}^m w_l \la C_l, X_l\ra 
    + \gamma \sum_{l = 1}^m w_l H({X_l}_{\gamma}^*) 
\end{multline}


Furthermore, since our algorithm solves problem $(P_1)$ with $f(x) = \sum_{l = 1}^m w_l \left\{ \la C_l,X_l\ra - \gamma H(X_l)\right\}$ and ${X_l}_{\gamma}^*$ is the solution, we have
\begin{multline}
    \label{eq:WBbyGDPr3}
    \sum_{l = 1}^m w_l \la C_l, \widehat{X}_l^k - {X_l}_{\gamma}^* \ra 
    \\
    =  \sum_{l = 1}^m w_l \left\{\la C_l, \widehat{X}_l^k \ra - \gamma H(\widehat{X}_l^k)\right\}    
    - \sum_{l = 1}^m w_l \left\{\la C_l, {X_l}_{\gamma}^* \ra - \gamma H({X_l}_{\gamma}^*)\right\}
    + \gamma \sum_{l = 1}^m w_l \left\{ H(\widehat{X}_l^k)-H({X_l}_{\gamma}^*)\right\}
    \\
     \stackrel{\circled{1}}{\leq} f(\hat{x}_k) + \vp(\eta_k) + \gamma \sum_{l = 1}^m w_l \left\{ H(\widehat{X}_l^k)-H({X_l}_{\gamma}^*)\right\},
\end{multline}
where $\circled{1}$ follows from the duality gap bound $f(\hat{x}_k) - f^* \leq f(\hat{x}_k) + \vp(\eta_k)$. 

Then by \eqref{eq:WBbyGDPr3} and \eqref{eq:WBbyGDPr2} we have
\begin{multline*}
    \sum_{l = 1}^m w_l \left\{ \la C_l, {X_l}_{\gamma}^* - X_l^* \ra + \la C_l, \widehat{X}_l^k - {X_l}_{\gamma}^* \ra\right\}
    \\
    \leq  
    \min _{ \substack{X \in Q\\ \bm{A} { \vectr(X)} = b}}\sum_{l = 1}^m w_l \left \{\la C_l, X_l\ra - \gamma H(X_l)\right\} 
    + \gamma \sum_{l = 1}^m w_l H({X_l}_{\gamma}^*) 
    - \min _{ \substack{X \in Q\\ \bm{A} { \vectr(X)} = b}}\sum_{l = 1}^m w_l  \la C_l, X_l\ra 
    \\
    + f(\hat{x}_k) + \vp(\eta_k) + \gamma \sum_{l = 1}^m w_l  \left \{ H(\widehat{X}_l^k)-H({X_l}_{\gamma}^*)\right\}.
\end{multline*}

Next we use that $-H(X_l) \in [-2\ln n,0]$ for any $X_l \in Q$, which implies
\begin{equation}
     \min _{ \substack{X \in Q\\ \bm{A} { \vectr(X)} = b}}\sum_{l = 1}^m w_l \left \{\la C_l, X_l\ra - \gamma H(X_l) \right\}  - \min _{ \substack{X \in Q\\ \bm{A} { \vectr(X)} = b}}\sum_{l = 1}^m w_l \la C_l, X_l\ra  \leq 0.
\end{equation}
and finally implies 
\begin{equation}
    \label{eq:WBbyGDPr23}
    \sum_{l = 1}^m w_l \left\{\la C_l, {X_l}_{\gamma}^* - X_l^* \ra + \la C_l, \widehat{X}_l^k - {X_l}_{\gamma}^* \ra\right\}
    \leq  f(\hat{x}_k) + \vp(\eta_k) + 2\gamma\ln n.
\end{equation}

Combining \eqref{eq:WBbyGDPr1} and \eqref{eq:WBbyGDPr23}, we obtain
\begin{equation}
    \label{eq:WBbyGDPr4}
    \sum_{l = 1}^m w_l \la C_l, \widehat{X}_l \ra \leq  \sum_{l = 1}^m w_l \la C_l, X_l^* \ra  + \sum_{l = 1}^m w_l \la C_l, \widehat{X}_l - \widehat{X}_l^k \ra  + f(\hat{x}_k) + \vp(\eta_k) + 2 \gamma \ln n. 
\end{equation}
We immediately see that, when the stopping criterion in step 6 of Algorithm \ref{Alg:WBbyAccS} is fulfilled, the output $\widehat{X}_l \in \{X \in Q | \bm{A} { \vectr(X)} = b\}$ satisfies $\sum_{l = 1}^m w_l \la C,\widehat{X}_l\ra - \sum_{l = 1}^m w_l \la C,X_l^* \ra \leq \eps$.

\color{black}

It remains to obtain the complexity bound. First, we estimate the number of iterations in Algorithm \ref{Alg:WBbyAccS} to guarantee $\sum_{l = 1}^m w_l \la C_l, \widehat{X}_l - \widehat{X}_l^k \ra \leq \frac{\varepsilon}{4}$ and, after that, estimate the number of iterations to guarantee $f(\hat{x}_k) + \vp(\eta_k) \leq \frac{\e}{4}$. 

Denote $q_l=(\widehat{X}_l^k)^T \one.$
From the scheme of~\cite{kroshnin19a} and since $\|\bm{A}  \vectr(X) -b \|_1 = \sum_{l = 1}^{m} \| q_l - q_{l+1}\|_1 $ after an update of $u$ variables we have
\begin{multline}
    \label{eq:WBbyGDPr5}
    \sum_{l = 1}^m w_l \la C_l, \widehat{X}_l - \widehat{X}_l^k \ra 
    \leq \max_{l}\|C_l\|_\infty \sum_{l = 1}^m w_l \|\widehat{X}_l - \widehat{X}_l^k\|_1
    \\
    \leq 2 \max_{l}\|C_l\|_\infty \sum_{l = 1}^m w_l \left(\|(\tilde{p}_l - p_l \|_1 + \| (\widehat{X}_l^k)^T \one - \bar q \|_1 \right) 
    \\
    \leq 2 \max_{l}\|C_l\|_\infty \e'
    + 2 \max_{l}\|C_l\|_\infty \max_l w_l  \|\bm{A}  \vectr(X) -b \|_1.
\end{multline}

It remains to show that $2\max_{l}\|C_l\|_\infty \max_l w_l  \|\bm{A}  \vectr(X) -b \|_1 \leq \eps /4$.

By Theorem~\ref{PD-bounds} 
\[
    \|\bm{A}  \vectr(X) -b \|_1 \leq \frac{16R\|\bm{A}\|_{E \to H}^2\sqrt{2N}}{\gamma k^2}.
\]

Setting 
\begin{equation}
    \frac{16RL\sqrt{2N}}{k^2}
    \\
    =
    \frac{16R\|\bm{A}\|_{E \to H}^2\sqrt{2N}}{\gamma k^2}
    \leqslant \frac{\eps}{8\max_{l}\|C_l\|_\infty \max_l w_l },
    \label{want}
\end{equation}

together with the choice of $\gamma = \frac{\eps}{2 \ln N}$ and since $\|\bm{A}\|_{E \to H} = \sqrt{2}$
, we have that, to obtain $\la C, \widehat{X}_l~-~\widehat{X}_l^k \ra~\leq~\frac{\e}{4}$, it is sufficient to choose

\begin{equation}
    \label{eq:WBbyGDPr8}
    k = O\left(\frac{N^{1/4}\sqrt{\|C_l\|_\infty \max_l w_l R\|C\|_{\infty}\ln N}}{\e}\right).
\end{equation}

At the same time, by Theorem~\ref{PD-bounds}, 
$$
f(\hat{x}_k) + \vp(\eta_k) \stackrel{ }{\leq} \frac{32R^2}{\gamma k^2}.
$$
Since we set $\gamma = \frac{\e}{2 \ln N}$, we conclude that in order to obtain $f(\hat{x}_k) + \vp(\eta_k) \leq \frac{\e}{4}$, it is sufficient to choose
\begin{equation}
\label{eq:WBbyGDPr9}
k = O\left(\frac{R\sqrt{\ln N}}{\e}\right).
\end{equation}

To estimate the number of iterations required to reach the desired accuracy, we should take maximum of \eqref{eq:WBbyGDPr8} and \eqref{eq:WBbyGDPr9}. We return to the bound established in Lemma~\ref{L:R}:

\[
    R^2=\|(\lambda^*,\mu^*)\|^2_2\leqslant N\left(\left(\max_{l}\|C_l\|_\infty-\frac{\gamma}{2}\min_{l,i} [\tilde{p}_l]_i\right)^2+\max_{l}\|C_l\|_\infty^2\right)
\]

or one can write
\[
    R=O\left(\sqrt{N}\|C\|_\infty\right).
\] 
The ratio of the bounds \eqref{eq:WBbyGDPr8} and \eqref{eq:WBbyGDPr9} is equal to $\frac{\sqrt{R}}{N^{1/4}\sqrt{\max_l w_l \|C\|_\infty}}$, so from our estimate of $R$ we can see that these bounds are of the same order. Hence, we finally obtain the estimate on the number of iterations 

\[O\left(\frac{N^{1/2}\sqrt{\ln N}\|C\|_\infty}{\varepsilon}\right).\]

Since each iteration requires $O(mN^2)$ arithmetic operations, which is the same as in the IBP algorithm, we get the total complexity

\[O\left(\frac{mN^{5/2}\sqrt{\ln N}\|C\|_\infty}{\varepsilon}\right).\]

\section{Implementation Details}

Looking through the proof of convergence for Algorithm~\ref{AAM-2} one can notice that line search subroutine need to fulfill two conditions: $\langle \nabla f(y^k), v^k - y^k \rangle \geqslant 0$ and $f(y^k)\leqslant f(x^k)$. We got significant increase of performance, when were using these condition as a stopping criteria for line search subroutine. Another increase of performance came from the observation that the value of $\beta$ satisfying the condition is often close to $\frac{k-1}{k+2}$, the value appearing in Nesterov's type accelerated methods \cite{JMLR:v17:15-084}. The other observation is that the value of $\beta$ satisfying the conditions frequently does not change from iteration to iteration with the same parity. So we use the value $\beta_{t-2}$ as a starting point for the line search subroutine to find $\beta_t$ on $t$-th iteration. These and other implementation details are available on
\url{https://github.com/nazya/AAM}

\end{document}